\documentclass{article}
\usepackage{graphicx}
\usepackage{latexsym,amsmath,amsfonts,amscd, amsthm}
\usepackage{epsfig}
\usepackage{changebar}
\usepackage{color}
\usepackage{bm}
\usepackage{tikz}
\usetikzlibrary{arrows,backgrounds,snakes}

\newtheoremstyle{plainNoItalics}{}{}{\normalfont}{}{\bfseries}{.}{ }{}

\theoremstyle{plain}
\newtheorem{thm}{Theorem}[section]

\theoremstyle{plainNoItalics}

\newtheorem{lem}[thm]{Lemma}
\newtheorem{defn}[thm]{Definition}
\newtheorem{rem}[thm]{Remark}
\newtheorem{prop}[thm]{Proposition}

\renewcommand{\theequation}{\thesection.\arabic{equation}}
\renewcommand{\thefigure}{\thesection.\arabic{figure}}
\renewcommand{\thetable}{\thesection.\arabic{table}}

\newcommand{\beq}{\begin{equation}}
\newcommand{\eeq}{\end{equation}}
\newcommand{\beqa}{\begin{eqnarray}}
\newcommand{\eeqa}{\end{eqnarray}}
\newcommand{\bit}{\begin{itemize}}
\newcommand{\eit}{\end{itemize}}
\newcommand{\bedef}{\begin{defn}}
\newcommand{\edefn}{\end{defn}}
\newcommand{\bpro}{\begin{prop}}
\newcommand{\epro}{\end{prop}}

\newcommand{\eps}{\varepsilon}

\def\Box{\mbox{ }\rule[0pt]{1.5ex}{1.5ex}}

\begin{document}

%\baselineskip=1.8pc

%\vspace*{.10in}

%=============  title  =========================
\begin{center}
{\bf
Error Estimates of the Integral Deferred Correction Method for Stiff Problems
}
\end{center}
\vspace{.2in}
\centerline{  
%:
Sebastiano Boscarino \footnote{Department of Mathematics and Computer Science, University of Catania, Catania, 95125, E-mail: boscarino@dmi.unict.it}, 
Jing-Mei
Qiu\footnote{Department of Mathematics, University of Houston,
Houston, 77004. E-mail: jingqiu@math.uh.edu. Research supported by
Air Force Office of Scientific Computing YIP grant FA9550-12-0318, NSF grant DMS-0914852 and DMS-1217008 and University of Houston.} 
}

\bigskip
\noindent
{\bf Abstract.}
In this paper, we present error estimates of the integral deferred correction method constructed with stiffly accurate implicit Runge-Kutta methods {with a nonsingular matrix $A$ in its Butcher table representation,} when applied to stiff problems characterized by a small positive parameter $\varepsilon$.
In our error estimates, we expand the global error in powers of $\varepsilon$ and show that the coefficients are global errors of the integral deferred correction method applied to a sequence of differential algebraic systems. 
A study of these errors and of the remainder of the expansion yields sharp error bounds for the stiff problem. 
Numerical results for the van der Pol equation are presented {to} illustrate our theoretical findings. 
 Finally, we study the linear stability properties of these methods.


%\newpage

\section{Introduction}
\label{sec1}
\setcounter{equation}{0}
\setcounter{figure}{0}
\setcounter{table}{0}

{The deferred correction (DC) method for solving an initial value problem in the form of
\beq\label{ODE}
y'(t) = f(t,y(t)), \ \ y(t_0) = y_0 \in \mathbb{R}^N,
\eeq
has been investigated intensively \cite{CDC, SRD, AHKW}. 
An advantage of the DC method is that one can use a simple numerical method, for instance a first order method, to compute the solution with higher order accuracy. This is accomplished by using a lower order numerical method to solve a series of correction equations during each time step. In each iteration, the order of the method increases.
In \cite{dutt2000spectral}, a new variant of the deferred correction method called the spectral deferred correction (SDC) was proposed. In SDC, a deferred correction procedure is applied to an integral formulation of the error equation in the DC method. It has been shown that the SDC method outperforms DC in many problems with promising numerical results \cite{dutt2000spectral}.} {This is mainly due to the integral formulation of the error equation, as numerical integration is considered to be a more stable and accurate process than numerical  differentiation. Moreover, the selection of quadrature nodes plays some role in the performance of the SDC method \cite{layton2005implications}}. In \cite{dutt2000spectral}, the quadrature nodes in the proposed SDC method are chosen to be Gauss-Lobatto, Gauss-Radau or Gauss-Legendre points for high order of accuracy. When the quadrature nodes are uniform, the SDC method is called the integral deferred correction (InDC) method. There are various SDC/InDC methods with different implementation strategies, e.g. in selecting time integrators in prediction and correction steps \cite{minion2003semi, layton2007implications, layton2008choice, huang2007arbitrary, christlieb2009integral, christlieb2009comments, idcark} and in coupling with the Krylov subspace method \cite{huang2007arbitrary}. 
Within the InDC framework, it is shown in \cite{christlieb2009integral, christlieb2009comments} that if an $r$\textsuperscript{th} order integrator is used to solve the error equation, then the accuracy of the scheme increases by $r$ orders after each correction loop. This analysis has recently been extended in \cite{idcark} for the InDC method constructed with implicit and semi-implicit integrators. 
In \cite{christlieb2009comments}, the InDC method constructed with high order Runge-Kutta (RK) methods has been reformulated as a RK method, whose Butcher tableau has been explicitly constructed.

The main goal of this paper is to study the convergence behavior of the InDC method constructed using implicit RK methods of different orders, when applied to a special class of stiff problems 
% containing a parameter $\varepsilon$ 
%\SB{(I prefer to live this part as in the Hairer and Wanner Book, this parameter characterises the problem)} 
called \emph{singular perturbation problems} (SPPs). A typical SPP has the form
\beq
\label{spp}
\begin{array}{l}
y'(t) = f(y(t),z(t)),\\ 
\varepsilon z'(t) = g(y(t),z(t)),
\end{array} 
\eeq
where $y$ and $z$ are vectors in {$\mathbb{R}^{N}$} with {$N$} being the dimension of the vectors and  $\varepsilon > 0$ is the \emph{stiffness} parameter. We call these vectors the differential component for $y$ and {the} algebraic one for $z$. Classical books on this subject are \cite{Tikhonov, Malley}. 
In system (\ref{spp}) we assume that
$0 < \varepsilon \ll 1$ and $f$ and $g$ are sufficiently differentiable vector-valued functions. 
%of the same dimensions as $y$ and $z$ respectively. 
The functions  $f$, $g$ and the 
initial values $y(0)$, $z(0)$ may depend smoothly on 
$\varepsilon$. For simplicity of notation, we suppress such dependence.
We require that system (\ref{spp}) satisfies 
\beq
\label{eq: gz}
\mu(g_z(y, z)) \le -1,
\eeq
in an $\varepsilon$-independent neighbourhood of the solution, where $\mu$ denotes the logarithmic norm with respect to some inner product. From a classical result in SPPs theory, the condition (\ref{eq: gz}) guarantees the existence of an $\varepsilon$-expansion, whose coefficients are the sum of a smooth function of the independent variable $t$ and an exponentially decaying function of the stretched variable $\tau = t/\varepsilon$ (initial layer). The exponentially decaying function is not present if the initial values of system (\ref{spp}) (which depend on $\eps$) are on the smooth solution,  see Chap.~VI.3 of \cite{hairer1993solving2} for more details. We thus suppose, in our analysis, that the initial values lie on the smooth solution, that $\varepsilon\ll H$ where $H$ is the time step size, and that the initial layer is over.  In fact, arbitrary initial values introduce an initial layer in the solution. One possible way to overcome this difficulty is simply to ensure that the numerical method resolves the initial layer by taking small step size of $\mathcal{O}(\eps)$.  

System (\ref{spp}) allows us to understand many phenomena observed for very stiff problems. 
Indeed, in \cite{hairer1993solving2} and in the original paper \cite{hairer1988error}, the authors showed that most of the RK methods presented in the literature suffer from the phenomenon of order reduction in the stiff regime.
To this aim, we investigate the same phenomenon when it appears in the InDC framework. In the past, such order reduction has been numerically investigated without much theoretical justification \cite{minion2003semi, idcark}.
The novelty of this paper is to provide rigorous and careful convergence analysis for the global error of the InDC method and investigate its stability property. 

In this paper, we study the global error of the InDC method when it is applied to SPPs in the form of (\ref{spp}), in order to seek an understanding on the order reduction phenomenon. First we consider the InDC method constructed with the backward Euler (BE) method, denoted as InDC-BE, and then with implicit RK (IRK) methods, denoted as InDC-IRK.

The main idea is to expand the error in powers of $\varepsilon$, whose coefficients are called error terms, and show convergence results for these error terms. Order reduction phenomenon exists for both differential and algebraic components in the InDC framework. Specifically, under suitable assumptions, the order of convergence for the first term in the $\varepsilon$-expansion of global error increases with high order if a high order RK method is applied in the correction steps of the InDC method;  whereas the order of convergence for the second term in $\varepsilon$-expansion is determined by the stage order of the RK method for the prediction step. We focus our analysis on the InDC method using uniform quadrature nodes, but excluding the left-most endpoint. The uniform distribution of nodes is important to increase accuracy by the corresponding high order, when a high order RK method is applied in correction steps for classical problems; we refer readers to \cite{christlieb2009integral} for details.
 The use of quadrature nodes excluding the left-most endpoint leads to an important stability condition for stiff problems, i.e. the method becomes L-stable if A-stable; we {discuss such stability issues in Section~\ref{sec: stab}}.
 %refer readers to \cite{layton2005implications} for details. 
{ We also remark that important assumptions on the IRK method are that the method is {stiffly accurate} and has nonsingular matrix $A$ {in its Butcher table representation}.} We will show that, if these properties are not satisfied, the corresponding InDC method becomes unstable and the numerical solution diverges. {A satisfactory explanation of this fact is given in the Appendix}.  

% We also remark that the property of stiffly accurate is important for implicit RK method in the prediction and correction steps for numerical stability.

The paper is organized in the following way. In the rest of this section,  
we present the basic notations of IRK methods for SPPs in \cite{hairer1993solving2} (for more details see \cite{hairer1988error}).
In Section~\ref{sec2}, we introduce the InDC-BE method for SPPs \eqref{spp}.
In Section~\ref{sec3}, main theoretical results are stated in the form of two Theorems; numerical evidence supporting these theoretical  results are summarized and presented. In Section~\ref{sec4}, we prove convergence results for the InDC-BE method. In Section~\ref{sec: stab}, we study the linear stability properties of these InDC methods. Conclusions are given in Section~\ref{sec5}.  {We organize the description of InDC-IRK methods, the $\eps$-expansion of the numerical solution, as well as the corresponding error estimates and the estimation of the remainder, into the Appendix for better readability of the paper}. Throughout the paper, for classical concepts and convergence results related to RK methods applied to SPPs, we will cite the classical book on the subject \cite{hairer1993solving2} (with the Chapter numbering) from time to time.
%%%%%%%%%%%%%%%%%%%%%%%%%%%%%%%%%%%
%%%%%%%%%%%%%%%%%%%%%%%%%%%%%%%%%%%
\subsection{
The IRK method applied to SPPs
}
\label{sec2.0}
%\setcounter{equation}{0}
%\setcounter{figure}{0}
%\setcounter{table}{0}

%The IDC framework requires RK methods in the prediction and correction steps.
In order to get more insight in the convergence estimates of InDC methods, 
it is useful to consider the convergence results for the RK methods when applied to (\ref{spp}). 
We observe that when the parameter $\varepsilon$ in system (\ref{spp}) is small, the corresponding differential equation
is stiff, and when $\varepsilon$ tends to zero, the differential equations become a differential
algebraic system. %A sequence of differential-algebraic systems arises in the study of SPPs.
The corresponding \emph{reduced} system, i.e. $\varepsilon = 0$, is the differential algebraic equation (DAE) 
\beq\label{reduced}
\begin{array}{l}
y' = f(y,z), \\
0 = g(y,z), 
\end{array}
\eeq
whose initial values are \emph{consistent} if $0 = g(y_0,z_0)$. We assume that the Jacobian 
\beq
\label{redg}
g_z(y, z) \qquad \textrm{is invertible},
\eeq 
in a neighbourhood of the solution of (\ref{reduced}). This
assumption guarantees the solvability of (\ref{reduced}) and that the equation $g(y, z) = 0$ possesses
a locally unique solution $z = \mathcal{G}(y)$ (Implicit Function Theorem), which inserted into (\ref{reduced}) gives
\beq
\label{eqy} 
y' = f(y,\mathcal{G}(y)).
\eeq
From now on we assume a Lipschitz condition for $\mathcal{G}$. Furthermore, under the assumption \eqref{redg}, equation (\ref{reduced}) is said to be a differential-algebraic equation of index 1. For a definition of the index of differential algebraic problems, we refer to \cite{gear1988differential, hairer1993solving2}. 

%\SB{DELETE:Below we give a review of the main convergence results for {IRK} methods applied to SPPs (\ref{spp}) 
% \cite{hairer1988error, hairer1993solving2}.
 %to investigate thoroughly the convergence estimates an IDC RK method. 
%Our discussions on these methods are based on notations introduced in Chap.~VI.3 of  \cite{hairer1993solving2}. 
%Before we state the main results for IRK methods,} 

%Now we recall the definition of a RK method, the concept of classical order $p$ and stage order $q$.
{Now in order to solve system (\ref{spp}) we apply %recall the definition of a RK method,  We consider 
an IRK method. This gives
\beq\label{eq:1-8}
\left(\begin{array}{c}
y_{n+1}\\
 z_{n+1}
\end{array}\right)
 = \left(\begin{array}{c}
y_{n}\\
 z_{n}
\end{array}\right)
 + h \sum_{i=1}^{s} b_i\left(\begin{array}{c}
                           k_{ni}\\
                            \ell_{ni}
                          \end{array}\right) ,
\eeq
where
\beq\label{eq:1-9}
\left(\begin{array}{c}
k_{ni}\\
\eps \ell_{ni}
\end{array}\right)
 = \left(\begin{array}{c}
f(Y_{ni},Z_{ni})\\
g(Y_{ni},Z_{ni})
\end{array}\right), 
\eeq 
and the internal stages are given by
\beq\label{eq:1-10}
\left(\begin{array}{c}
Y_{ni}\\
 Z_{ni}
\end{array}\right)
 = \left(\begin{array}{c}
y_{n}\\
 z_{n}
\end{array}\right)
 + h  \sum_{j=1}^{s} a_{ij}\left(\begin{array}{c}
                          k_{nj}\\
                           \ell_{nj}
                          \end{array}\right).
\eeq
Such method is characterized by the coefficient matrix $A = ({a}_{ij})$ and vectors $c=(c_1,...,c_s)^T$, $b=(b_1,...,b_s)^T$. 
They can be represented by a $tableau$ in the usual Butcher notation,
\beq
\label{eq: B_table}
\begin{array}{c|c}
{c} & {A}\\
\hline
 & {b^T} \end{array}.
\eeq
The coefficients $c$ are given by the basic consistency relation
%\beq\label{eq:candc}
$ c_i = \sum_{j=1}^{s} a_{ij}$.

We now suppose that the matrix $A$ is invertible and put $\eps = 0$, we obtain by algebraic manipulations from \eqref{eq:1-8}, \eqref{eq:1-9}, \eqref{eq:1-10} that,
\beq\label{eq:1-10bis}
\begin{array}{l}
\displaystyle y_{n+1}= 
y_{n} + h \sum_{i=1}^{s} b_i f(Y_{ni},Z_{ni})\\
\displaystyle  z_{n+1} =
R(\infty)z_{n} + h \sum_{i=1}^{s} b_i w_{ij} Z_{nj},
\end{array}
\eeq
where
\beq\label{eq:1-9bis}
\begin{array}{l}
\displaystyle Y_{ni} = y_n + h  \sum_{j=1}^{s} a_{ij} f(Y_{ni},Z_{ni})\\
\displaystyle 0 = g(Y_{ni},Z_{ni}),
\end{array}
\eeq
with $R(\infty) = 1- \sum_{i,j=1}^s b_i w_{ij}$, where $R(z)$ is the stability function of the method and $w_{ij}$ the elements of the inverse of the matrix $A$.
We note that the numerical solution $z_{n+1}$ is independent of $\eps$ and this represents an interesting approach to solve the reduced system (\ref{reduced}). 
In general the numerical solutions (\ref{eq:1-10bis}) do not lie on the manifold $g(y,z) = 0$. Of special importance here is the following definition which will be an important assumption in the next for the analysis.

\begin{defn}
\label{DefSA} An IRK method is called \emph{stiffly accurate} (SA) if $b^T = e^T_s A$ with $e^T_s = (0,...,0,1)$,
i.e., methods for which the numerical solution is identical to the last internal stage.
\end{defn}

Now we have a couple of remarks in order here.

%\SB{
%\begin{defn}\label{Lstab}
%We say that a method is called $L$-stable if it is $A$-stable and if its stability function $R(z) \to 0$ when $z \to \infty$.
%\end{defn}
%}

\begin{rem}
{
By the non-singularity of the matrix $A$ and with Definition \ref{DefSA}, we have $R(\infty) = 0$ for a SA IRK method.
% \SB{and then by remark \ref{Lstab}}, 
This makes an $A$-stable SA IRK method $L$-stable. Note that a method is called $L$-stable if it is $A$-stable and if its stability function $R(z) \to 0$ when $z \to \infty$. For details, see Chap.~IV.3 in \cite{hairer1993solving2}.
}
\end{rem}

\begin{rem}
By Definition \ref{DefSA}, we get for the numerical solutions $y_{n+1} = Y_{ns}$, and $z_{n+1} = Z_{ns}$, i.e. they are identical to the last internal stage of the method. Furthermore, by the second equation in (\ref{eq:1-9bis}), we have $Z_{ni} = \mathcal{G}(Y_{ni})$ and then $g(y_{n+1}, z_{n+1})=0$, i.e. the numerical solutions lie on the manifold and it follows that the numerical solution $z_{n+1}$ depends on $y_{n+1}$, i.e. $z_{n+1} = \mathcal{G} (y_{n+1})$.
\end{rem}

\begin{rem}
If the method is stiffly accurate, we say that the numerical solutions of the numerical method \eqref{eq:1-10bis}, \eqref{eq:1-9bis}, with $Z_{ni} = \mathcal{G}(Y_{ni})$ and $z_{n+1} = \mathcal{G} (y_{n+1})$ are  identical to the solutions of the Eq.  (\ref{eqy}) with the same Runge-Kutta method,  \cite{hairer1993solving2}.
\end{rem}

Now we review the main convergence results of IRK methods for SPPs, for a detailed review we refer the reader to \cite{hairer1988error, hairer1993solving2}. This result represents the starting point of convergence analysis for InDC methods applied to (\ref{spp}). 

%\QQ{ (Shall it be Theorem 3.5 in the book? Further more the estimate for the $z$, is there an $\eps$??? Please double check!!!)
Under the assumptions of Theorem 3.8 in Chap.~IV.3 in \cite{hairer1993solving2}, the global error of an IRK method satisfies the following convergence results
 \begin{eqnarray*}
 y_{n} - y(t_n) = \mathcal{O}(h^p) +\mathcal{O}(\eps h^{q+1}), \quad   z_{n} - z(t_n) = \mathcal{O}(h^{q+1}).
 \end{eqnarray*}
In addition, if the method is stiffly accurate, we have
 \begin{eqnarray*}
 z_{n} - z(t_n) = \mathcal{O}(h^p) +\mathcal{O}(\eps h^{q}),
 \end{eqnarray*}
where $p$ is the {\em classical} order of the method, and $q$ is the {\em stage order} of the method, (i.e. condition $C(q)$ of section IV.5 in \cite{hairer1988error}). 

Our idea here is to use the error analysis of  IRK methods applied to SPPs obtained in Chap.~VI.3 of \cite{hairer1993solving2}, and extend them to the InDC methods. In fact, in order to do that, we perform an asymptotic expansion of smooth solutions of the system (\ref{spp}) and similarly for the numerical solutions of an IRK method applied to (\ref{spp}). The errors of the $y$ and $z$-component are formally considered as 
\beq
\label{errorExp}
y_n - y(t_n)  =  \sum_{\nu \ge 0} \varepsilon^{\nu} (y_{n,\nu} -y_\nu(t_n)), \quad  z_n - z(t_n)  =  \sum_{\nu \ge 0} \varepsilon^{\nu} (z_{n,\nu} -z_\nu(t_n)),
\eeq
where values $y_{\nu}(t)$, $z_{\nu}(t)$ are coefficients of the $\varepsilon$-expansion of the smooth solution for (\ref{spp})  and $y_{n,0}$, $z_{n,0}$, $y_{n, 1}$, $z_{n,1}, ...$, represent the numerical solution of the RK method applied to DAEs of arbitrary order.  Furthermore, the first differences $y_{n,0} -y_0(t_n)$ and $z_{n,0} -z_0(t_n)$ in the expansion (\ref{errorExp}) are the global errors of the RK method applied to the reduced system (\ref{reduced}), i.e.  system of index 1. The other differences for $\nu > 0$ in (\ref{errorExp}) are related to the numerical solutions of the RK method when applied to the DAEs of higher index. 
%%%%%%%%%%%%%%
%\SB{DELETE:Finally, we state the main result of global errors estimate (\ref{errorExp}) in Theorem~\ref{Thm:PrincTh} below for IRK methods when applied to SPPs.}
For details, see %\SB{DELTEChap. VI. 3 in} 
\cite{hairer1993solving2}.% Theorem 3.3, 3.4, 3.8 and Corollary 3.10 of \cite{hairer1993solving2}. 

}

\section{InDC Formulations Applied to SPPs}
\label{sec2}
\setcounter{equation}{0}
\setcounter{figure}{0}
\setcounter{table}{0}
In this section, we consider InDC-IRK method for the solution of SPPs written in the form of (\ref{spp}). The use of uniform nodes is important for the increase of high order of accuracy, if high order RK methods are used in correction loops. This is related to the concept of ``smoothness of the rescaled error vector", when we apply high order RK methods in correction loops, for more details see \cite{christlieb2009integral}. The use of quadrature nodes excluding the left-most endpoint leads to an important stability condition for stiff problems, i.e. the method is L-stable {if A-stable with $R(\infty) = 0$}, see \cite{layton2005implications}. Then, in this paper, we consider the InDC methods with uniform nodes excluding the left-most endpoint.

\subsection{InDC Framework}
We consider InDC procedure \cite{dutt2000spectral} applied to a SSP,
\begin{eqnarray}
\label{mainproblem1}
\begin{array}{l}
y'(t) = f(y,z), \ \ \ y(t_0) = y_0,  \\ 
\varepsilon z'(t) = g(y,z), \ \ \ z(t_0) = z_0. 
\end{array}
\end{eqnarray}
The time interval $[0, T]$ is discretized into intervals $[t_n, t_{n+1}]$, $n = 0,1,...,N-1$ such that
\begin{eqnarray*} 
0 = t_0 < t_1 < t_2 < ... < t_n < ...< t_N = T,
\end{eqnarray*}
with the step size $H$. Then, each interval $[t_n, t_{n+1}]$ is discretized again into $M$ uniform subintervals with quadrature nodes referred to as
\begin{eqnarray}
\label{eq: node} 
t_n \doteq \tau_0 < \tau_1 <\cdots < \tau_M \doteq t_{n+1}.
%t_n = t_{n, 0} < t_{n, 1} < \cdots < t_{n, M} = t_{n+1}, 
%\ \ \ m =1,...,M
\end{eqnarray}
Let $h = \frac{H}{M}$ be the size of a substep. {For simplicity of notation, we assume that $h$ is constant.} In this paper, the interval $[t_n,  t_{n+1}]$ will be referred to as a time step while a subinterval $[\tau_m, \tau_{m+1}]$ will be referred to as a substep. We remark that the size of time interval $[t_n, t_{n+1}]$ may vary as the InDC method is a one-step, multi-stage method. We assume the InDC quadrature nodes are uniform, which is a crucial assumption for high order improvement in accuracy, when we apply general high order IRK methods in prediction and correction steps for a classical ODE system \eqref{ODE}, (see discussions in \cite{christlieb2009integral}).
We also note that since $h = \frac{H}{M}$, we will use $\mathcal{O}(h^p)$ and $\mathcal{O}(H^p)$ interchangeably throughout the paper.

Let's assume we have obtained numerical solutions $\hat{y}^{(0)}_m$ and $\hat{z}^{(0)}_m$ approximating the exact solution at $\tau_m$ by using a low order numerical method for (\ref{mainproblem1}) {for a single time interval $[t_n, t_{n+1}]$ with $m=1, \cdots M$}. Here superscript $(0)$ is used to denote the prediction step in the InDC method. {Let us assume that} we build continuous polynomial interpolants $\hat{y}^{(0)}(t)$ and $\hat{z}^{(0)}(t)$ interpolating these discrete values. Now we define the error functions 
\beq
\label{eq: error_function}
e^{(0)}(t) = y(t)-\hat{y}^{(0)}(t), \quad d^{(0)}(t) = z(t)-\hat{z}^{(0)}(t), \quad {t \in [t_n, t_{n+1}]}.
\eeq
Note that $e^{(0)}(t)$ and $d^{(0)}(t)$ are not polynomials in general.
We specify the residual function with respect to $y$ and $z$ via the following set of differential equations
\beq
\label{eq: delta_diff}
\begin{array}{l}
\delta^{(0)}(t) = f(\hat{y}^{(0)}(t),\hat{z}^{(0)}(t))-(\hat{y}^{(0)})'(t), \\[2mm]
\rho^{(0)}(t) = g(\hat{y}^{(0)}(t),\hat{z}^{(0)}(t))-(\varepsilon \hat{z}^{(0)})'(t). 
\end{array}
\eeq
%Integrating \eqref{eq: delta_diff} from $t_0$ to $t$ gives,
%\begin{eqnarray}
%\label{integral1}
%\left\{
%\begin{array}{l}
%\delta^{(0)}(t) = y_0 + \int_{t_0}^{t} f(\hat{y}^{(0)}(s),\hat{z}^{(0)}(s))ds-\hat{y}^{(0)}(t),\\
%\rho^{(0)}(t) = \varepsilon z_0 + \int_{t_0}^{t} g(\hat{y}^{(0)}(s),\hat{z}^{(0)}(s))ds-\varepsilon \hat{z}^{(0)}(t). 
%\end{array}
%\right.
%\end{eqnarray}
Thus, by subtracting \eqref{eq: delta_diff} from \eqref{mainproblem1}, the error equations about the error functions \eqref{eq: error_function} become
\beq
\label{defint}
\begin{array}{l}
(e^{(0)})'(t) - \delta^{(0)}(t)= f(e^{(0)}(t)+\hat{y}^0(t),d^{(0)}(t)+\hat{z}^{(0)}(t))-f(\hat{y}^{(0)}(t),\hat{z}^{(0)}(t)),\\[2mm]
\varepsilon (d^{(0)})'(t)-\rho^{(0)}(t) = g(e^{(0)}(t)+\hat{y}^0(t),d^{(0)}(t)+\hat{z}^{(0)}(t))-g(\hat{y}^{(0)}(t),\hat{z}^{(0)}(t)).
\end{array}
\eeq
{A low order numerical method can be used to obtain numerical solutions $\hat{e}^{(0)}_m$ and $\hat{d}^{(0)}_m$ at $\tau_m$ by discretizing the error equations \eqref{defint}. Then the numerical solution can be improved as} 
\[
\hat{y}^{(1)}_m = \hat{y}^{(0)}_m +  \hat{e}^{(0)}_m, \quad \hat{z}^{(1)}_m = \hat{z}^{(0)}_m +  \hat{d}^{(0)}_m, \quad \forall m = 0, \cdots M.
\]
Such correction procedures can be repeated in each local time step $[t_n, t_{n+1}]$.
In summary, the strategy of InDC methods is to use a simple numerical method to compute numerical solutions $\hat{y}^{(0)}(t)$ and $\hat{z}^{(0)}(t)$ as prediction, and then to solve a series of correction equations in the integral form based on equations (\ref{defint}), each correction improves the accuracy of numerical solutions from the previous iteration.     

\begin{rem}  (About notations.) In our description of InDC, we let $y_m$ $z_m$, $e^{(k)}_m$, $d^{(k)}_m$ denote the exact solutions and exact error functions (without hat); and let $\hat{y}^{(k)}_m$, $\hat{z}^{(k)}_m$, $\hat{e}^{(k)}_m$, $\hat{d}^{(k)}_m$ denote the numerical approximations (with hat) to the exact solutions and error functions. 
We use subscript $m$ to denote the location $t = \tau_m$ and use superscript $(k)$ to denote the prediction ($k=0$) and correction loops ($k=1, \cdots$). We let $\bar{\cdot}$ denote the vector on InDC quadrature nodes, for example, $\bar{y} = (y_1, \cdots, y_M)$. 
\end{rem}

\subsection{InDC-BE method }

In this subsection, we consider InDC-BE method for the solution of system  \eqref{mainproblem1}. 
We use uniformly distributed quadrature nodes $\tau_1,...,\tau_M$ given by \eqref{eq: node} excluding the left-most endpoint.

\begin{enumerate}
\item 
(Prediction step) Use a {BE} discretization to compute 
%approximate solutions 
$$\bar{\hat{y}}^{(0)} = (\hat{y}^{(0)}_1,...,\hat{y}^{(0)}_m,...,\hat{y}^{(0)}_M)$$ 
as the approximation of the exact solution $\bar{y} = (y_1,...,y_m,...,y_M)$ for (\ref{mainproblem1}) at quadrature nodes $\tau_1,...,\tau_M$. We make the same for the $z$-component. This gives
\beq
\label{eq: Euler-eps}
\begin{array}{l} 
\hat{y}^{(0)}_{m+1} = \hat{y}^{(0)}_m + h f(\hat{y}^{(0)}_{m+1},\hat{z}^{(0)}_{m+1}),\\[2mm]
\varepsilon \hat{z}^{(0)}_{m+1} = \varepsilon \hat{z}^{(0)}_m + h g(\hat{y}^{(0)}_{m+1},\hat{z}^{(0)}_{m+1}),
\end{array}
\eeq
for $m = 0,1,...M-1$.
\item
(Correction loop). {Let $\hat{y}^{(k-1)}$ and $\hat{z}^{(k-1)}$ denote the numerical solutions at the $(k-1)^{th}$ sequence correction, for $k = 1,...,K$ with $K$ the number of correction steps.}  
\begin{enumerate}
\item Denote the error function at the $(k-1)^{th}$ correction by $e^{(k-1)}(t) = y(t) - \hat{y}^{(k-1)}(t)$, where $y(t)$
is the exact solution and $\hat{y}^{(k-1)}(t)$ is {a polynomial of degree $(M-1)$ interpolating}  $\bar{\hat{y}}^{(k-1)}$ at quadrature nodes $\tau_1,...,\tau_M$. Similarly denote $d^{(k-1)}(t) = z(t) - \hat{z}^{(k-1)}(t)$. 
Let $\delta^{(k-1)}(t)$ and $\rho^{(k-1)}(t)$ be defined by equation \eqref{eq: delta_diff}, but with the upper script $(0)$ replaced with $(k-1)$.
We compute the numerical error vector $\bar{\hat{e}}^{(k-1)} = (\hat{e}_1^{(k-1)},...,\hat{e}_M^{(k-1)})$ {where $\hat{e}_m^{(k-1)}$ is the approximation of $e^{(k-1)}(\tau_m)$ by applying a {BE} method to the integral form of (\ref{defint})} with $\hat{e}_m^{(k-1)}$ approximating $e^{(k-1)}(\tau_m)$ by applying a {BE} method to the integral form of (\ref{defint}),
\beq
\label{errorEqs}
\begin{array}{lll}
\displaystyle \hat{e}^{(k-1)}_{m+1} &=& \hat{e}^{(k-1)}_m + h \Delta f^{(k-1)}_{m+1} 
+ \int_{\tau_m}^{\tau_{m+1}}\delta^{(k-1)}(s)ds,\\[2mm]
\displaystyle \varepsilon \hat{d}^{(k-1)}_{m+1} & =& \varepsilon \hat{d}^{(k-1)}_m +  h \Delta g^{(k-1)}_{m+1}
+\int_{\tau_m}^{\tau_{m+1}}\rho^{(k-1)}(s)ds,
\end{array}
\eeq
where 
\beq
\begin{array}{lll}
\displaystyle \Delta f^{(k-1)}_{m+1} &=& f(\hat{y}^{(k-1)}_{m+1} + \hat{e}^{(k-1)}_{m+1},\hat{z}^{(k-1)}_{m+1}+\hat{d}^{(k-1)}_{m+1}) - f(\hat{y}^{(k-1)}_{m+1},\hat{z}^{(k-1)}_{m+1}),\\[2mm]
\displaystyle \Delta g^{(k-1)}_{m+1} &=&g(\hat{y}^{(k-1)}_{m+1} + \hat{e}^{(k-1)}_{m+1},\hat{z}^{(k-1)}_{m+1}+\hat{d}^{(k-1)}_{m+1}) - g(\hat{y}^{(k-1)}_{m+1},\hat{z}^{(k-1)}_{m+1}),
\end{array}
\eeq
and
\beq
\label{appint}
\begin{array}{l}
\int_{\tau_m}^{\tau_{m+1}}\delta^{(k-1)}(s)ds  =  \int_{\tau_m}^{\tau_{m+1}}f(\hat{y}^{(k-1)}(s),\hat{z}^{(k-1)}(s))ds - \hat{y}^{(k-1)}_{m+1} + \hat{y}^{(k-1)}_{m},\\[2mm]
\int_{\tau_m}^{\tau_{m+1}}\rho^{(k-1)}(s)ds = \int_{\tau_m}^{\tau_{m+1}}g(\hat{y}^{(k-1)}(s),\hat{z}^{(k-1)}(s))ds - \varepsilon \hat{z}^{(k-1)}_{m+1} + \varepsilon \hat{z}^{(k-1)}_{m}.
\end{array}
\eeq
The integral terms $\int_{\tau_m}^{\tau_{m+1}}$ in equations \eqref{appint} are approximated by a numerical quadrature.  
Especially, let $S$ be the integration matrix{;} its $(m, k)$ element is
\[
S^{m,k} = \frac{1}{h}\int_{\tau_m}^{\tau_{m+1}}\alpha_k(s) d s, \quad \mbox{for} \quad m=0, \cdots, M-1, \quad k=1, \cdots M,
\] 
where 
$\alpha_k(s)$
is the Lagrangian basis function based on the node $\tau_k$. {Note that $S^{m, k}$ can be obtained from the computation based on a standard interval $[0, 1]$.}
Let 
\beq
\label{eq: Sm}
S^m(\bar{f}) = \sum_{j = 1}^{M} S^{m,j}f(y_j,z_j),
\eeq
then 
\[
hS^m(\bar{f})-\int_{\tau_m}^{\tau_{m+1}}f(y(s), z(s))ds = \mathcal{O}(h^{M+1}),
\] 
for any smooth function $f$.  In other words, the quadrature formula given by $hS^{m}(\bar{f})$ approximates the exact integration with $(M+1)^{th}$ order of accuracy {\em locally}.
%\begin{rem}
%Considering the following change of variable $s = t_{0} + \sigma h$ in the interval $[t_n, t_{n+1}]$ we get $\alpha_{k}(t_0 + \sigma h) =  \prod _{i\neq k}^M\frac{\sigma-k}{i-k}$, i.e. there is only a dependence on $M$ and not on $h$. Then the integral 
%\beq
%S^{m,k} = \frac{1}{h}\int^{t_{m+1}}_{t_m}  \alpha_k(s)ds = \int^{m+1}_{m}  \alpha_k(t_0 + \sigma h )d\sigma 
%\eeq
%depends on $M$ and not on $h$.
%\end{rem}
\item Update the approximate solutions $\bar{\hat{y}}^{(k)} = \bar{\hat{y}}^{(k-1)}+ \bar{\hat{e}}^{(k-1)}$ and $ \bar{\hat{z}}^{(k)} = \bar{\hat{z}}^{(k-1)}+  \bar{\hat{d}}^{(k-1)}$.
\end{enumerate}
\end{enumerate}
\begin{rem} 
Using the notation introduced in equation \eqref{eq: Sm}, we get from equation (\ref{errorEqs}) and \eqref{appint},
\beq
\label{errorEqs2}
\begin{array}{l} 
\hat{y}^{(k)}_{m+1} = \hat{y}^{(k)}_m + h \Delta f_{m+1}^{(k-1)} + h S^{m}(\bar{\hat{f}}^{(k-1)}),\\[2mm]
\varepsilon \hat{z}^{(k)}_{m+1} =\varepsilon  \hat{z}^{(k)}_m + h \Delta g_{m+1}^{(k-1)}  + h S^{m}(\bar{\hat{g}}^{(k-1)}).
\end{array}
\eeq
\end{rem}
\begin{rem}
Since we consider the nodes excluding the left most quadrature point $t_0$, the order of approximation for integration/interpolation will be one order lower than the usual one considered in \cite{dutt2000spectral, christlieb2009integral}. 
\end{rem} 
\begin{rem}
The InDC-BE described above, can be generalized to the InDC-IRK method, for solving SPPs \eqref{mainproblem1}. To avoid heavy notations from the InDC-IRK method and for a better presentation of the paper, we organize the description of InDC-IRK method and the corresponding error estimates in %Section \ref{sec: InDC-IRK} of 
Appendix.  
\end{rem}

%%%%%%%%%%%%%%%%%%%%%%%

\subsection{$\eps$-asymptotic expansion}
\label{sec: eps-exp}

In this section, we introduce $\eps$-asymptotic expansion of the exact and numerical solution for system \eqref{mainproblem1}. This $\eps$-asymptotic expansion will be useful to study the behavior of the local error for the InDC method. 

{
We are mainly interested in smooth solutions of  \eqref{mainproblem1} which provide the $\eps$-asymptotic expansions for $t>0$, of the form
%Assume that the initial values for the system \eqref{mainproblem1} are such that its solution is smooth, then the following $\eps$-asymptotic expansion of the form 
\beq
\label{eq: exact-esp}
y(t)=\sum_{j=0}^{\infty} y_{j}(t) \eps^j, \quad
z(t)=\sum_{j=0}^{\infty} z_{j} (t) \eps^j.
\eeq
As just pointed out in the introduction, we suppose that the initial values of \eqref{mainproblem1} lie on the smooth solution, i.e. that an expansion of the form (\ref{eq: exact-esp}) holds}. 

From (\ref{eq: exact-esp}) we note that the exact solutions have a power series in $\eps$ and, considered a truncated series, a remainder after any $N+1$ number of terms could be obtained or estimated. In particular, for any $t \in [0, \bar{t}]$, the remainder is bounded above by a term $C_N\eps^{N+1}$ with $C_N>0$, for $\eps$ small enough, i.e.
\beq
\label{eq: exact-espbis}
y(t)=\sum_{j=0}^{N} y_{j}(t) \eps^j + \mathcal{O}(\eps^{N+1}),\quad
z(t)=\sum_{j=0}^{N} z_{j} (t) \eps^j + \mathcal{O}(\eps^{N+1}).
\eeq
%\beq
%\label{eq: exact-espbis}
%\displaystyle \begin{array}{l}
%%\sum_{\nu=0}^{\nu} e^{(k)}_{m, \nu} \eps^\nu + \mathcal{O}(\varepsilon^{\nu+1}) \\
 %\sum_{\nu=0}^{\nu} d^{(k)}_{m, \nu} \eps^\nu +\mathcal{O}(\varepsilon^{\nu+1})
%y(y) = y_{0}(t)  +  y_{1}(t) \eps  + \cdots  y_{\nu}(t) \eps^\nu +  \mathcal{O}(\varepsilon^{\nu+1}) \\
%\sum_{\nu=0}^{\nu} d^{(k)}_{m, \nu} \eps^\nu +\mathcal{O}(\varepsilon^{\nu+1})
%z(t) = z_{0}(t)  +  z_{1}(t) \eps  + \cdots  z_{\nu}(t) \eps^\nu +  \mathcal{O}(\varepsilon^{\nu+1}),  \\
%\end{array}
%\eeq
 %The series  (\ref{eq: exact-esp}) are asymptotic expansions of the solutions and converge asymptotically to them.

Furthermore we note that a sequence of DAEs arise in the study of \eqref{mainproblem1}. In fact, the coefficients in the expansion (\ref{eq: exact-esp}) are the solutions of DAEs of different indices, for more details see Chap.~VI.3 of \cite{hairer1993solving2}. 
This is obtained by inserting the $\varepsilon$-expansion of the exact solution \eqref{eq: exact-esp} into \eqref{mainproblem1} and collecting terms of equal powers of $\eps$. 
\beq
\label{eq: exact_eps_0}
\eps^0: \quad
\left \{
\begin{array}{l}
y'_0 = f(y_0, z_0) \\
0 =  g(y_0, z_0) 
\end{array}
\right.
,
\eeq
\beq
\label{eq: exact_eps_1}
\eps^1: \quad
\left \{
\begin{array}{l}
y'_1 = f_y (y_0, z_0)y_1+f_z (y_0, z_0)z_1\doteq \mathbb{F}_1 \\[2mm]
z'_0 = g_y (y_0, z_0)y_1+g_z (y_0, z_0)z_1 \doteq \mathbb{G}_1\\
\end{array}
\right.
,
\eeq
\beq
\cdots \nonumber
\eeq
\beq
\label{eq: exact_eps_nu}
\eps^{\nu}: \quad
\left \{
\begin{array}{ll}
y'_{\nu} &= f_y (y_0, z_0)y_{\nu}+f_z (y_0, z_0)z_{\nu} + \phi_{\nu} (y_0, z_0,\cdots, y_{\nu-1},z_{\nu-1})\doteq \mathbb{F}_{\nu} \\[2mm]
z'_{\nu-1} &= g_y (y_0, z_0)y_{\nu}+g_z (y_0, z_0)z_{\nu} + \psi_{\nu} (y_0, z_0,\cdots, y_{\nu-1},z_{\nu-1})\doteq \mathbb{G}_{\nu}\\
\end{array}
\right.
,
\eeq
with initial values $y_\nu(0)$, $z_\nu(0)$ known from \eqref{eq: exact-esp}. We observe that system \eqref{eq: exact_eps_0} under the condition \eqref{redg} is a DAE of index 1. According to  \cite{hairer1993solving2}, if we consider  \eqref{eq: exact_eps_0} and \eqref{eq: exact_eps_1} together, we have a differential algebraic system of index 2. In general  \eqref{eq: exact_eps_0}-\eqref{eq: exact_eps_nu} is a differential algebraic system of index $\nu$.\

%Evaluating \eqref{eq: asymptoticSol} at $\tau_{m}$ gives 
%\beq
%\left (
%\begin{array}{l}
%y_m\\
%z_m
%\end{array}
%\right )
%=
%\left (
%\begin{array}{l}
%\sum_{\nu=0}^{\infty} y_{m, \nu} \eps^\nu \\
%\sum_{\nu=0}^{\infty} z_{m, \nu} \eps^\nu
%\end{array}
%\right )
%\eeq
%
{Now let us look for an $\eps$-asymptotic expansion of the numerical solution at the $k^{th}$ correction step of the InDC-BE method in the form
\beq
\label{eq: exp_expand}
\hat{y}^{(k)}_m= \sum_{\nu=0}^{\infty} \hat{y}^{(k)}_{m, \nu} \eps^\nu, \qquad
\hat{z}^{(k)}_m =\sum_{\nu=0}^{\infty} \hat{z}^{(k)}_{m, \nu} \eps^\nu.
\eeq
The case of $k=0$ corresponds to the prediction step of InDC method. Then, inserting the above ansatz \eqref{eq: exp_expand} into the numerical scheme \eqref{eq: Euler-eps}-\eqref{appint}, and collecting terms of equal powers of $\varepsilon$, we have the following: }
%for the case $\nu = 0, 1$
\bit
\item for the prediction step ($k=0$)
\begin{eqnarray}
\label{eq:Euler-0}
&\eps^0&: \quad
\left \{
\begin{array}{l}
\hat{y}^{(0)}_{m+1,0} = \hat{y}^{(0)}_{m,0} + h f(\hat{y}^{(0)}_{m+1,0}, \hat{z}^{(0)}_{m+1,0}), \\ [2mm]
0 = g(\hat{y}^{(0)}_{m+1,0}, \hat{z}^{(0)}_{m+1,0}),
\end{array}
\right.\\
\label{eq:Euler-1}
&\eps^1&: \quad
\left \{
\begin{array}{l}
\hat{y}^{(0)}_{m+1,1} = \hat{y}^{(0)}_{m,1} + h \hat{\mathbb{F}}^{(0)}_{m+1,1},
\\  [2mm]
\hat{z}^{(0)}_{m+1,0} = \hat{z}^{(0)}_{m,0} +  h \hat{\mathbb{G}}^{(0)}_{m+1,1},
\end{array}
\right.
\end{eqnarray}
where
\beq
\left \{
\begin{array}{l}\label{index1f}
\hat{\mathbb{F}}^{(0)}_{m+1,1} \doteq f_y (\hat{y}^{(0)}_{m+1,0}, \hat{z}^{(0)}_{m+1,0}) \hat{y}^{(0)}_{m+1,1} + f_z (\hat{y}^{(0)}_{m+1,0}, \hat{z}^{(0)}_{m+1,0}) \hat{z}^{(0)}_{m+1,1}, \\[2mm]
\hat{\mathbb{G}}^{(0)}_{m+1,1} \doteq
 g_y (\hat{y}^{(0)}_{m+1,0}, \hat{z}^{(0)}_{m+1,0})
\hat{y}^{(0)}_{m+1,1} + g_z (\hat{y}^{(0)}_{m+1,0}, \hat{z}^{(0)}_{m+1,0})
\hat{z}^{(0)}_{m+1,1},
\end{array}
\right.
\eeq
%\QQ{
%\beq
%\mbox{the case of} \eps^\nu
%\eeq
%}
\item {for the correction steps ($k\ge1$),}
\begin{eqnarray}
\label{eq:Euler-0-k}
&\eps^0&: \quad
\left \{
\begin{array}{lll}
\hat{y}^{(k)}_{m+1,0} &=& \hat{y}^{(k)}_{m,0} + h \Delta \hat{f}^{(k-1)}_{m+1,0}
+ h S^m(\bar{\hat{f}}^{(k-1)}_0), \\ [2mm]
0 &=& h \Delta \hat{g}^{(k-1)}_{m+1,0}
 +  hS^m(\bar{\hat{g}}^{(k-1)}_0), \\[2mm]
\end{array}
\right.
\\
\label{eq:Euler-1-k}
&\eps^1&: \quad
\left \{
\begin{array}{ll}
\hat{y}^{(k)}_{m+1,1} & = \hat{y}^{(k)}_{m,1} + h \Delta \hat{\mathbb{F}}^{(k-1)}_{m+1, 1}
+ h S^m(\bar{\hat{\mathbb{F}}}^{(k-1)}_1), \\[2mm] 
\hat{z}^{(k)}_{m+1,0} &= \hat{z}^{(k)}_{m,0} + h \Delta \hat{\mathbb{G}}^{(k-1)}_{m+1, 1}
+ h S^m(\bar{\hat{\mathbb{G}}}^{(k-1)}_1),\\ [2mm]
\end{array}
\right.
\end{eqnarray}
{where in \eqref{eq:Euler-0-k},} 
\beq
\left \{
\begin{array}{lll}\label{BEf}
\Delta \hat{f}^{(k-1)}_{m+1,0}& =&  f(\hat{y}^{(k)}_{m+1,0}, \hat{z}^{(k)}_{m+1,0})
- f(\hat{y}^{(k-1)}_{m+1,0}, \hat{z}^{(k-1)}_{m+1,0}),\\ [2mm]
\Delta \hat{g}^{(k-1)}_{m+1,0}& =&  g(\hat{y}^{(k)}_{m+1,0}, \hat{z}^{(k)}_{m+1,0})
- g(\hat{y}^{(k-1)}_{m+1,0}, \hat{z}^{(k-1)}_{m+1,0}),\\
\end{array}
\right.
\eeq
{and in \eqref{eq:Euler-1-k},}
\begin{eqnarray}
\Delta \hat{\mathbb{F}}^{(k-1)}_{m+1,1} &=&\hat{\mathbb{F}}^{(k)}_{m+1,1} - \hat{\mathbb{F}}^{(k-1)}_{m+1,1}\notag\\ [2mm]
&=&\left(f_y (\hat{y}^{(k)}_{m+1,0}, \hat{z}^{(k)}_{m+1,0})
\hat{y}^{(k)}_{m+1,1}  + f_z (\hat{y}^{(k)}_{m+1,0}, \hat{z}^{(k)}_{m+1,0})
\hat{z}^{(k)}_{m+1,1} \right)\notag \\ [2mm]
&&-\left(f_y (\hat{y}^{(k-1)}_{m+1,0}, \hat{z}^{(k-1)}_{m+1,0})
\hat{y}^{(k-1)}_{m+1,1}  + f_z (\hat{y}^{(k-1)}_{m+1,0}, \hat{z}^{(k-1)}_{m+1,0})
\hat{z}^{(k-1)}_{m+1,1} \right),  %[2mm]
\label{eq: tempDF1}
\end{eqnarray}
where 
\beqa
\label{eq: Fk}
\hat{\mathbb{F}}^{(k)}_{m+1,1} = f_y(\hat{y}^{(k)}_{m+1,0},\hat{z}^{(k)}_{m+1,0}) \hat{y}^{(k)}_{m+1,1} + f_z(\hat{y}^{(k)}_{m+1,0},\hat{z}^{(k)}_{m+1,0}) \hat{z}^{(k)}_{m+1,1}. 
\eeqa
%\textcolor{red}{ MOVED THIS PART IN LEMMA 4.4 PLEASE TAKE A LOOK}
%\beqa \label{eq: temp1}
%\Delta \hat{\mathbb{F}}^{(k-1)}_{m+1,1 &=&f_y (y_{m+1,0}, {z}_{m+1,0}) \hat{e}^{(k-1)}_{m+1,1} + f_z ({y}_{m+1,0}, {z}_{m+1,0})
%\hat{d}^{(k-1)}_{m+1,1} + \mathcal{O}(h^{s_{k-1}+1}) 
%\\ [2mm]
%&\doteq & f_y  \hat{e}^{(k-1)}_{m+1,1}   + f_z \hat{d}^{(k-1)}_{m+1,1} + \mathcal{O}(h^{s_{k-1}+1}), \notag
%\eeqa
%where 
%\beqa
%\label{eq: Fk}
%\begin{array}{lll}
%\hat{\mathbb{F}}^{(k)}_{m+1,1} = f_y(\hat{y}^{(k)}_{m+1,0},\hat{z}^{(k)}_{m+1,0}) \hat{y}^{(k)}_{m+1,1} + f_z(\hat{y}^{(k)}_{m+1,0},\hat{z}^{(k)}_{m+1,0}) \hat{z}^{(k)}_{m+1,1}. 
%\eeqa
\eit
We note that both equations~\eqref{eq:Euler-0}-\eqref{eq:Euler-1} for the prediction step ($k = 0$), and equations~\eqref{eq:Euler-0-k}-\eqref{eq:Euler-1-k} for the correction step ($k\ge1$), are consistent discretizations of equation \eqref{eq: exact_eps_0}-\eqref{eq: exact_eps_1}. It is possible to generalize the $\varepsilon$-asymptotic expansion to $\eps^\nu$ ($\nu \ge2$), but we skip this to avoid heavy notations.

%are consistent discretizations of equation \eqref{eq: exact_eps_0}-\eqref{eq: exact_eps_1}, and for the correction step   too.

%\textcolor{red}{MOVED THIS PART IN LEMMA 4.4}
%Equation \eqref{eq: temp1} is justified by assuming $y_{m,0} - \hat{y}^{(k)}_{m,0} = \mathcal{O}(h^{k+2})$ locally for all $k$\QQ{, which will be justified by estimating the local error of InDC method for the index 1 problem, see Lemma~\ref{lemma2} in Section~\ref{sec4}.}
%For simplicity of notations we let $ f_y (y_{m+1,0}, {z}_{m+1,0}) = f_y$ when there is no confusion, and similarly for $f_z$. Hence, from equation \eqref{eq: temp1},  
%\beq
%\label{eq: Delta_F}
%\Delta \hat{\mathbb{F}}^{(k-1)}_{m+1,1} = 
%f_y  \hat{e}^{(k-1)}_{m+1,1}   + f_z \hat{d}^{(k-1)}_{m+1,1} + \mathcal{O}(h^{s_{k-1}+1}).
%\eeq
%Similarly, we have
%\beq
%\label{eq: Delta_G}
%\Delta \hat{\mathbb{G}}^{(k-1)}_{m+1,1} 
%=  g_y  \hat{e}^{(k-1)}_{m+1,1}   + g_z \hat{d}^{(k-1)}_{m+1,1} + \mathcal{O}(h^{s_{k-1}+1}). 
%\eeq
% and
% $z_{m,0} - \hat{z}^{(k)}_{m,0} = \mathcal{O}(h^{s_{k}+1})$. For simplicity of notations we let
%$ f_y (t_{m+1}, y_{m+1,0}, {z}_{m+1,0}) = f_y$, and similarly for $f_z$, $g_y$ and $g_z$.
%In equation \eqref{eq:Euler-1-k}, 
%$\bar{\hat{\mathbb{F}}}^{(k)}_1 = (\hat{\mathbb{F}}^{(k)}_0, \cdots, \hat{\mathbb{F}}^{(k)}_M)$, 
%$\bar{\hat{\mathbb{G}}}^{(k)}_1 = (\hat{\mathbb{G}}^{(k)}_0, \cdots, \hat{\mathbb{G}}^{(k)}_M)$, 
%with

{Finally,  let $\eps$-asymptotic expansion of error functions $e^{(k)}(t)$,  $d^{(k)}(t)$ at the $k^{th}$ iteration be
\beq
\label{eq: eps_error}
\left (
\begin{array}{l}
e^{(k)}_{m}\\
d^{(k)}_{m}
\end{array}
\right )
=
\left (
\begin{array}{l}
\sum_{\nu=0}^{\infty} e^{(k)}_{m, \nu} \eps^\nu \\
\sum_{\nu=0}^{\infty} d^{(k)}_{m, \nu} \eps^\nu
\end{array}
\right )
=
\left (
\begin{array}{l}
\sum_{\nu=0}^{\infty} (y_{m, \nu} - \hat{y}^{(k)}_{m, \nu}) \eps^\nu \\
\sum_{\nu=0}^{\infty} (z_{m, \nu} - \hat{z}^{(k)}_{m, \nu}) \eps^\nu
\end{array}
\right ).
\eeq
}
%\QQ{(From Qiu: this is not clear (need to be rephrased). more over, the remainder for the $e$ does it have the factor $1/h$???)
%We note that (\ref{eq: eps_error}) represent a power series expansion in $\varepsilon$ and a remainder, consider a truncated series, could be obtained explicitly or estimated. Then given the truncated series of (\ref{eq: exp_expand}) and (\ref{eq: exact-esp})  we have: %In brief, we say that for the error functions we have:
 %\beq\label{EstRem}
%\left (
%\begin{array}{l}
%e^{(k)}_{m}\\
%d^{(k)}_{m}
%\end{array}
%\right )
%=
%\left (
%\begin{array}{l}
%\sum_{\nu=0}^{q^{(0)}+1} e^{(k)}_{m, \nu} \eps^\nu + \mathcal{O}(\varepsilon^{\nu+1}/h) \\
 %\sum_{\nu=0}^{q^{(0)}+1} d^{(k)}_{m, \nu} \eps^\nu + \mathcal{O}(\varepsilon^{\nu+1}/h)
%\end{array}
%\right ).
%\eeq
%This result gives an estimate of the remainder for (\ref{eq: eps_error}). 
%for $\varepsilon$ sufficiently small and for $t$ on any bounded interval ($0 \le t \le \bar{t}$). 
%The coefficients $e^{(k)}_{m, \nu} = y_{m, \nu} - \hat{y}^{(k)}_{m, \nu}$ and $d^{(k)}_{m, \nu} = z_{m, \nu} - \hat{z}^{(k)}_{m, \nu}$ are given  by (\ref{eq: exact-esp}) and  (\ref{eq: exp_expand}). We will discuss about the estimates of the remainder  in Section \ref{SectRem} in the Appendix.
%Then this guarantees that even approximations that are limited to a finite value of $N$, i.e when a truncated series is considered, can be of great practical use even if the infinite series can diverges. 
%}

{Note that in the above ansatz, we consider truncated series of (\ref{eq: eps_error}) with estimate of the remainder as}

%\beq\label{EstRem} 
%\left (
%\begin{array}{l}
%e^{(k)}_{m}\\
%d^{(k)}_{m}
%\end{array}
%\right )
%=
%\left (
%\begin{array}{l}
%\sum_{\nu=0}^{N} e^{(k)}_{m, \nu} \eps^\nu + \QQ{\mathcal{O}(\varepsilon^{N+1})} \\
% \sum_{\nu=0}^{N} d^{(k)}_{m, \nu} \eps^\nu +\QQ{\mathcal{O}(\varepsilon^{N+1})}
%\end{array}
%\right ).
%\eeq
\beq\label{EstRem} 
\left (
\begin{array}{l}
e^{(k)}_{m}\\
d^{(k)}_{m}
\end{array}
\right )
=
\left (
\begin{array}{l}
%\sum_{\nu=0}^{\nu} e^{(k)}_{m, \nu} \eps^\nu + \mathcal{O}(\varepsilon^{\nu+1}) \\
 %\sum_{\nu=0}^{\nu} d^{(k)}_{m, \nu} \eps^\nu +\mathcal{O}(\varepsilon^{\nu+1})
\displaystyle e^{(k)}_{m, 0}  +  e^{(k)}_{m, 1} \eps  + \cdots  e^{(k)}_{m, \nu} \eps^\nu +  \mathcal{O}(\varepsilon^{\nu+1}) \\
%\sum_{\nu=0}^{\nu} d^{(k)}_{m, \nu} \eps^\nu +\mathcal{O}(\varepsilon^{\nu+1})
\displaystyle d^{(k)}_{m, 0}  +  d^{(k)}_{m, 1} \eps  + \cdots  d^{(k)}_{m, \nu} \eps^\nu + \mathcal{O}(\varepsilon^{\nu+1}) \\
\end{array}
\right ),
\eeq
where a finite number of terms $\nu$ are taken. We will see that $\nu$ is related to the value $q^{(0)}$, i.e. the stage order of the implicit RK method in the prediction step $k = 0$.

%where a finite number of terms $N$ are taken. We will see next that $N$ is related to the value $q^{(0)}$, i.e. the stage order of the implicit RK method in the prediction step $k = 0$. %with $N= q^{(0)}$ in the series (\ref{eq: eps_error}) are taken where $q^{(0)}$ represents the stage order of the implicit RK method in the prediction step $k = 0$. 

In {the} later part of this paper, our goal is to give rigorous estimates of the coefficients $e^{(k)}_{m, \nu} = y_{m, \nu} - \hat{y}^{(k)}_{m, \nu}$ and $d^{(k)}_{m, \nu} = z_{m, \nu} - \hat{z}^{(k)}_{m, \nu}$, given  by (\ref{eq: exact-esp}) and  (\ref{eq: exp_expand}) for $1 \le \nu \le q^{(0)}+1$ and, finally, estimates of the remainders will be given.   %as well as their remainders.
 
%and a global estimate of the error functions and of the their remainders will be obtained explicitly in Section \ref{SectRem} in the Appendix.

Similarly, we consider the $\eps$-asymptotic expansion of numerical approximations of error functions $\hat{e}^{(k)}(t)$, $\hat{d}^{(k)}(t)$ at the $k^{th}$ iteration
\beq
\label{eq: eps_error_hat}
\left (
\begin{array}{l}
\hat{e}^{(k)}_{m}\\
\hat{d}^{(k)}_{m}
\end{array}
\right )
=
\left (
\begin{array}{l}
\sum_{\nu=0}^{\infty} \hat{e}^{(k)}_{m, \nu} \eps^\nu\\
\sum_{\nu=0}^{\infty} \hat{d}^{(k)}_{m, \nu} \eps^\nu
\end{array}
\right )
=
\left (
\begin{array}{l}
\sum_{\nu=0}^{\infty} (\hat{y}^{(k+1)}_{m, \nu} - \hat{y}^{(k)}_{m, \nu}) \eps^\nu \\
\sum_{\nu=0}^{\infty} (\hat{z}^{(k+1)}_{m, \nu} - \hat{z}^{(k)}_{m, \nu}) \eps^\nu
\end{array}
\right ).
\eeq
{ Note that} combining \eqref{eq: eps_error} and \eqref{eq: eps_error_hat}, we get with $k, \nu \ge0$,  $m = 0, \cdots M$,
\beq
\label{eq: err_errhat}
%\left\{
%\begin{array}{l}
e^{(k)}_{m, \nu} = \hat{e}^{(k)}_{m, \nu} + e^{(k+1)}_{m, \nu},\quad
d^{(k)}_{m, \nu} = \hat{d}^{(k)}_{m, \nu} + d^{(k+1)}_{m, \nu}.
%\end{array}
%\right.
\eeq

\begin{rem}
Similar $\eps$-asymptotic expansions can be given for the numerical solutions of the InDC-IRK method. Again, to avoid heavy notations, we organize them in Appendix.%\ref{sec: InDC-IRK-expand}. } 
\end{rem}

\section{Main results and numerical evidence}
\label{sec3}
\setcounter{equation}{0}
\setcounter{figure}{0}
\setcounter{table}{0}

In this section, we present the main theoretical results in the form of theorems, and provide numerical evidence supporting the main theorems. We will provide a rigorous mathematical proof in the next section. 

\subsection{Main results}
The aim of this section is to present convergence results of the InDC-BE and InDC-IRK method  when applied to \eqref{mainproblem1}. 

\begin{thm} 
\label{thm: IDC_BE}
Consider the stiff system \eqref{spp}, \eqref{eq: gz} with initial values $y(0)$, $z(0)$ admitting a smooth solution. 
Consider the InDC-BE method constructed with $M$ uniformly distributed quadrature nodes excluding the left-most point and $K$ correction steps. 
Then the global errors after $K$ correction satisfy, 
\beq
\label{eq: final_estimate1}
\begin{array}{lll}
e^{(K)}_n = \hat{y}^{(K)}_n - y(t_n) &=& \mathcal{O}(H^{\min\{K+1, M\}}) + \mathcal{O}(\eps H),\\ %+ \mathcal{O}(\eps^2) + \mathcal{O}(\eps^3/H)\\
d^{(K)}_n=\hat{z}^{(K)}_n - z(t_n) &=& \mathcal{O}(H^{\min\{K+1, M\}}) + \mathcal{O}(\eps H),%+ \mathcal{O}(\eps^2) + \mathcal{O}(\eps^3/H),
\end{array}
\eeq
for $\eps \le cH$ and for any fixed constant $c>0$, where  $H = Mh$ is one InDC time step. The estimates hold uniformly for $H\le H_0$ and $nH \le Const$.
\end{thm}
\begin{thm} 
\label{thm: IDC_RK}
Consider the stiff system \eqref{spp}, \eqref{eq: gz} with initial values $y(0)$, $z(0)$ admitting a smooth solution. 
Consider the InDC method constructed with $M$ uniformly distributed quadrature nodes excluding the left-most point and  a stiffly accurate IRK method of order $p^{(0)}$, stage order $q^{(0)}$ with $(q^{(0)} < p^{(0)})$ for the prediction step. Apply IRK methods of different classical orders $(p^{(1)}, p^{(2)}, \ldots,  p^{(K)})$ 
in the correction loops, $k=1, \cdots K$. {Assume that each of these {IRK} methods in the prediction and correction loops are stiffly accurate and the matrices $A$ are nonsingular.}
Then the global errors after $K$ correction loops satisfy the estimates
\beq
\begin{array}{lll}\label{final_estimate}
e^{(K)}_n = \hat{y}^{(K)}_n - y(t_n) &=& \mathcal{O}(H^{\min\{s_{K}, M\}}) +\mathcal{O}(\eps H^{q^{(0)}}),\\%+\cdots +\mathcal{O}(\eps^{\nu} H^{q^{(0)}+1-\nu})  + \mathcal{O}(\eps^{\nu+1}/H) \\
d^{(K)}_n = \hat{z}^{(K)}_n - z(t_n) &=& \mathcal{O}(H^{\min\{s_K, {M}\}}) +\mathcal{O}(\eps H^{q^{(0)}}),%+\cdots +\mathcal{O}(\eps^{\nu} H^{q^{(0)}+1-\nu}) + \mathcal{O}(\eps^{\nu+1}/H) 
\end{array}
\eeq
for $\eps \le cH$ and for any fixed constant $c>0$, $s_{K} = \sum_{k=0}^{K} p^{(k)}$, and $H = Mh$ is one InDC time step. The estimates hold uniformly for $H\le H_0$ and $nH \le Const$.
\end{thm}
 
From the above two theorems, it is observed that the order of convergence for the first terms in (\ref{eq: final_estimate1}) and (\ref{final_estimate}) increases with the correction iteration $k$ , whereas the order for later terms does not change with the number of corrections $k$.
 
We note that (\ref{eq: eps_error}),  but replacing $m$ with $n$, can be adopted to represent the $\eps$-asymptotic expansion of the global error functions $e^{(K)}_n$ and $d^{(K)}_n$ at the $K$-th correction step, where $e^{(K)}_{n,\nu}$ and $d^{(K)}_{n,\nu}$ for $\nu = 0,1,\cdots$, are the global errors of {InDC stiffly accurate (SA) IRK method (InDC SA-IRK}), applied to the differential algebraic systems of different indices \eqref{eq: exact_eps_0}-\eqref{eq: exact_eps_nu}. In section \ref{sec4}, we only prove Theorem \ref{thm: IDC_BE} for estimating $e^{(K)}_{n,\nu}$ and $d^{(K)}_{n,\nu}$ with $\nu=0, 1$ for the InDC-BE method.

To avoid heavy notations and technical details, we prove Theorem \ref{thm: IDC_RK} for general InDC-IRK methods   and  estimate the remainder {of the expansion (\ref{EstRem})} in sections {\ref{mainT} and \ref{SectRem}} in the Appendix.

\subsection{Numerical evidence}
We present some numerical evidence of  the estimates given in Theorem~\ref{thm: IDC_BE} and Theorem~\ref{thm: IDC_RK}. 
Below, we consider the following InDC methods constructed with $M$ quadrature points. %(replacing "some \SB{InDC IRK and SA-IRK} methods.")} 
\bit
\item The InDC-BE method with $k$ correction steps (InDC-BE-M-k).
The BE method has order {$p=1$} and stage order {$q=1$}.
\item
The InDC method constructed with a second order stiffly accurate DIRK method in
$k$ correction steps (InDC-DIRK2-SA-M-k). The second order DIRK method ({\em DIRK2-SA}) has the Butcher tableau 
\beq
\begin{array}{c|cc}
\gamma & \gamma &0\\
1 & 1-\gamma & \gamma\\
\hline
 & 1-\gamma & \gamma 
 \end{array} \ \ \ \ \  
\eeq
where $\gamma = 1-\frac{\sqrt{2}}{2}$. This method is stiffly accurate with order {$p=2$} and stage order {$q=1$}.
\item The InDC method constructed with a second order non stiffly accurate midpoint method in
$k$ correction steps (InDC-DIRK2-NSA-M-k). 
The second order midpoint method ({\em DIRK2-NSA}) has the Butcher tableau 
\beq
\begin{array}{c|c}
1/2 & 1/2\\
\hline
 & 1
 \end{array} \ \ \ \ \ 
 .
\eeq
This method is not stiffly accurate, with order {$p=2$} and stage order {$q=1$}.

\item 
The InDC method constructed with a second order stiffly accurate Lobatto IIIA method (trapezoidal rule) in
$k$ correction steps (InDC-LobattoIIIA2-M-k). 
This method has the Butcher tableau
\beq
\begin{array}{c|cc}
0&0 &0\\
1& 1/2 & 1/2\\
\hline
 & 1/2 &1/2
 \end{array} \ \ \ \ \ .
\eeq
It is stiffly accurate with matrix $A$ singular. It is $A$-stable but not $L$-stable ($R(\infty) \neq 0$), order {$p=2$} and stage order {$q=1$}.
\item The InDC method constructed with a third order stiffly accurate Radau IIA method in the prediction step and with the BE method in 
$k$ correction steps (InDC-Radau-BE-M-k).
The third order Radau IIA method has the Butcher tableau
\beq
\begin{array}{c|cc}
1/3 & 5/12 &-1/12\\
1 & 3/4 & 1/4\\
\hline
 & 3/4 & 1/4\\
 \end{array} \ \ \ \ \ .
\eeq
This method is stiffly accurate with order {$p=3$} and stage order {$q=2$}.
\eit
The indicated order of convergence by Theorem~\ref{thm: IDC_BE} and Theorem~\ref{thm: IDC_RK} for the $y$ and $z$ components in the SPPs are summarized in Table~\ref{tab: error}. Below we discuss the convergence rates specified in Table \ref{tab: error}.
\bit
\item 
For the InDC-BE-M-k method, the order of convergence will increase with $k$ for the first error term in equation \eqref{eq: eps_error} when $\eps \ll H$ and  $k\le M-1$, leading to a term of $H^{\min(k+1, M)}$ for the differential and algebraic component in \eqref{eq: final_estimate1}.
{The BE method has stage order $q = 1$}. The order of convergence for the second error term in equation \eqref{eq: eps_error}
will be determined by the stage order of the prediction $q^{(0)} = 1$ when $k$ increases, leading to a term of $\eps H$ in equation \eqref{eq: final_estimate1}.
\item 
For the InDC-DIRK2-SA-M-k, the order of convergence will increase with $k$ by $2$ for the first error term in equation \eqref{eq: eps_error} when $\eps \ll H$ and  $k\le M-1$, leading to a term of $H^{\min(2(k+1), M)}$ for the differential and algebraic component in equation \eqref{eq: final_estimate1}. DIRK2-SA method has stage order $q = 1$. The order of convergence for the second error term in equation \eqref{eq: eps_error} will be determined by the stage order of the prediction $q^{(0)} = 1$ when $k$ increases, leading to a term of $\eps H$ in equation \eqref{eq: final_estimate1}.
 
\item 
An important ingredient, suggested by the analysis, is to require that the methods to be stiffly accurate, i.e. $a_{sj} = b_j$ for $j = 1, \cdots , s$ and that the matrix $A$ is nonsingular. Such a choice provides a significant benefit for the convergence of the numerical solution, without which the numerical solutions will diverge. For example, if we consider using the second order non stiffly accurate DIRK method in both the prediction and $k$ correction steps of an InDC framework with $M$ quadrature points (InDC-DIRK2-NSA-M-k), divergence results are expected (see Figure \ref{fig2.2}). Note that in the analysis for InDC-IRK method in the appendix, a satisfactory theoretical explanation of this fact is given. {Finally, if we consider using methods with singular matrix $A$, as for example the second order Lobatto IIIA method, in both the prediction and $k$ correction steps of an InDC framework with $M$ quadrature points (InDC-LobattoIIIA2-M-k, right plot in Figure \ref{fig2.2}), again divergence is expected.}

\item
For the InDC-Radau-BE-M-k, the order of convergence will increase 
with $k$ by $1$ for the first error term in equation \eqref{eq: eps_error} when $\eps \ll H$ and  $k\le M-1$, leading to a term of $H^{\min(3+k, M)}$ for the differential and algebraic component in equation \eqref{eq: final_estimate1}.
Radau IIA method has stage order $q = 2$. The order of convergence for the second error term in equation \eqref{eq: eps_error} will be determined by the stage order of the prediction $q^{(0)} = 2$ when $k$ increases, leading to a term of $\eps H^2$ in equation \eqref{eq: final_estimate1}.
\eit

\begin{table}[htb]
\begin{center}
\caption{{
Global error predicted by Theorem~\ref{thm: IDC_BE} and Theorem~\ref{thm: IDC_RK} with $H \gg \eps$. Note that `SA'/`NSA' means stiffly accurate/not stiffly accurate.
}\label{tab: error}
}
\bigskip
\begin{tabular}{|c | c|c|}
\hline
\cline{1-3} Method & $y-$comp &$z-$comp \\
\hline
\cline{1-3}  InDC-BE-M-k & $H^{\min(k+1,M)} + \eps H$ & $H^{\min(k+1,M)} + \eps H$ \\
\hline
\cline{1-3}   InDC-DIRK2-SA-M-k& $H^{\min(2(k+1),M)} + \eps H$ & $H^{\min(2(k+1),M)} + \eps H$ \\
\hline
\cline{1-3}  InDC-DIRK2-NSA-M-k  & diverges  & diverges\\
\hline
\cline{1-3}  InDC-LobattoIIIA2-SA-M-k  & diverges  & diverges\\
\hline
\cline{1-3}  InDC-Radau-BE-M-k  & $H^{\min(3+k,M)} + \eps H^2$ & $H^{\min(3+k,M)} + \eps H^2$\\
\hline
\end{tabular}
\end{center}
\end{table}

For numerical verification, we first consider a scalar example \cite{hairer1993solving2}
\beq
\eps z' =  - z + \cos(t) 
\eeq
with the analytical solution
\[
z(t) = \frac{\cos(t) + \eps \sin(t)}{1+\eps^2} + C\exp(-t/\eps),
\]
where $C = z(0) - 1$ is determined by the initial condition. For a consistent initial condition, let $C=0$. This is a good example to investigate the order of convergence for the $\eps^1$ term in equation~\eqref{errorExp}, as the error for $\eps^0$ is $0$. Indeed, for stiff parameter $\eps = 10^{-6}$ only a region of first order convergence is observed for the BE method, where the global and local error given for the $z$-component is $\mathcal{O}(\eps H)$ (see Corollary 3.10 in \cite{hairer1993solving2}). Figure~\ref{fig1} gives the one step error (local error) and global error of  BE method, expected $\mathcal{O}(\eps H)$ is observed. We also test the InDC-DIRK2-NSA-3-1 and InDC-LobattoIIIA2-SA-4-2 method.
Numerical results are presented in Figure~\ref{fig2.2}. Divergence results are observed when time step is large compared to $\eps$ if an InDC-correction is performed.

\begin{figure}[htb]
\begin{center}
\includegraphics[width=2.5in]{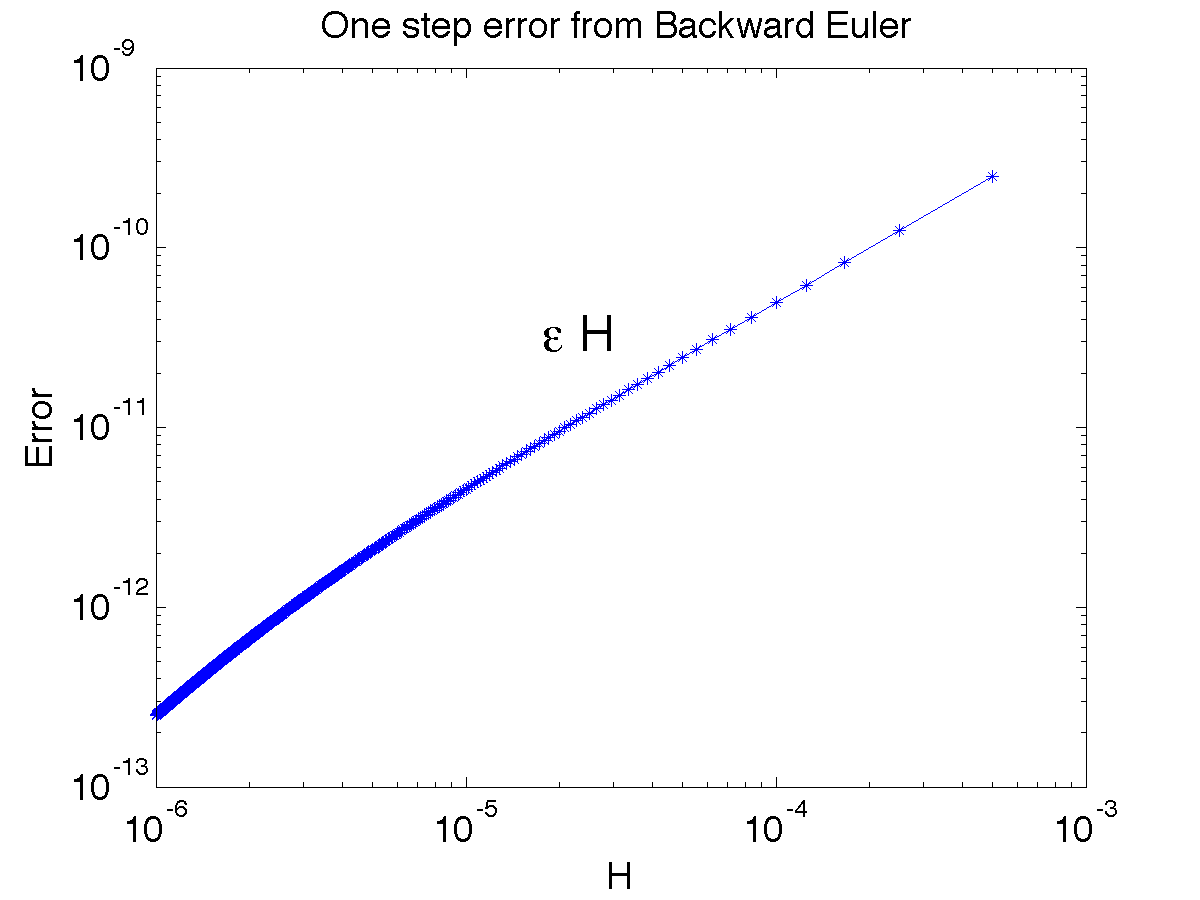}
\includegraphics[width=2.5in]{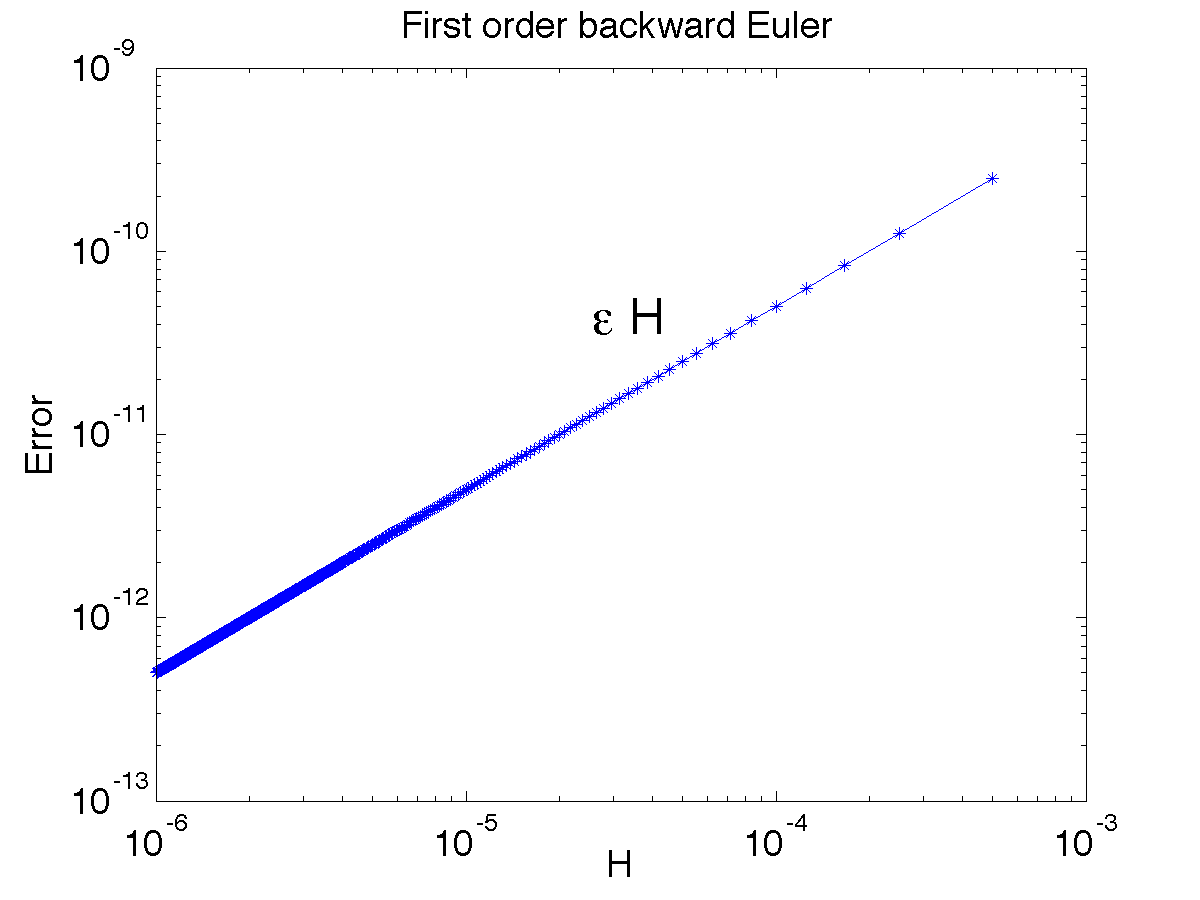}
\end{center}
\caption{Scalar example.
Local, i.e. one step error (left plot) and global error at $T = 0.5$ (right plot) of  BE method. $\mathcal{O}(\epsilon H)$ is observed in both plots with $\varepsilon = 10^{-6}$.
}
\label{fig1}
\end{figure}

\begin{figure}
\centering
\includegraphics[width=2.5in]{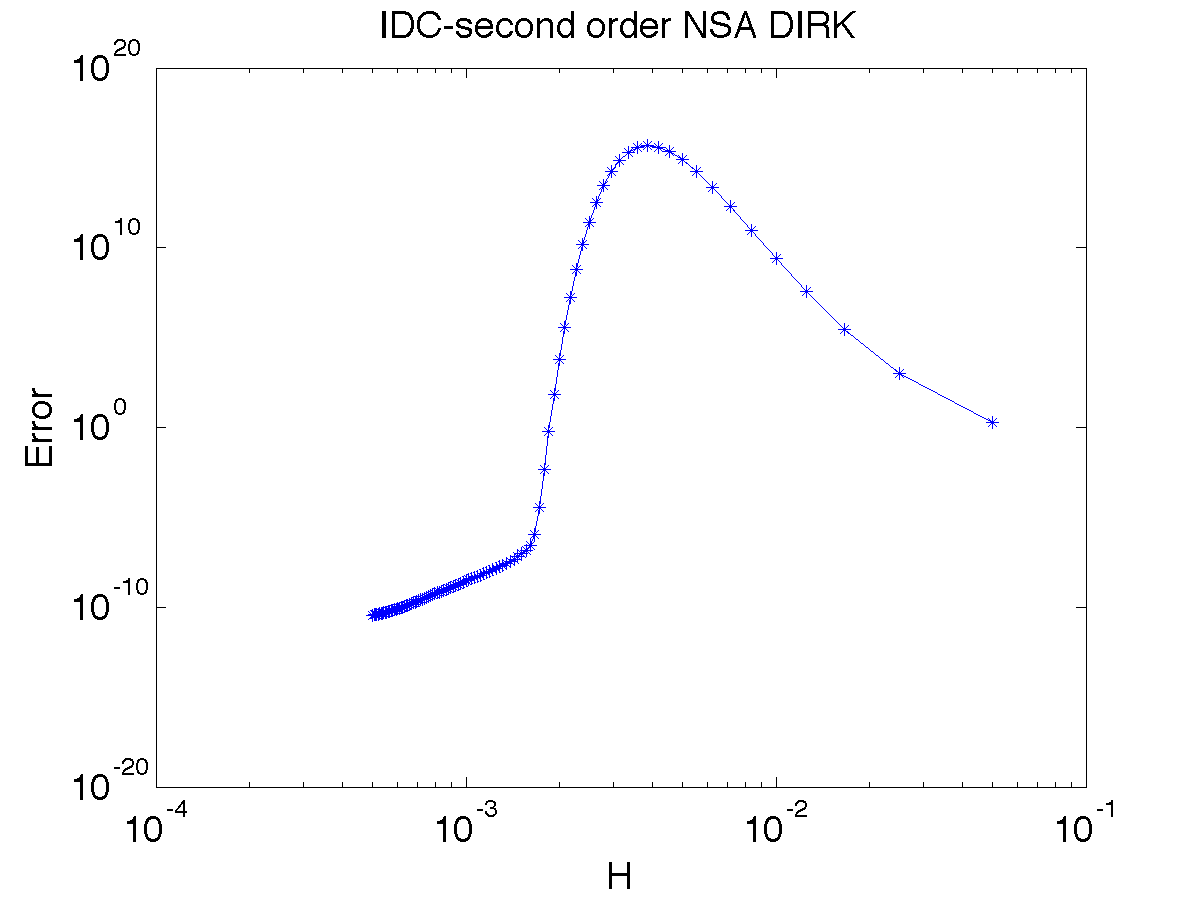}
\includegraphics[width=2.5in]{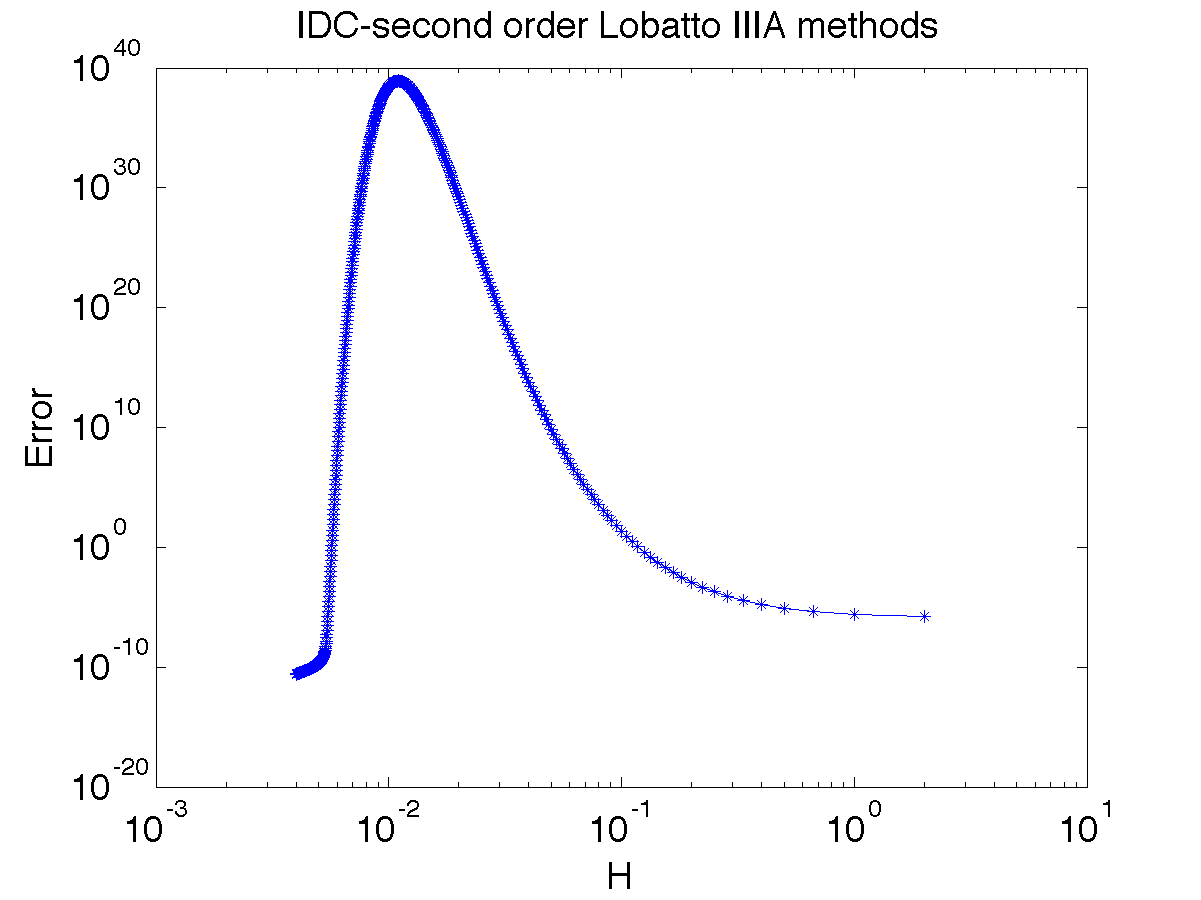}
\caption{Scalar example. $\eps = 10^{-4}$.
Left: global error ($T = 0.1$) of the InDC-second order DIRK method that is not SA with three quadrature points and one correction step.
{Right: global error ($T = 0.1$) of the InDC-second order Lobatto IIIA method with matrix $A$ singular, four quadrature points and one correction step.} 
}
\label{fig2.2}
\end{figure}

Now we consider the van der Pol equation \cite{hairer1993solving2} with the well-prepared initial data up to $\mathcal{O}(\eps^3)$
\beq
\left\{
\begin{array}{l}
y' = z \\
\eps z' = (1-y^2) z - y
\end{array}
\right.,
\quad
\left\{
\begin{array}{l}
y(0) = 2 \\
z(0) = -\frac23 + \frac{10}{81}\eps -\frac{292}{2187} \eps^2
\end{array}
\right.
\eeq
\bit
\item
Numerical results of the InDC-BE-3-2 method  are presented in the upper row of Figure~\ref{fig6}. The order of convergence for $\eps^0$ term would increase with the correction loops. The $\eps^1$ term of error behaves like $\mathcal{O}(\eps H)$ for both $y$ and $z$ components. 
\item The numerical results of the InDC-DIRK2-SA-4-1 method are presented in the middle row of Figure~\ref{fig6}. The order of convergence for $\eps^0$ term would increase with second order with the correction loop. The $\eps^1$ term of error behaves like $\mathcal{O}(\eps H)$ for both $y$ and $z$ components. 
\item The numerical results of the InDC-Radau-BE-6-2 method are presented in the botton row of Figure~\ref{fig6}. The order of convergence for $\eps^0$ term would increase with first order correction loop and is observed to be $\mathcal{O}(H^5)$. The $\eps^1$ term of error behaves like $\mathcal{O}(\eps H^2)$ for both $y$ and $z$ components. 
\eit

Numerical observations in Figure~\ref{fig6} are consistent with Theorem~\ref{thm: IDC_BE}, ~\ref{thm: IDC_RK} and Table~\ref{tab: error}. Especially, it is observed that the InDC SA-IRK methods exhibit order reduction both in differential and algebraic components. They produce an estimate for the $y$ and $z$ component in the form of equation \eqref{final_estimate}. 
For example, in Figure ~\ref{fig6}, we observe a behavior like $e^{(k)}_n = \mathcal{O}(H^{3}) +  \mathcal{O}(\eps H)$. Furthermore, if the step size $H > \eps^{\frac{1}{s_k-q^{(0)}}}$, $\mathcal{O}(H^{s_k})$ is dominant, otherwise the term $ \mathcal{O}(\eps H^{q^{(0)}})$ is observed.
We observe that in the neighborhood of $H \approx \eps^{\frac{1}{s_k-q^{(0)}}}$, we have a cancellation of error terms between $\mathcal{O}(H^{s_k})$
and $\eps \mathcal{O}(H^{q^{(0)}})$, if error constants are of opposite signs, see for example the plots in middle and bottom rows of Figure~\ref{fig6}.

\begin{figure}
\centering
\includegraphics[width=2.5in]{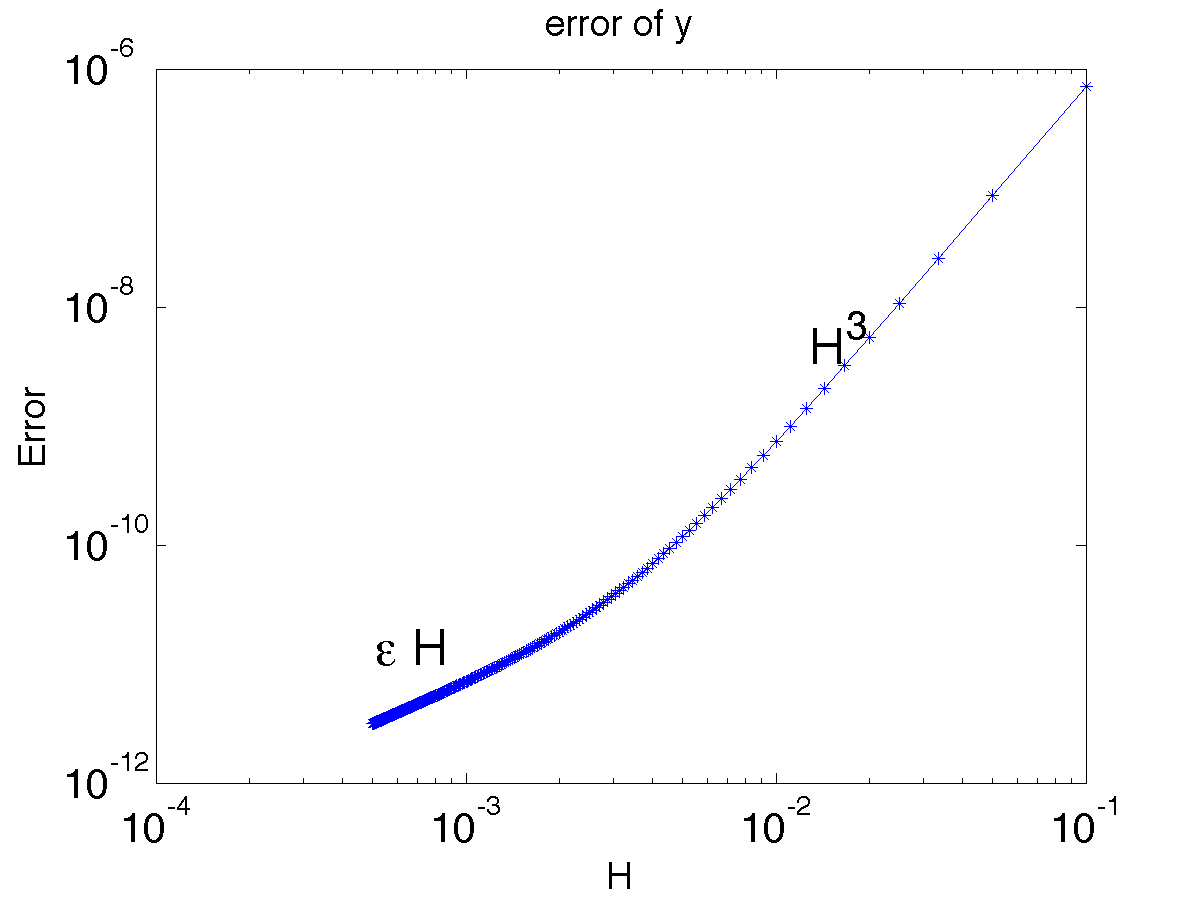},
\includegraphics[width=2.5in]{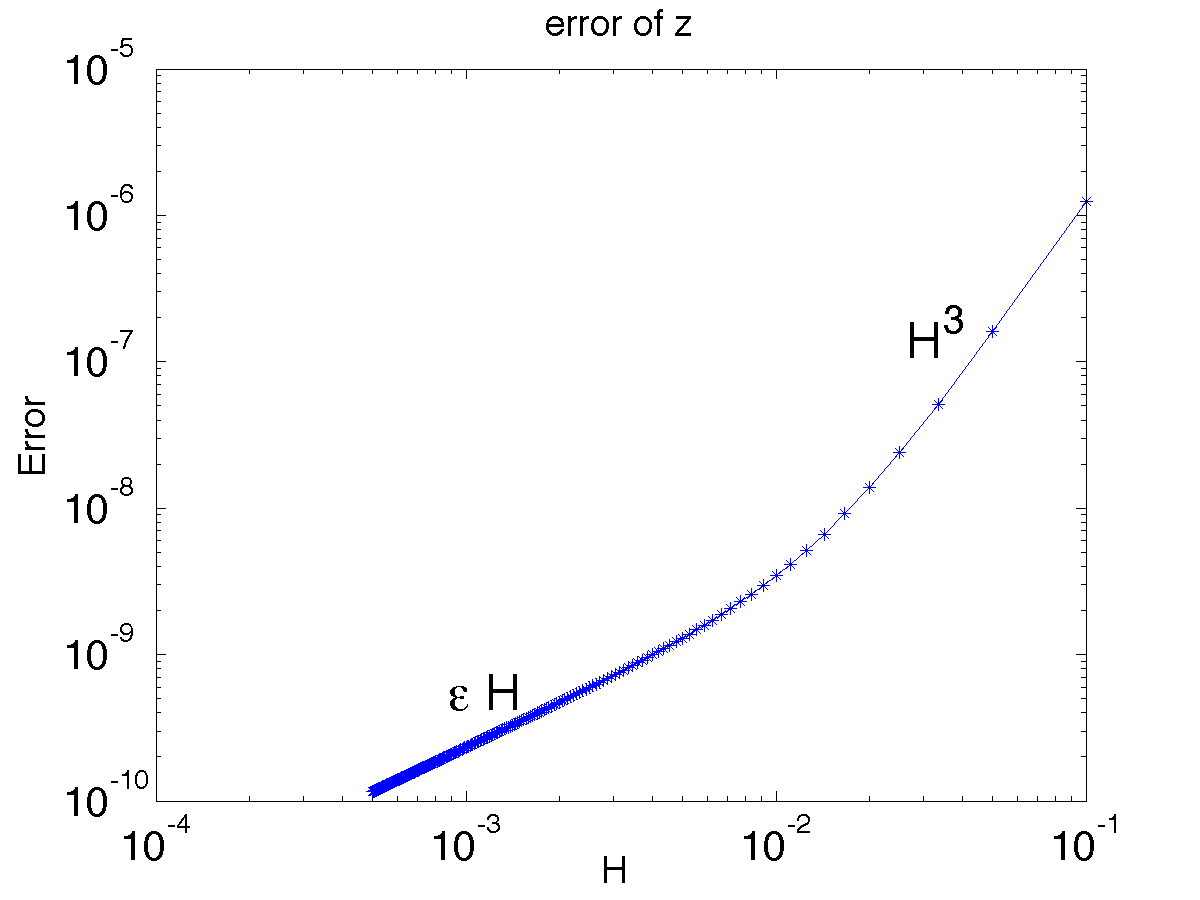}
\centering
\includegraphics[width=2.5in]{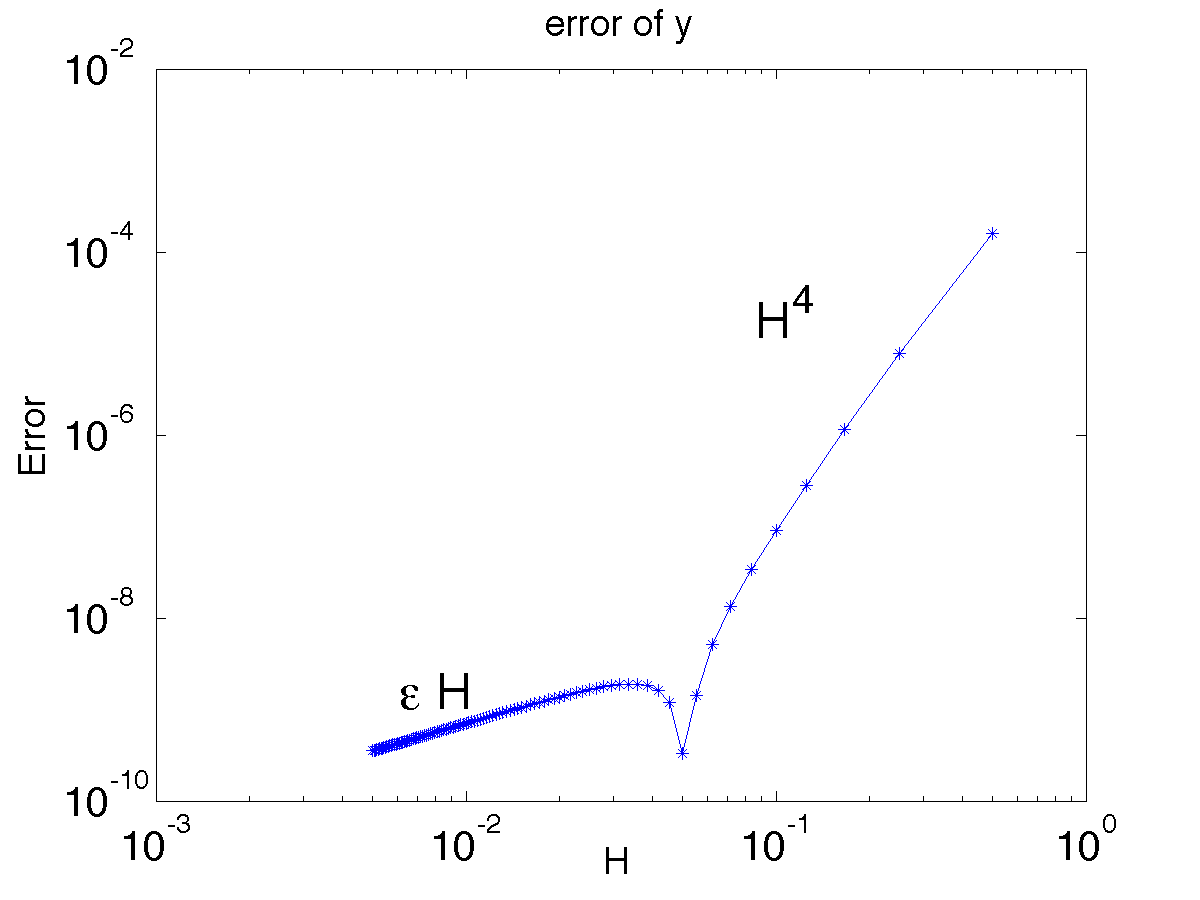},
\includegraphics[width=2.5in]{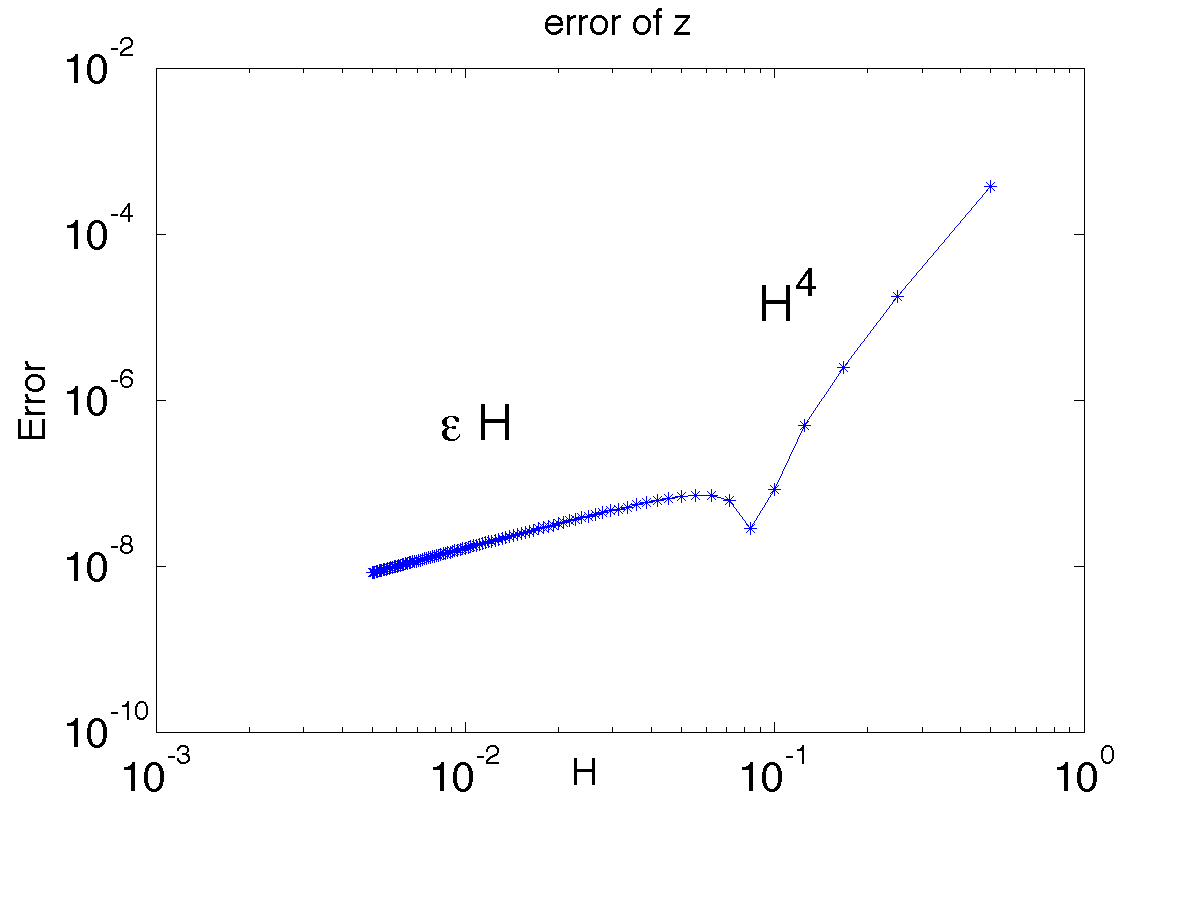}
\centering
\includegraphics[width=2.5in]{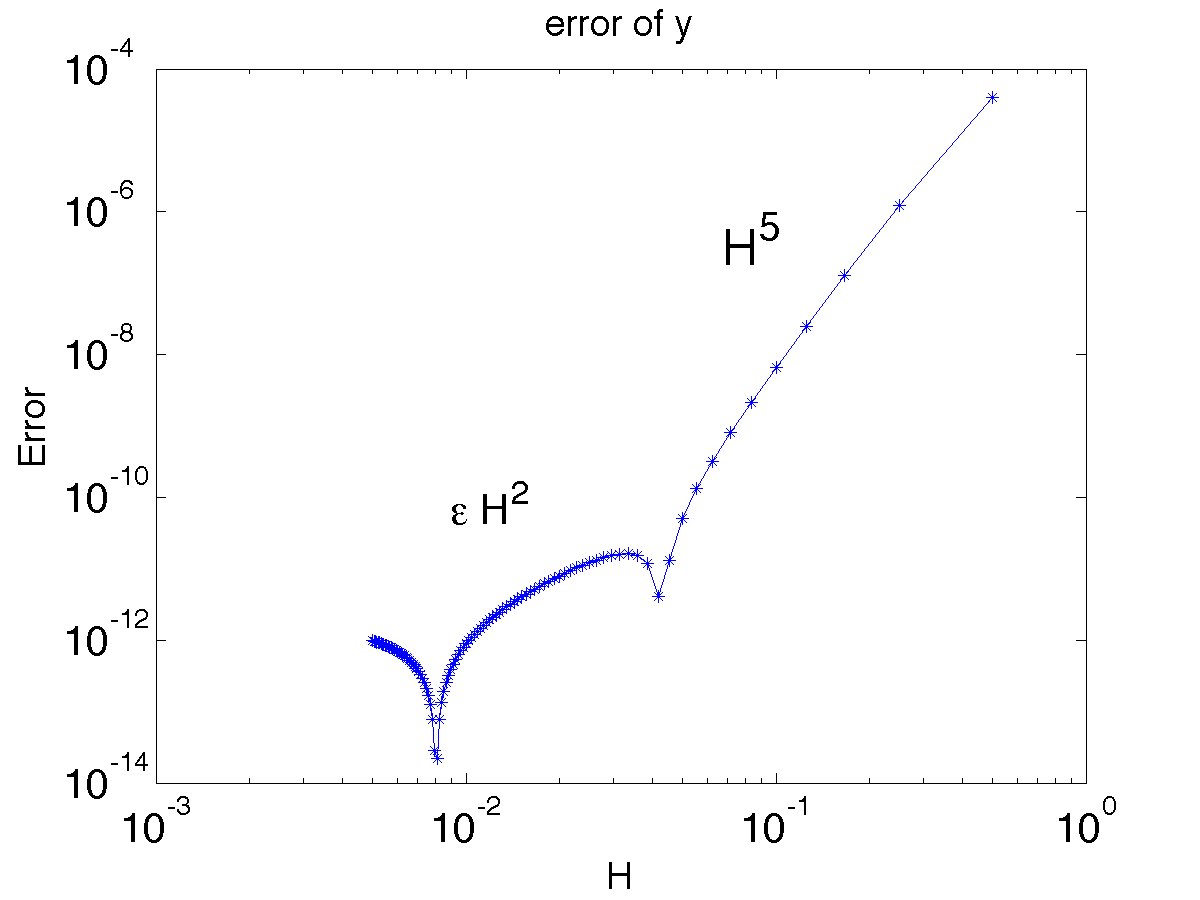},
\includegraphics[width=2.5in]{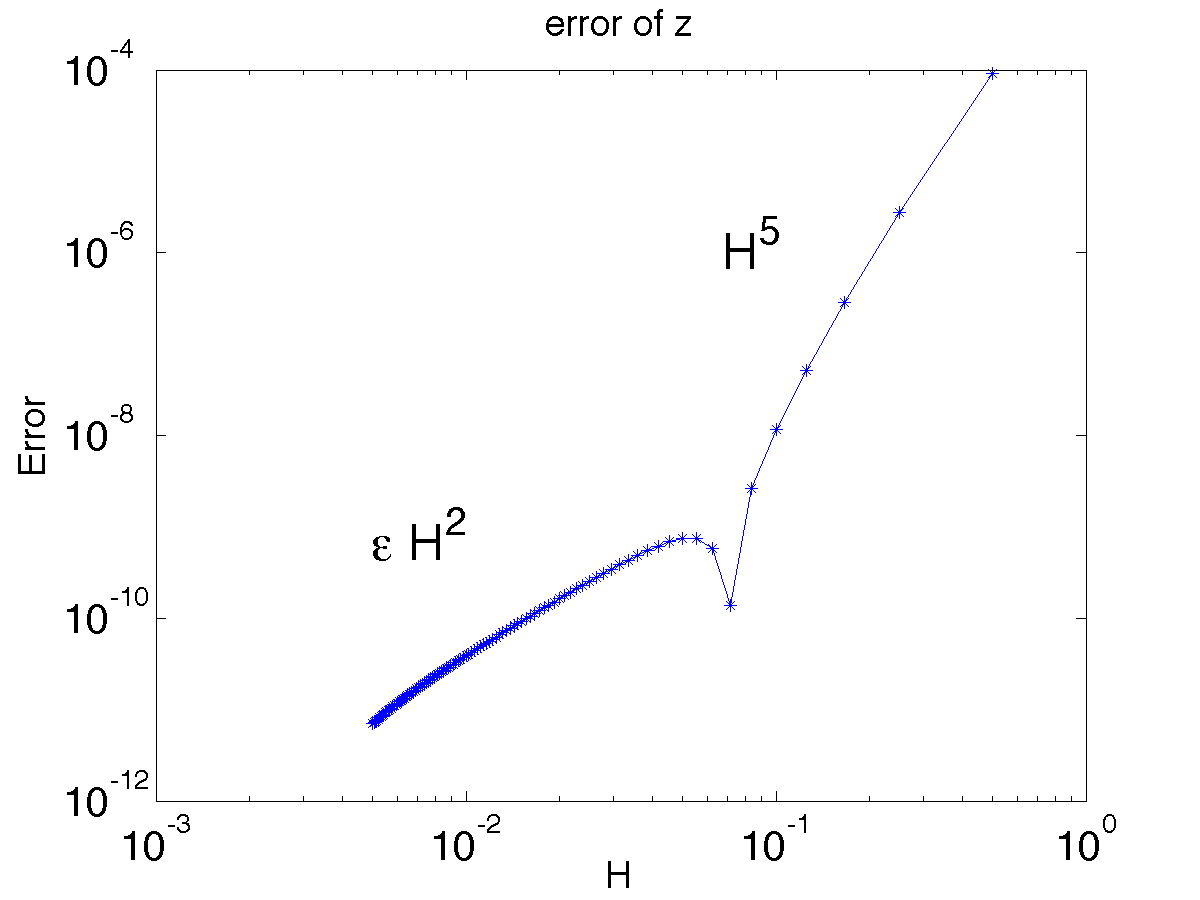}
\caption{
Van der Pol equation.
Global error ($T = 0.5$) of the InDC-BE-3-2 method (upper row); 
and of the InDC-DIRK2-SA-4-1 method  (middle row);
and of the InDC-Radau-BE-6-2 method  (bottom row).
$\eps = 10^{-6}$.
}
\label{fig6}
\end{figure}

\section{Proofs of main results}
\label{sec4}
\setcounter{equation}{0}
\setcounter{figure}{0}
\setcounter{table}{0}
In this section, we prove Theorem~\ref{thm: IDC_BE}, which is a special case of Theorem~\ref{thm: IDC_RK}. By proving Theorem~\ref{thm: IDC_BE} through several lemmas, we demonstrate the basic ingredients of the general proof for Theorem~\ref{thm: IDC_RK} presented in the Appendix. 
%yet we present the proof for Theorem~\ref{thm: IDC_BE} first . The proof is then generalized for Theorem~\ref{thm: IDC_RK}. 
Our error estimates are based on the $\eps$-expansion outlined in Section~\ref{sec: eps-exp}.

\subsection{Error estimates for Theorem \ref{thm: IDC_BE}.}

We perform local error estimate for Theorem \ref{thm: IDC_BE} by two Lemmas. We again note that since $h = \frac{H}{M}$, we use $\mathcal{O}(h^p)$ and $\mathcal{O}(H^p)$ interchangeably below in our proof. We then prove the global error estimate based on the two Lemmas.

\begin{lem}
\label{lemma1}
($\eps^0$ error term). 
Let us assume that the reduce system (\ref{reduced}) with $\varepsilon=0$ satisfies {\eqref{redg}} and that the initial values are consistent. 
Consider the InDC-BE method constructed with $M$ uniformly distributed quadrature nodes excluding the left-most point and $k$ correction steps, i.e. \eqref{eq:Euler-0} for the prediction step and \eqref{eq:Euler-0-k} for the correction one, with $k=1, \cdots, K$.
Then the numerical solutions satisfy the following local error estimates
at each interior node of InDC $\tau_m$ with $m = 0,...,M$,
\begin{eqnarray}
\label{eq: lemma2_local}
\begin{array}{l}
e^{(k)}_{m,0} = y_{m,0} - \hat{y}^{(k)}_{m,0} = \mathcal{O}(h^{\min(k+2,M+1)}), 
\quad
d^{(k)}_{m,0} = z_{m,0} - \hat{z}^{(k)}_{m,0} = \mathcal{O}(h^{\min(k+2,M+1)}),
\end{array}
\end{eqnarray} 
with
\beq
\label{eq: lemma2_g1}
g(\hat{y}^{(k)}_{m, 0}, \hat{z}^{(k)}_{m, 0}) = 0,
\eeq
for $k=0, \cdots, K$.
\end{lem}
\noindent
{\bf Proof.}
For $k = 0$, equation \eqref{eq: lemma2_g1} is a consequence of the consistency of the initial conditions for system  \eqref{eq: exact_eps_0}. Then we start to prove  the local error estimate \eqref{eq: lemma2_local} for the prediction step ($k=0$). For the exact solution, by equation \eqref{eq: exact_eps_0} and assumption \eqref{redg}, we have for the $y_0(t)$ component, equation (\ref{eqy}) and %\QQ{suggest to delete``, indicating that $y_0(t)$ and $z_0(t)$ lies on the manifold"} 
$g(y_0(t),z_0(t)) = 0$. By equation \eqref{redg}, it follows that $z_0(t) = \mathcal{G}(y_0(t))$. 
 
For the numerical solution, we have \eqref{eq:Euler-0}. By  
$
g( \hat{y}^{(0)}_{m,0}, \hat{z}^{(0)}_{m,0}) =0,
$
with $m = 0,...,M$, 
we get  $\hat{z}^{(0)}_{m,0} =\mathcal{G}(\hat{y}^{(0)}_{m,0})$
with $\hat{y}^{(0)}_{m,0}$ being numerical solution of the ordinary differential equation (\ref{eqy}). %$y'_0(t) = \hat{f}(y_0(t))$ with $\hat{f} \doteq f(y_0(t), \mathcal{G}(y_0(t)))$. 
Then from classical error estimates for the BE method, we have for the local truncation error  $|\hat{y}^{(0)}_{m, 0} - y_{m, 0}|  \le C_m h^2$ with $m = 0,...,M$, for some constant $C_m$ independent of $H$. Therefore, $|\hat{y}^{(0)}_{m, 0} - y_{m, 0}| = \mathcal{O}(h^2)$ and 
by $\hat{z}^{(0)}_{m,0} =\mathcal{G}(\hat{y}^{(0)}_{m,0})$ and the Lipschitz condition of $\mathcal{G}$, it follows that $|\hat{z}^{(0)}_{m, 0} - z_{m, 0}| = \mathcal{O}(h^2)$ with $m = 0,...,M$. 

Now we prove  the local error estimate \eqref{eq: lemma2_local} and equation \eqref{eq: lemma2_g1} for the correction step $k=1$, assuming a fixed  $M\ge1$. By $g(\hat{y}^{(0)}_{m, 0},\hat{z}^{(0)}_{m, 0}) =0$ in the prediction step, from the second equation in \eqref{eq:Euler-0-k}, we obtain
$
g(\hat{y}^{(1)}_{m, 0},\hat{z}^{(1)}_{m, 0}) =0,
$
with $m = 0,...,M$, i.e. equation \eqref{eq: lemma2_g1} with $k = 1$. Then, from the condition (\ref{redg}) it follows  $\hat{z}^{(1)}_{m,0} = \mathcal{G}(y^{(1)}_{m, 0})$, and this gives from  \eqref{eq:Euler-0-k}
\begin{eqnarray}
\label{eq: lemma2_be}
\begin{array}{l}
\hat{y}^{(1)}_{m+1,0} = \hat{y}^{(1)}_{m,0} + h (\hat{f}(\hat{y}^{(1)}_{m+1,0})-
\hat{f}( \hat{y}^{(0)}_{m+1,0}))
 + h S^{m}(\bar{\hat{f}}^{(0)}_0), 
\end{array}
\end{eqnarray}
where $\hat{f}(\hat{y}^{(1)}_{m+1,0})  = f(\hat{y}^{(1)}_{m+1,0},\mathcal{G}(\hat{y}^{(1)}_{m+1,0}))$,  and
$
S^{m}(\bar{\hat{f}}^{(0)}_0) = S^{m}(\hat{f}(\bar{\hat{y}}^{(0)}_0, \mathcal{G}(\bar{\hat{y}}^{(0)}_0))).
$ 
The method \eqref{eq: lemma2_be} for updating $\hat{y}^{(1)}_{m,0}$ represents the first correction step of the InDC-BE method to solve the non-stiff ordinary differential equation \eqref{eqy}.
Therefore, from classical error estimates of InDC-BE method when applied to a non-stiff ordinary differential equation in \cite{christlieb2009integral}, we have $|y_{m,0} - \hat{y}^{(1)}_{m,0}| \le C_m h^3$ for some constant $C_m$ independent of $h$ with $h \le h_0$. Therefore $|y_{m,0} - \hat{y}^{(1)}_{m,0}| = \mathcal{O}(h^3)$ and
by $\hat{z}^{(1)}_{m,0} =  \mathcal{G}(\hat{y}^{(1)}_{m,0})$ and Lipschitz condition of $\mathcal{G}$, we get $|z_{m,0} - \hat{z}^{(1)}_{m,0}| = \mathcal{O}(h^3)$, $\forall m=1, \cdots M$ and $h \le h_0$. The estimate for general $k>1$ can be proved in a similar fashion and by mathematical induction with respect to $k$. 
$\Box$
\begin{lem}
\label{lemma3}
($\eps^1$ error term). Assume condition \eqref{eq: gz} holds and initial values of the differential algebraic system \eqref{eq: exact_eps_0}-\eqref{eq: exact_eps_1} are consistent. Consider the InDC-BE method constructed with $M$ uniformly distributed quadrature nodes excluding the left-most point, and with \eqref{eq:Euler-0}-\eqref{eq:Euler-1} for the prediction step and \eqref{eq:Euler-0-k}-\eqref{eq:Euler-1-k} for the correction step with $k=1, \cdots, K$ for solving the differential algebraic system \eqref{eq: exact_eps_0} -\eqref{eq: exact_eps_1}. 
Then the local error estimates of the InDC-BE method
\beq
\label{eq: L5}
e^{(k)}_{m,1} = y_{m,1} - \hat{y}^{(k)}_{m,1} = \mathcal{O}(h^2), \quad
d^{(k)}_{m,1} = z_{m,1} - \hat{z}^{(k)}_{m,1} = \mathcal{O}(h), 
\eeq
hold for $m=1, \cdots, M$ at the interior nodes of InDC, and for $k=0, \cdots, K$.
\end{lem}
\noindent
{\bf Proof.} The proof for the case of $k=0$ (prediction step) is a consequence of Lemma 4.4 in Chap. VII.4 in \cite{hairer1993solving2}. 
We then consider the first correction step with $k=1$ and assume a fixed  $M\ge1$.
We prove \eqref{eq: L5} by mathematical induction w.r.t. $m$. Especially, we know
$e^{(1)}_{m,1} = d^{(1)}_{m,1} = 0$, with $m=0$. We assume \eqref{eq: L5}
is valid for $0, \cdots, m$. We will prove that \eqref{eq: L5} is valid for $m+1$.
The integration of \eqref{eq: exact_eps_1} over
$[\tau_m, \tau_{m+1}]$ gives
\beq
\label{eq:exact-1}
\varepsilon^1: \quad
\left\{
\begin{array}{lll}
y_{m+1, 1} = y_{m, 1} + \int_{\tau_m}^{\tau_{m+1}} \mathbb{F}_1(\tau) d\tau,\\
z_{m+1, 0} = z_{m, 0} + \int_{\tau_m}^{\tau_{m+1}} \mathbb{G}_1(\tau) d\tau,\\
\end{array}
\right.
\eeq
with $\mathbb{F}_1$ and $\mathbb{G}_1$ defined in \eqref{eq: exact_eps_1}.
%%%%%%%%%%%%%%%
We consider now
\beq
\label{ed1}
e^{(1)}_{m+1,1} = y_{m+1,1}-\hat{y}^{(1)}_{m+1,1}, \quad d^{(1)}_{m+1,1} = z_{m+1,1}-\hat{z}^{(1)}_{m+1,1},
\eeq
i.e. the difference between the exact and numerical solution at $\tau_{m+1}$. From (\ref{eq: tempDF1}), as well as from the estimates \eqref{eq: lemma2_local} in Lemma~\ref{lemma1}, we have
\beqa \label{eq: Delta_F}
\Delta \hat{\mathbb{F}}^{(k-1)}_{m+1,1} =f_y \hat{e}^{(k-1)}_{m+1,1} + f_z 
\hat{d}^{(k-1)}_{m+1,1} + \mathcal{O}(h^{k+1}).
\eeqa
Here we used the abbreviations $ f_y = f_y (y_{m+1,0}, {z}_{m+1,0})$ and similarly for $f_z$. 
Equally, we have
\beq
\label{eq: Delta_G}
\Delta \hat{\mathbb{G}}^{(k-1)}_{m+1,1} 
=  g_y  \hat{e}^{(k-1)}_{m+1,1}   + g_z \hat{d}^{(k-1)}_{m+1,1} + \mathcal{O}(h^{k+1}). 
\eeq
Then from \eqref{eq: Delta_F} and \eqref{eq: Delta_G} for $k = 1$ it follows
\beq
\label{eq:new-error}
\left\{
\begin{array}{l}
\Delta \hat{\mathbb{F}}^{(0)}_{m+1,1}= \big( f_y\hat{e}^{(0)}_{m+1,1} + f_z\hat{d}^{(0)}_{m+1,1}) + \mathcal{O}(h^2),\\[3mm]
\Delta  \hat{\mathbb{G}}^{(0)}_{m+1,1} =
\big( g_y\hat{e}^{(0)}_{m+1,1} + 
g_z\hat{d}^{(0)}_{m+1,1}) + \mathcal{O}(h^2).
\end{array}
\right.
\eeq
Now subtracting equation \eqref{eq:Euler-1-k} from equation \eqref{eq:exact-1} this gives
\begin{eqnarray}
\label{eq:num-1}
\varepsilon^1: \quad 
\left\{
\begin{array}{l}
{e}^{(1)}_{m+1,1} = {e}^{(1)}_{m,1}  - h \Delta  \hat{\mathbb{F}}^{(0)}_{m+1,1}- h S^{m}(\bar{\hat{\mathbb{F}}}_1^{(0)})
+\int_{\tau_m}^{\tau_{m+1}} {\mathbb{F}}_1 (\tau)d\tau, \\[3mm]
{d}^{(1)}_{m+1,0} = {d}^{(1)}_{m,0} - h \Delta  \hat{\mathbb{G}}^{(0)}_{m+1,1}- h S^{m}(\bar{\hat{\mathbb{G}}}_1^{(0)})
+\int_{\tau_m}^{\tau_{m+1}} {\mathbb{G}}_1 (\tau) d\tau.
\end{array}
\right.
\end{eqnarray}
On the right-hand side of the equations in (\ref{eq:num-1}) we add and subtract the following quantities: $h S^{m}(\bar{\mathbb{F}}_1)$ and $h S^{m}(\bar{\mathbb{G}}_1)$,  these are the integrals of $(M-1)^{th}$ degree interpolating
polynomials on $(\tau_m, \mathbb{F}_1(\tau_m))_{m=1}^{M}$ and $(\tau_m, \mathbb{G}_1(\tau_m))_{m=1}^{M}$ over the subinterval $[\tau_m, \tau_{m+1}]$, hence they are accurate to the order $\mathcal{O}(h^{M+1})$ locally, i.e.  $\int_{\tau_m}^{\tau_{m+1}}   {\mathbb{F}}_1 (\tau) d\tau -h S^{m}( \bar{\mathbb{F}}_1) = \mathcal{O}(h^{M+1})$. By the  local error estimates in Lemma~\ref{lemma1}, as well as equation \eqref{eq: L5} for $k=0$, it follows that 
$S^{m}( \bar{\mathbb{F}}_1)-S^{m}(\bar{\hat{\mathbb{F}}}_1)$ and $S^{m}(\bar{\mathbb{G}}_1)-S^{m}(\bar{\hat{\mathbb{G}}}_1)$ are accurate to the order $\mathcal{O}(h)$. 
Thus, {from \eqref{eq:num-1} we get}
\begin{eqnarray}
\label{eq: new_error_2_y}
\left\{
\begin{array}{lll}
{e}^{(1)}_{m+1,1} &=& {e}^{(1)}_{m,1} -  h \left( f_y\hat{e}^{(0)}_{m+1,1} +  f_z\hat{d}^{(0)}_{m+1,1} \right)  + \mathcal{O}(h^2),\\[3mm]
\label{eq: new_error_2_z}
{d}^{(1)}_{m+1,0} &=& {d}^{(1)}_{m,0} - h \left( g_y\hat{e}^{(0)}_{m+1,1} +  g_z\hat{d}^{(0)}_{m+1,1} \right) 
+ \mathcal{O}(h^2).
\end{array}
\right.
\end{eqnarray}
Now from \eqref{eq: err_errhat} and \eqref{eq: lemma2_local}, we have
\begin{eqnarray}
\label{eq: eee}
\left\{
\begin{array}{l}
\hat{e}^{(0)}_{m,1} = \hat{y}^{(1)}_{m,1}-\hat{y}^{(0)}_{m,1} = e^{(0)}_{m,1}-e^{(1)}_{m,1} = -e^{(1)}_{m,1} + \mathcal{O}(h^2),\\[3mm]
\hat{d}^{(0)}_{m,1} = \hat{z}^{(1)}_{m,1}-\hat{z}^{(0)}_{m,1} = d^{(0)}_{m,1}-d^{(1)}_{m,1} = -d^{(1)}_{m,1} + \mathcal{O}(h),
\end{array}
\right.
\end{eqnarray}
and put it into equation \eqref{eq: new_error_2_y} gives,
\beq
\label{eq: new_error_3_y}
\left\{
\begin{array}{l}
{e}^{(1)}_{m+1,1} = {e}^{(1)}_{m,1} + h \left( f_y{e}^{(1)}_{m+1,1} +  f_z{d}^{(1)}_{m+1,1} \right) + \mathcal{O}(h^2),\\[3mm]
{d}^{(1)}_{m+1,0} = {d}^{(1)}_{m,0} + h \left( g_y{e}^{(1)}_{m+1,1} +  g_z{d}^{(1)}_{m+1,1} \right) + \mathcal{O}(h^2).
\end{array}
\right.
\eeq
Now using the estimate \eqref{eq: lemma2_local} about ${d}^{(1)}_{m,0}$, from the second equation in \eqref{eq: new_error_3_y} we obtain 
\beq
\label{eq: de1}
{d}^{(1)}_{m+1,1} =  -g^{-1}_z g_y{e}^{(1)}_{m+1,1}  + \mathcal{O}(h),
\eeq
with the invertibility of $g_z$. 
Inserting this into the first equation in \eqref{eq: new_error_3_y} gives
\begin{eqnarray}
\label{eq: new_error_4_y1}
e^{(1)}_{m+1,1} &=& (1- h (f_y - f_z g_z^{-1}g_y))^{-1}{e}^{(1)}_{m,1} + \mathcal{O}(h^2). %= \mathcal{O}(h^2).
\end{eqnarray}
Finally ${e}^{(1)}_{m+1,1} = \mathcal{O}(h^2)$  follows from \eqref{eq: new_error_4_y1}, and ${d}^{(1)}_{m+1,1} = \mathcal{O}(h)$ from \eqref{eq: de1}.
The estimate for general $k>1$ can be proved in a similar fashion and by mathematical induction with respect to $k$.
$\Box$

\begin{rem} \label{InDC_SA}
In \cite{christlieb2009comments}, the InDC method constructed with explicit RK methods in the prediction and correction steps has been reformulated as a high-order explicit RK method whose Butcher tableau is explicitly constructed. Similarly, the InDC-BE can be viewed as an IRK method with the corresponding Butcher tableau. Below we present the Butcher tableau for the InDC-BE method with one loop of correction step. 
This takes the form
\begin{align}\label{CBTab}
  \begin{array}{c|cc}
    \vec{c} & T & Z\\
    \vec{c} & P & T \\
    \hline
    & \vec{b}_1^T & \vec{b}_2^T 
  \end{array}
\end{align}
where 
$
  \vec{c} = \frac1M \left[1, \cdots, M \right]^T,
$
$Z$ is a $M\times M$ matrix of zeros,  
$T$ and $P$ are $M \times M$ matrices,
with
 \begin{align*}
   {T}= \frac1M \left[
  \begin{array}{ccccc}
     1 & 0 & 0&\ldots&0\\ 
     1 & 1 & 0& \ldots&0\\
    \vdots & \vdots & \ddots & \vdots & \vdots \\
     1 & 1 & 1 &\ldots&1
  \end{array}
  \right],
\end{align*}
\begin{align*}
  P = \left[
    \begin{array}{ccccc}
      (\tilde{S}_{11} - \frac1M) & \tilde{S}_{12} &  \ldots & \tilde{S}_{1,M-1} & \tilde{S}_{1,M}  \\
       (\tilde{S}_{21}- \frac1M) & (\tilde{S}_{22} - \frac1M) &  \ldots & \tilde{S}_{2,M-1} & \tilde{S}_{2,M} \\
      \vdots & \vdots & \ddots & \vdots & \vdots \\
       (\tilde{S}_{M,1} - \frac1M) & 
      (\tilde{S}_{M,2} -\frac1M) & \ldots 
      & (\tilde{S}_{M,M-1}-\frac1M) & (\tilde{S}_{M,M}-\frac1M)
    \end{array}\right],
\end{align*}
where the term
$\tilde{S}_{ij}=\int_{t_0}^{t_i} \alpha_j(s)ds$ with $\alpha_j(s)$ the Lagrangian basis functions for the node $\tau_j$, and the vector
$$
\vec{b}_1^T = \left(\left(\tilde{S}_{M,1} - \frac{1}{M}\right), \left(\tilde{S}_{M,2} - \frac1M\right), \cdots, \left(\tilde{S}_{M,M} - \frac1M)\right) \right), \quad \vec{b}_2^T = \frac{1}{M}(1,1,\cdots, 1).$$
%as defined in equation \eqref{eq: lag_basis}. %Adopting a Matlab-style notation 
\end{rem}
Now  from remark \ref{InDC_SA} the following proposition follows. 
\begin{prop} 
\label{prop: idc_be_sa}
The InDC-BE method with $K$ correction steps is an implicit stiffly accurate IRK method with an invertible matrix $A$ in the Butcher tableau \eqref{eq: B_table}. Especially when $K=1$,  we get
% given by
\begin{eqnarray}\label{matrixIDC_A}
A = 
\left( \begin{array}{cc}
T & Z\\
 P & T
 \end{array}\right).
 \end{eqnarray}
%\QQ{$A$ is  invertible.}   
\end{prop}
%\QQ{
%The Proposition can be generalized to InDC-BE method with $K$ correction steps directly.
%} 
%\begin{prop} 
%\label{prop: idc_be_sa}
%The InDC-BE method with $K$ correction step is an implicit stiffly accurate IRK method with an invertible matrix $A$.
%\end{prop}
%}
%\SB{HERE: ADD STABILITY ANALYSIS FOR IDC METHOD WITH IMPLICIT EULER ???? and numerically show that this method has the better stability region????  }
\begin{rem}
\label{rem:bn}
In the estimates in Lemma \ref{lemma3}, we show that
there is no improvement for $e^{(k)}_{m, 1}$ and $d^{(k)}_{m, 1}$ as $k$ increases, see equation \eqref{eq: L5}. %for the order of convergence  $y_1$ and $z_1$ in InDC corrections. 
This is consistent with our numerical evidences presented in the previous section.
The reason is that {\em both} the local and global error for the $z$-component %$\hat{z}_1^{(k)}$ approximating $z_1$ 
in the prediction and correction steps is of first order. 
This sets the bottleneck for the order increase % in the term $\mathcal{O}(h^2)$
in the second equation of \eqref{eq: new_error_2_y}.
\end{rem}
 
We are now in the position to prove Theorem~\ref{thm: IDC_BE} by the local error estimates of the two lemmas above.
%}\\%we obtain the global error estimates of  Theorem~\ref{thm: IDC_BE} }.
%\\
%REMOVE ALL THIS:\\
%-----------------------------------------\\
%The Proposition below brings the local error estimates from four Lemmas above to a global error estimates.% for terms in equation \eqref{ED}.
%\textcolor{red}{
%\begin{prop}
%\label{prop: global_IDC_BE}
%Under the same assumption of Theorem~\ref{thm: IDC_BE}, then for any fixed constant $c > 0$, the following global error estimates %for equation \eqref{ED} 
%holds  with
%\[
%e^{(K)}_{n,0} = \mathcal{O}(H^{\min(K+1,{M})}) + \mathcal{O}(H), \quad d^{(K)}_{n,0} = \mathcal{O}(H^{\min(K+1,{M})}) + \mathcal{O}(H)
%\]
%%\[
%%e^{(K)}_{n,1}  =\mathcal{O}(H), \quad d^{(K)}_{n,1}  =\mathcal{O}(H),
%%\]
%for $\eps\le cH$. % and  $\nu \le 2$. 
%The estimates hold uniformly for $H\le H_0$ and $nH \le Const$.
%\end{prop}
%}
%----------------------------------------\\
\noindent
{\em Proof of Theorem~\ref{thm: IDC_BE}.} 
Our first step here is to estimate $e^{(K)}_{n,0}$ and $d^{(K)}_{n,0}$. For this, from Lemma~\ref{lemma1} we have after one step from $t_0$ to $t_1$, the local error estimate %give the local estimate,
\beq
y_0(t_1) - \hat{y}^{(K)}_{M,0} = \mathcal{O}(H^{\min(K+2,{M+1})}),
\eeq
with $m = M$ and $\tau_M = t_1$ in equation \eqref{eq: lemma2_local}.
{In the estimate of the global error from local error,  we obtain} %the global error estimate 
\[
e^{(K)}_{n,0} = y_0(nH) - \hat{y}^{(K)}_{n, 0} =\mathcal{O}(H^{\min(K+1,{M})}).
\]
It thus follows from (\ref{eq: lemma2_g1}), and by the Lipschitz condition of $\mathcal{G}$, that
\[
d^{(K)}_{n,0} = z_0(nH) - \hat{z}^{(K)}_{n, 0} =\mathcal{O}(H^{\min(K+1,{M})}).
\]
Now our next aim is to estimate $e^{(K)}_{n,1}$  and $d^{(K)}_{n,1} $. From Lemma~\ref{lemma3}, we have for the local error estimate
\beq
y_1(t_1) - \hat{y}^{(k)}_{M,1} = \mathcal{O}(H^2).
\eeq
By Lemma~\ref{lemma3}, the proof of the global error estimates for $y$ and $z$  is similar to that of Theorem 4.5 and 4.6 in Chap.~VII.4 of \cite{hairer1993solving2}. Thus we obtain
\[
e^{(K)}_{n,1} = y_1(nH) - \hat{y}^{(k)}_{n, 1} =\mathcal{O}(H), \quad
d^{(K)}_{n,1} = z_1(nH) - \hat{z}^{(k)}_{n, 1} =\mathcal{O}(H),
\]
which proves the statement.
$\Box$

\begin{rem}
Similar error estimates can be given for the InDC SA-IRK method. We present and prove these error estimates in Appendix. 
\end{rem}

\section{Stability properties}
\label{sec: stab}

One important aspect of stability of numerical integrators can be visualized by the stability region \cite{hairer1993solving2} in a complex plane around the origin. 
For implicit methods discussed in Section~\ref{sec3}, we plot their stability regions in Figure~\ref{fig_stab}. In these plots, the region outside the bounded domains are the stability regions. Note that in these methods, the left-most quadrature point is always excluded in the construction of the InDC method. 
The following observations can be made.

\begin{enumerate}
\item In general, as more InDC corrections are performed, the stability region shrinks.    
\item The InDC-BE method appears to be A-stable with $M$ quadrature nodes and with up to $M-1$ correction loops for $M=4$ and $M=6$. For the other $M$'s, it appears that such conclusion still holds in our tests for $M\le8$. Note that such methods can be viewed as diagonally implicit RK methods. 
\item The InDC-DIRK2-SA method appears to be $A(\alpha)$-stable with one and two correction loops, with the angle $\alpha$ decreasing as more correction loops are taken. 
\item Among three different implicit RK methods, the InDC method constructed with the BE method appears to have the largest stability region. 

\end{enumerate}

\begin{figure}[htb]
\begin{center}
\includegraphics[width=2.5in]{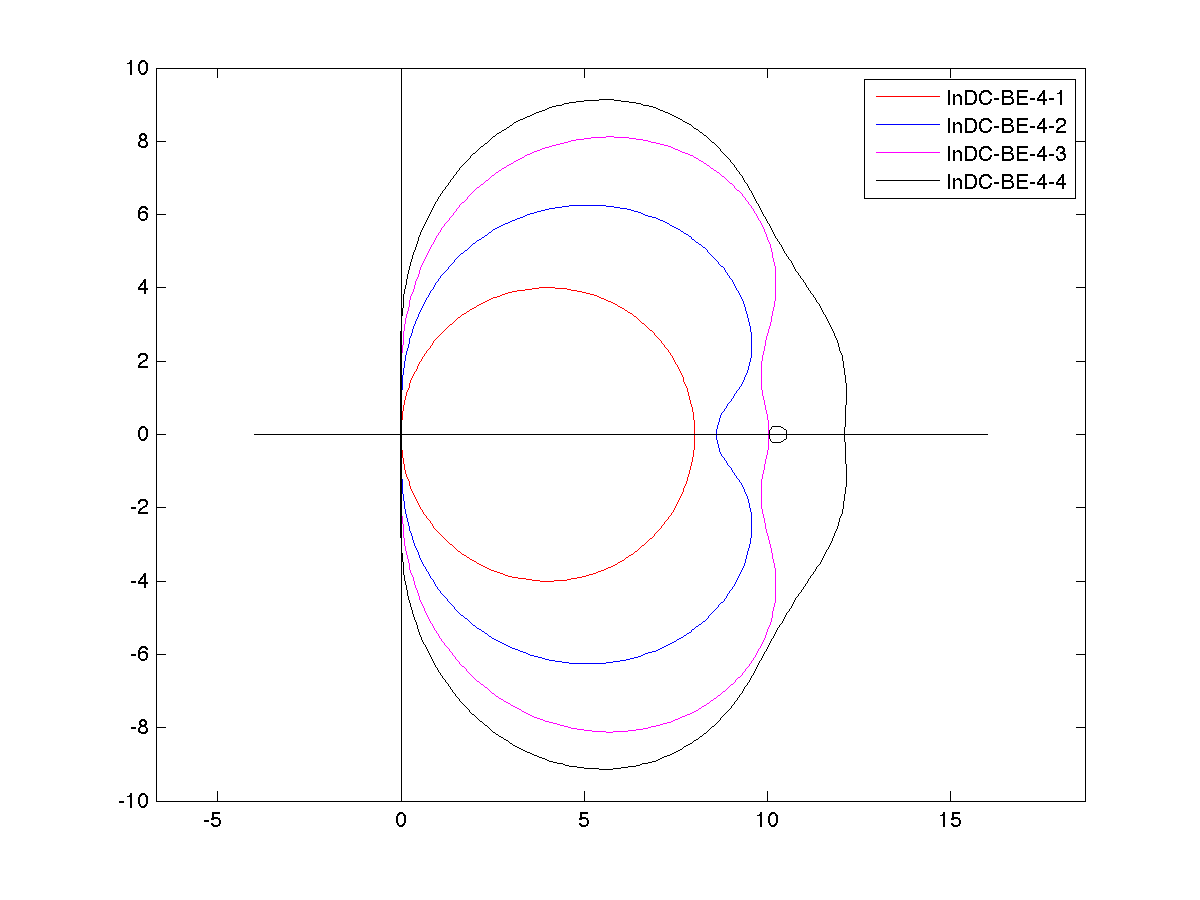}
\includegraphics[width=2.5in]{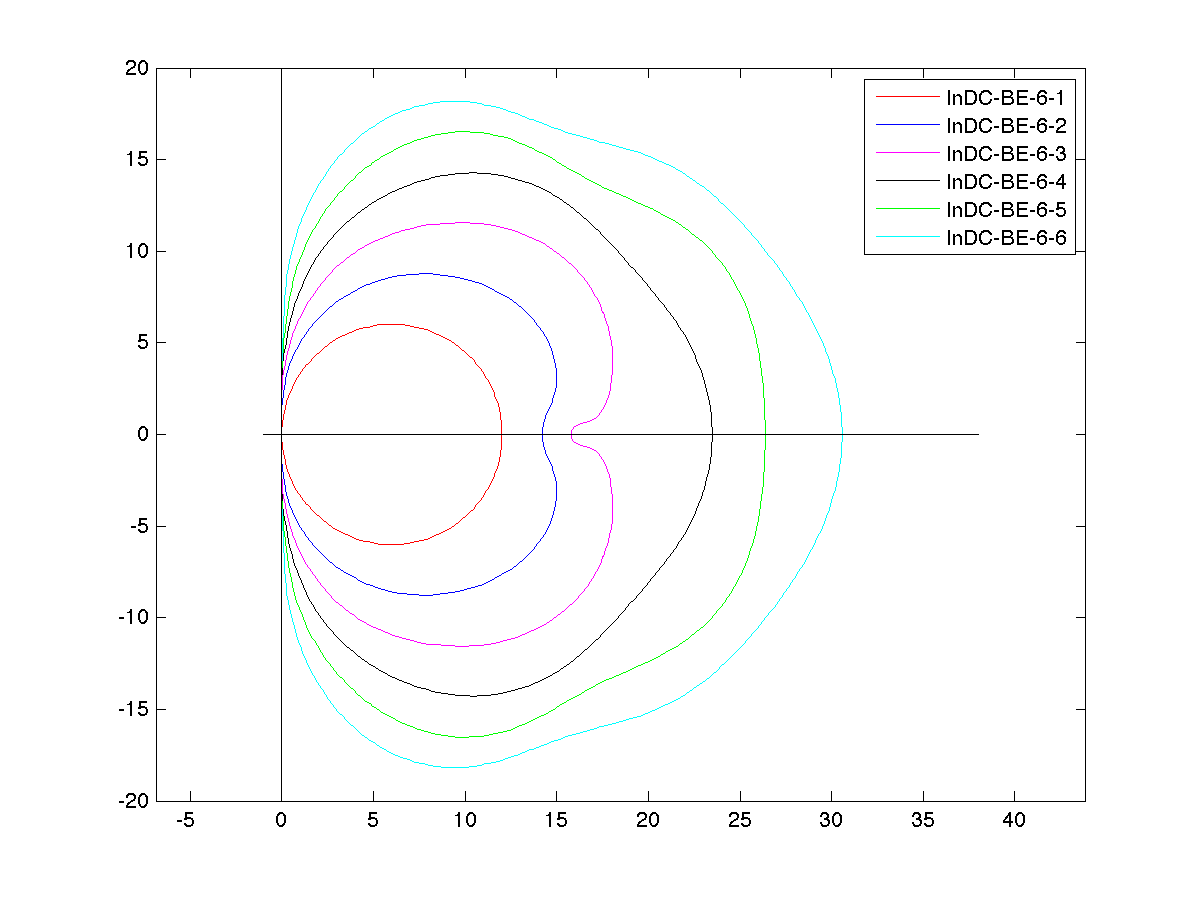}\\
\includegraphics[width=2.5in]{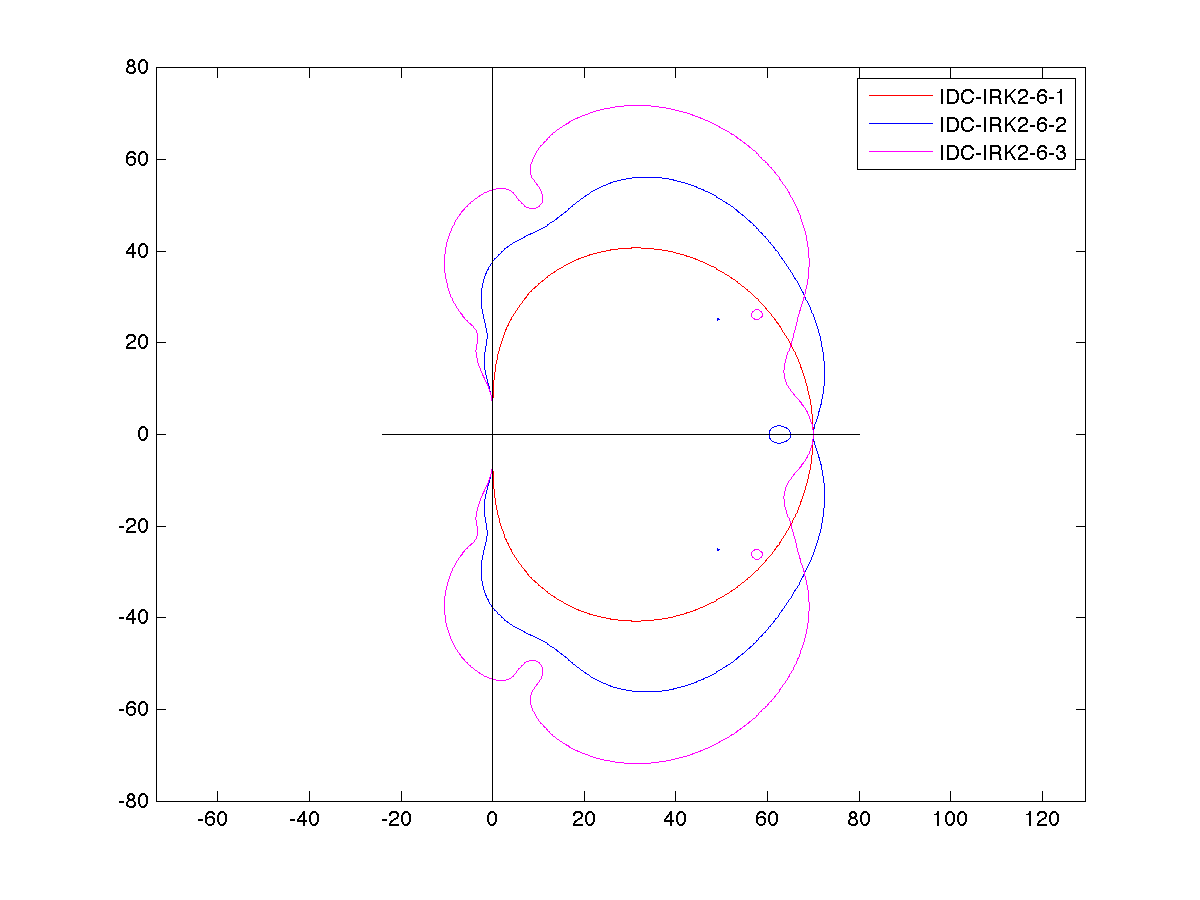}
\includegraphics[width=2.5in]{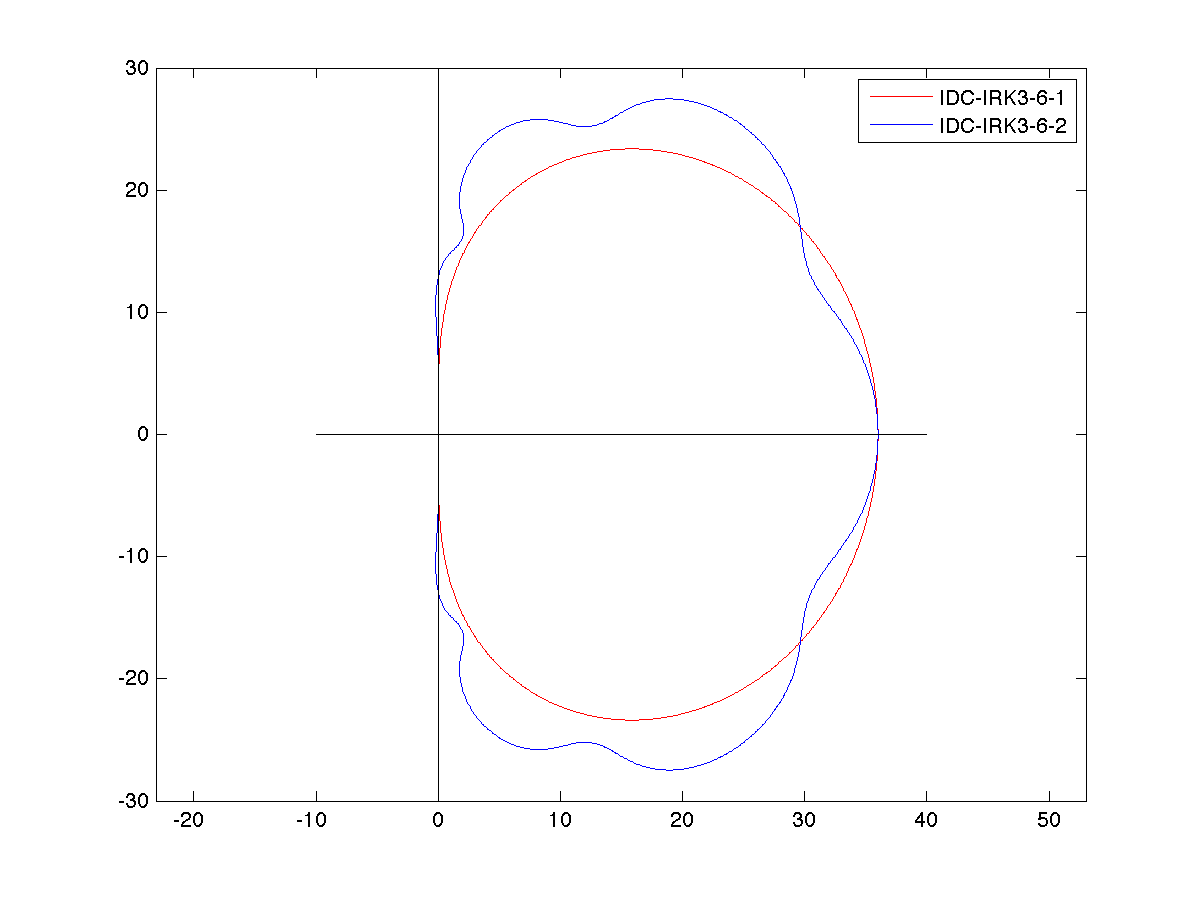}
\end{center}
\caption{Stability regions for InDC-BE method with 4 quadrature points (upper left) and 6 quadrature points (upper right) with various correction loops as indicated in the legend. 
Lower left plot shows the stability region for the InDC-DIRK2-SA method with 6 quadrature points with zero, one and two correction loops as indicated in the legend. 
Lower right plot shows the stability region  for the InDC-Radau method with 6 quadrature points with zero and one correction loops as indicated in the legend. 
}
\label{fig_stab}
\end{figure}

{Now we prove the following proposition. We notice that a similar result on the InDC method using the first order BE scheme was established earlier in \cite{layton2005implications}.}
%%%%%%%%%%%%%%%
%\SB{DELETE. Another important aspect of the stability is the $L$-stability, which concerns the behavior of the stability function $R(z)$ as $z\rightarrow \infty$. Such stability is important for stiff problems such as the SPP discussed in this paper to ensure the correct asymptotic behavior. 
%In the following, we prove the $L$-stability of the InDC method constructed with  stiffly accurate IRK methods when $M$ quadrature points (but excluding the left-most points) are used, if the method is A-stable. Similar result on the InDC method using the first order BE scheme was established earlier in \cite{layton2005implications}.}
%%%%%%%%%%%%
\begin{prop}
Let $\mathcal{R}(z)$ be the stability function of the InDC method constructed by a stiffly accurate IRK method {with nonsingular matrix $A$ in the corresponding Butcher table}.  We assume $M$ uniform quadrature points, but excluding the left-most point, are used. Then $\lim_{|z|\rightarrow\infty}\mathcal{R}(z) = 0$. That is, the method is $L$-stable, if $A$-stable.
\end{prop}
\begin{proof} 
{We consider the linear scalar problem $y' = \lambda y$ with $z = \lambda h$ and $y(t=0)=1$.} 
{For a stiffly accurate IRK method with nonsingular matrix $A$ the numerical solution for this linear scalar problem is equal to the last internal stage and the corresponding stability function is
\beq
\label{eq: stab_r}
R(z) = {\bf e}^T (I - A z)^{-1} \mathbf{1}
%R(z) = 1 + z b^T(I-z A)^{-1} \mathbf{1} = {\bf e}^T (I - A z)^{-1} \mathbf{1} 
%\ \mbox{\QQ{check about this!}}
\eeq
with ${\bf e} = (0, \cdots, 0, 1)^T$ and $\mathbf{1} = (1, 1, \cdots, 1)^T$ %\QQ{In \eqref{eq: stab_r}}%, the second equality is due the the fact that the IRK method is stiffly accurate.}
such that for $z \to \infty$, we get $R(\infty) = 0$, \cite{hairer1993solving2}.
} %Proposition 3.8, Chap. IV. in 
%\SB{Delete:From Proposition 3.8 in Chap. IV. in \cite{}, $R(\infty)= 1- b^T A^{-1} \mathbf{1} = 0$, where the second equality is due to its stiffly accurate property $A^T {\bf e} = b^T$}. 
Let $\mathcal{R}^{(k)}_m (z)$ be the amplification factor of the InDC method in the $k$-th iteration at the $m$-th quadrature point. 
In the prediction step, 
\[
\mathcal{R}^{(0)}_m (z) = \left(R \left(\frac zM \right) \right)^m.
% \prod_{i=1}^m r_i(z/M).
\]
%\SB{with $\tilde{z} = \frac zM$}.
Hence $\lim_{|z|\rightarrow\infty}\mathcal{R}^{(0)}_m (z)  = 0$, for $m=1, \cdots, M$, as for the IRK method $\lim_{|z|\rightarrow\infty}{R} \left(\frac z M\right)  = 0$. Let $ \vec{\mathcal{R}}^{(0)} = (\mathcal{R}^{(0)}_1, \cdots, \mathcal{R}^{(0)}_M)$, then  $\lim_{|z|\rightarrow\infty} \vec{\mathcal{R}}^{(0)}(z) = \mathbf{0}$, where $\mathbf{0}$ is a zero vector. {Note that this is true only if the left-most point is excluded.}
 
In the first correction step, on the first subinterval $[0, \Delta t/M]$, the amplification factor of the updated solution at the IRK intermediate stages can be represented as a vector $\vec{r}^{(1)}_1$ with the length of the vector being $s$, the stage number of the IRK method. Then, {with the help of Butcher table notation in Remark~\ref{InDC_SA},}
\beq
\label{eq: stab}
\vec{r}^{(1)}_1(z) = \mathbf{1}  + \frac{z}{M} A (\vec{r}^{(1)}_1 - P \vec{\mathcal{R}}^{(0)}) + z S \vec{\mathcal{R}}^{(0)} =  \mathbf{1}  + \frac{z}{M} A \vec{r}^{(1)}_1+ z(-\frac{A P}{M}  + S) \vec{\mathcal{R}}^{(0)}, \quad \mbox{with} \ {\bf 1} = (1, \cdots,  1, 1)'.
\eeq
Here we let $P$ and $S$ are interpolation and integration matrices of size $s \times M$; they are coefficients that maps the $M$ function values at quadrature nodes to approximate function values at $s$ IRK intermediate stages ($c_i/M$, $i=1, \cdots, s$) and over $s$ integrals ($[0, c_i/M]$, $i=1, \cdots, s$). 
Hence,
\[
\vec{r}^{(1)}_1(z) =(I-\frac{z}{M} A)^{-1} \mathbf{1} +(I-\frac{z}{M} A)^{-1}  z(-\frac{A P}{M}  + S) \vec{{\mathcal R}}^{(0)}.
\]
Since the {IRK} method is stiffly accurate, then
{
\beqa
{\mathcal{R}}^{(1)}_1(z)  &=& {\bf e}^T \cdot \vec{r}^{(1)}_1(z), \qquad \mbox{with} \ {\bf e} = (0, \cdots,  0, 1)^T\notag\\
&=& {\bf e}^T (I-\frac{z}{M} A)^{-1} \mathbf{1}  + {\bf e}^T \cdot (I-\frac{z}{M} A)^{-1}  z(-\frac{A P}{M}  + S) \vec{{\mathcal R}}^{(0)}\notag \\
&\stackrel{\eqref{eq: stab_r}}{=}&R(z_i) + {\bf e}^T \cdot (I-\frac{z}{M} A)^{-1}  z(-\frac{A P}{M}  + S) \vec{{\mathcal R}}^{(0)}. 
\eeqa
}
Hence
%Recall that the stability function of a stiffly accurate IRK method can be expressed in terms of the $A$ matrix and $b$, $c$ vectors in a Butcher table \eqref{} as,
%\beq
%\label{eq: stab_2}
%R(z) = {\bf e} (I- z A)^{-1} \mathbf{1}, \qquad \mbox{with} \ {\bf e} = (0, \cdots,  0, 1) \ \mbox{and} \ \mathbf{1} = (1, \cdots, 1),
%\eeq
%and $\lim_{|z|\rightarrow\infty}R(z) = 0$ for a stiffly accurate IRK method. 
%Combining \eqref{eq: stab} and \eqref{eq: stab_2}
%$R^{(1)}_1(z) = R(z/M) + z {\bf e} (- \frac1M A*I + S) \vec{R}_0$, from which one can get 
\[
\lim_{|z|\rightarrow\infty} {\mathcal R}^{(1)}_1(z) = \lim_{|z_i|\rightarrow\infty} R(z_i) + {\bf e}\lim_{|z|\rightarrow\infty}  \left((I-\frac {z}{M}A)^{-1} z \right) (- \frac{A P}M + S)  \lim_{|z|\rightarrow\infty}  \vec{{\mathcal R}}_0 =0,
\]
since $\lim_{|z|\rightarrow\infty} \vec{{\mathcal R}}_0 = 0$ from the prediction step. {Similar procedure could be repeated for other subintervals by a mathematical induction argument with respect to the $m$, from which we have $\lim_{|z|\rightarrow\infty}{\mathcal R}^{(1)}_m (z)  = 0$, for $m=1, \cdots M$. Specifically, $\lim_{|z|\rightarrow\infty}{\mathcal R}^{(1)}_M (z)  = 0$ after the first correction loop. }

{The same conclusion holds for the future correction steps by the mathematical induction argument with respect to the correction loop $k$, i.e. $\lim_{|z|\rightarrow\infty} {\mathcal {R}}_m^{(k)}(z) = 0$, $m=1, \cdots M$.
} 
\end{proof}

\begin{rem}
The above result can be generalized to the case when quadrature nodes are not uniformly distributed. On the other hand, the assumption to exclude the left-most point is necessary to guarantee the $L$-stability. Specifically, if the left-most point is included, then $\lim_{|z|\rightarrow\infty} \vec{{\mathcal R}}_0 \neq 0$, due to its first component. 
\end{rem}
\begin{rem}
From the stability plots in Figure~\ref{fig_stab}, the InDC-BE methods are L-stable when $M\le 8$ and for the number of iterations $k\le M$. 
\end{rem}

\section{Conclusions}
\label{sec5}
\setcounter{equation}{0}
\setcounter{figure}{0}
\setcounter{table}{0}
This paper studies the order of convergence of the InDC-BE and InDC-IRK methods when applied to SSPs, using uniform distribution of quadrature points excluding the leftmost point. We applied the technique of asymptotic expansion in powers of $\eps$ for the smooth exact solution and for the  corresponding numerical solution presented in \cite{hairer1988error, hairer1993solving2}. Two Theorems on global error estimate in the form of $\varepsilon$-expansion are presented and proved. Especially, we point out that the InDC methods improve the order of the $\eps$-independent error, but there is no order improvement on the higher order terms $\eps^\nu$ ($\nu\ge1$).
Such asymptotic analysis enables us to understand the phenomenon of order reduction for InDC methods when applied to stiff problems.  
{A solution in order to solve this problem is not a trivial matter. In fact, as mentioned in Remarks \ref{rem:bn} and \ref{notY}, the bottleneck is the order reduction phenomenon in the prediction step. Further deep studies are required. It is an interesting topic for future investigation but it is beyond our scope in this paper.}
Numerical results on van der Pol equations confirm these convergence results. 
%\SB{I think to remove this sentence: "In the future, we will study the global error and stability property of InDC framework constructed with high order implicit-explicit (IMEX) RK methods."}
%\SB{Especially, there is no order improvement for the $\eps^\nu$ ($\nu\ge1$) error terms in InDC corrections, see Remark~\ref{rem:bn} and ~\ref{notY}.

%\QQ{suggest to delete: these are true statements, but relatively negative. Readers can draw conclusions themselves from the analysis.  "This suggest us that no advantage we can obtain using InDC framework involving BE or SA-IRK method in the correction step with respect to a classical IRK method as DIRK or Radau IIA method. Then an InDC approach is not recommended for stiff problems, it shows no good performance when compared with the classical results presented in \cite{hairer1993solving2} and \cite{hairer1988error}. However, if we consider non stiff problems InDC approach is an optimal tool because starting with a lower order method the accuracy of the solutions from the correction steps continuously improves."} 
  
%\input{appendix}
%\input{appendix_081915}
\section{Appendix.}
\setcounter{equation}{0}
\setcounter{figure}{0}
\setcounter{table}{0}
\label{sec: InDC-IRK}

In the appendix, we extend the error estimates of the InDC-BE method to InDC-IRK method when applied to SPPs. 
We first describe the InDC-IRK method applied to (\ref{spp}),  then perform an $\eps$-expansion of the numerical solution of this method, and finally we  prove Theorem~\ref{thm: IDC_RK}.

\subsection{InDC-IRK method}

%Below, we describe InDC-IRK method applied to (\ref{spp}). 
We consider the InDC-IRK method constructed with $s$-stage IRK methods, where $A$ matrices in the Butcher tableau \eqref{eq: B_table} are invertible.
For the internal stages in the IRK method, we introduce the integration matrix and interpolation matrix as following
\beq
\label{eq: s_cmi}
h S^{c_{mi}, k} =  \int_{\tau_m}^{\tau_m+c_{mi} h}\alpha_k(s) d s, 
\quad
P^{c_{mi}, k} = \alpha_k(\tau_m+c_{mi} h),
\eeq
$\forall m=0, \cdots, M-1, \quad \forall k=1, \cdots M$ and  $\forall mi = 1, \cdots s$,
where $mi$ is index used for the $i^{th}$-stage of the IRK method over the subinterval $[\tau_m, \tau_{m+1}]$.
Here $\alpha_k(s)$ is the Lagrangian basis function based on the node $\tau_k$. 
Let 
\[
S^{c_{mi}}(\bar{f}) = \sum_{j = 1}^{M} S^{c_{mi}, j} f(y_j,z_j), \quad P^{c_{mi}}(\bar{f}) = \sum_{j = 1}^{M} P^{c_{mi}, j} f(y_j,z_j),
\]
{then we have}
\beq
\label{eq: spect_inte_accu}
hS^{c_{mi}}(\bar{f})-\int_{\tau_m}^{\tau_m+c_i h}f(y(s), z(s))ds = \mathcal{O}(h^{M+1}),
\eeq
\beq
\label{eq: spect_inter_accu}
P^{c_{mi}}(\bar{f})-f(y(\tau_m+c_i h), z(\tau_m+c_i h)) = \mathcal{O}(h^{M}),
\eeq
for any smooth function $f$. In other words, the quadrature formula given by $hS^{c_{mi}}(\bar{f})$ approximates the exact integration with $(M+1)^{th}$ order accuracy locally, while the interpolation formula given by  $P^{c_{mi}}(\bar{f})$ approximates the exact solution at RK internal stages with $M^{th}$ order accuracy locally.

To compute the numerical error approximating the error function $e^{(k-1)}(\tau_m)$, $d^{(k-1)}(\tau_m)$ with a general IRK method to (\ref{defint}),  {we obtain}
\beq
\label{princNum}
\left(\begin{array}{c}
\hat{e}_{m+1}^{(k-1)}\\[2mm]
 \eps \hat{d}_{m+1}^{(k-1)}
\end{array}\right)
 = \left(\begin{array}{c}
\hat{e}_{m}^{(k-1)} + h\int_{0}^{1}\delta(\tau_m + \tau h)d\tau\\[2mm]%\sigma(y^{[j-1]}_m)\\
\eps \hat{d}_{m}^{(k-1)} + h \int_{0}^{1}\rho(\tau_m + \tau h)d\tau
\end{array}\right) 
 + h \sum_{i=1}^{s} b_i \left(\begin{array}{c}
                          \Delta \hat{\mathcal{K}}^{(k-1)}_{mi}\\[2mm]
                            \Delta \hat{\mathcal{L}}^{(k-1)}_{mi}
                          \end{array}\right),
\eeq
and \beq\label{inter_stag}
\left(\begin{array}{c}
\hat{E}_{mi}^{(k-1)}\\[2mm]
\eps \hat{D}_{mi}^{(k-1)}
\end{array}\right)
 = \left(\begin{array}{c}
\hat{e}_{m}^{(k-1)} + h \int_{0}^{c_{mi}}{\delta}(\tau_m + \tau h)d\tau\\[2mm]
\eps \hat{d}_{m}^{(k-1)} + h \int_{0}^{c_{mi}}{\rho}(\tau_m + \tau h)d\tau
\end{array}\right)
 + h  \sum_{j=1}^{s} a_{ij}  \left(\begin{array}{c}
                         \Delta \hat{\mathcal{K}}^{({k-1})}_{mj}\\[2mm]
                      \Delta \hat{\mathcal{L}}^{({k-1})}_{mj}
                          \end{array}\right),
\eeq
with
\beqa 
\label{EvF}
\left(\begin{array}{c}
\Delta \hat{\mathcal{K}}^{({k-1})}_{mi}\\[2mm]
\Delta \hat{\mathcal{L}}^{(k-1)}_{mi}
\end{array}\right)
& \doteq& \left(\begin{array}{c}
f(\hat{Y}^{(k)}_{mi} , \hat{Z}^{(k)}_{mi})-P^{c_{mi}}(\bar{\hat{f}}^{(k-1)}) \\[2mm]
g(\hat{Y}^{(k)}_{mi}, \hat{Z}^{(k)}_{mi})-P^{c_{mi}}(\bar{\hat{g}}^{(k-1)})
\end{array}\right) \\[2mm]
\label{EvF2}
&=&\left(\begin{array}{c}
f(\hat{Y}^{(k)}_{mi} , \hat{Z}^{(k)}_{mi})-f(P^{c_{mi}} (\bar{\hat{y}}^{(k-1)}),  P^{c_{mi}} (\bar{\hat{z}}^{(k-1)})) \\[2mm]
g(\hat{Y}^{(k)}_{mi}, \hat{Z}^{(k)}_{mi})-g(P^{c_{mi}} (\bar{\hat{y}}^{(k-1)}), P^{c_{mi}}(\bar{\hat{z}}^{(k-1)})) 
\end{array}\right) + \mathcal{O}(h^{M}), 
\eeqa
where we put
\beq
\label{Newk}
\hat{Y}^{(k)}_{mi} =  P^{c_{mi}} (\bar{\hat{y}}^{(k-1)})+ \hat{E}^{(k-1)}_{mi}, \quad \hat{Z}^{(k)}_{mi} = P^{c_{mi}} (\bar{\hat{z}}^{(k-1)}) + \hat{D}^{(k-1)}_{mi},
\eeq 
and equation \eqref{EvF2} is due to the high order interpolation accuracy of $P^{c_{mi}}$, see equation \eqref{eq: spect_inter_accu}.
We can rewrite the system \eqref{princNum} and \eqref{inter_stag} as
\beq\label{newK}
\left(\begin{array}{c}
\hat{y}_{m+1}^{(k)} -  h S^{m,(k-1)}_{\bar{\hat{{f}}}}\\[2mm]
\eps\hat{z}_{m+1}^{(k)} - h S^{m,(k-1)}_{\bar{\hat{{g}}}}
\end{array}\right)
 = \left(\begin{array}{c}
\hat{y}_{m}^{(k)}\\
\eps\hat{z}_{m}^{(k)} 
\end{array}\right) 
 + h \sum_{i=1}^{s} b_i \left(\begin{array}{c}
                          \Delta \hat{\mathcal{K}}^{({k-1})}_{mi}\\[2mm]
                           \Delta \hat{\mathcal{L}}^{({k-1})}_{mi}
                          \end{array}\right), 
\eeq
\beq\label{newapproach}
\left(\begin{array}{c}
\hat{Y}_{mi}^{(k)}- hS^{c_{mi},(k-1)}_{\bar{\hat{f}}}\\[2mm]
\eps \hat{Z}_{mi}^{(k)}-  hS^{c_{mi},(k-1)}_{\bar{\hat{g}}}
\end{array}\right)
 = \left(\begin{array}{c}
\hat{y}_{m}^{(k)} \\
\eps \hat{z}_{m}^{(k)} 
\end{array}\right)
 + h  \sum_{j=1}^{s} a_{ij}  \left(\begin{array}{c}
                         \Delta \hat{\mathcal{K}}^{({k-1})}_{mj}\\[2mm]
                        \Delta \hat{\mathcal{L}}^{({k-1})}_{mj}
                          \end{array}\right),
\eeq
with
\beq \label{newS}
\left(\begin{array}{c}
  S^{m,(k-1)}_{\bar{\hat{f}}} = S^m(\bar{\hat{f}}^{(k-1)})\\[2mm]
 S^{m,(k-1)}_{\bar{\hat{g}}} =  S^m(\bar{\hat{g}}^{(k-1)})
                            \end{array}\right),\quad
\left(\begin{array}{c}
  S^{c_{mi},(k-1)}_{\bar{\hat{f}}} = S^{c_{mi}}(\bar{\hat{f}}^{(k-1)})\\[2mm]
 S^{c_{mi},(k-1)}_{\bar{\hat{g}}} =  S^{c_{mi}}(\bar{\hat{g}}^{(k-1)})
                            \end{array}\right).                           
\eeq
\begin{rem}
\label{rem: sa}
Under the assumption $A$ invertible,
%Assuming that $A$ is invertible, following a similar procedure as in equation \eqref{InvA} and \eqref{Solz}, 
from the second equation of \eqref{newapproach} we obtain in vectorial form 
\[
 h \Delta \bar{\hat{\mathcal{L}}}^{(k-1)} = A^{-1} (\eps\bar{\hat{Z}}^{(k)}- \eps\hat{z}^{(k)}_{m} \mathbf{1} - h S^{\bar{c}}(\bar{\hat{g}}^{(k-1)}) ),
\]
with $\Delta \bar{\hat{\mathcal{L}}}^{(k-1)} = (\Delta \hat{\mathcal{L}}^{({k-1})}_{m1}, \cdots, \Delta \hat{\mathcal{L}}^{({k-1})}_{ms})^T$, $\mathbf{1} = (1, \cdots, 1)^T$ and $\bar{c} =  (c_{m1}, \cdots, c_{ms})$. % are RK internal stages.
Inserting this into the second equation of \eqref{newK}, we get
\begin{eqnarray}\label{eq: z_linear}
\begin{array}{lll}
\eps\hat{z}^{(k)}_{m+1} = \eps R(\infty)\hat{z}^{(k)}_{m} + \eps b^T A^{-1} \bar{\hat{Z}}^{(k)} + h (S^m(\bar{\hat{g}}^{(k-1)}) - b^T A^{-1}S^{\bar{c}}(\bar{\hat{g}}^{(k-1)}) ).\\
\end{array}
\end{eqnarray}
Of special importance now are stiffly accurate RK methods, i.e., methods which satisfy $b^T A^{-1} = e_s^T$. This implies $R(\infty) = 0$ and
$  b^T A^{-1} S^{\bar{c}}(\bar{\hat{g}}^{(k-1)}) = e_s^T S^{\bar{c}}(\bar{\hat{g}}^{(k-1)}) =S^m(\bar{\hat{g}}^{(k-1)}) $. Hence by \eqref{eq: z_linear} we have: 
%\beq
%\label{eq: z_linear}
$\hat{z}^{(k)}_{m+1}  = \hat{Z}^{(k)}_{ms}$.%R(\infty)\hat{z}^{(k)}_{m} + b^T A^{-1}\bar{\hat{Z}}^{(k)} = \hat{Z}^{(k)}_{ms},
%\end{array}
%$\eeq
%We remark that equation \eqref{eq: z_linear} is in a similar spirit to \eqref{Solz2} for IRK method.
\end{rem}

\subsection{$\eps$-asymptotic expansion of InDC-IRK methods}% embedded with IRK methods.}
\label{sec: InDC-IRK-expand}

We formally expand the quantities $\Delta \hat{\mathcal{K}}^{(k-1)}_{mi}, 
\Delta\hat{\mathcal{L}}^{(k-1)}_{mi}$ from (\ref{EvF}) and $\hat{Y}_{mi}^{(k)}, \ 
\hat{Z}_{mi}^{(k)},\ \hat{y}_{m+1}^{(k)}, \ \hat{z}_{m+1}^{(k)}$ from \eqref{Newk} and \eqref{newK}
into powers of $\varepsilon$ with $\varepsilon$-independent coefficients
\beq
\label{ExpNum}
\begin{array}{l}
\hat{y}_{m}^{(k)} = \hat{y}_{m,0}^{(k)} + \varepsilon \hat{y}_{m,1}^{(k)} + \varepsilon^2 \hat{y}_{m,2}^{(k)} + \cdots,\\ [2mm]
\hat{Y}_{mi}^{(k)} = \hat{Y}_{mi,0}^{(k)} + \varepsilon \hat{Y}_{mi,1}^{(k)} + \varepsilon^2 \hat{Y}_{mi,2}^{(k)} + \cdots,\\  [2mm]
\Delta\hat{\mathcal{K}}^{(k-1)}_{mi}=\Delta\hat{\mathcal{K}}^{(k-1)}_{mi,0} + \varepsilon \Delta\hat{\mathcal{K}}^{(k-1)}_{mi,1} + \varepsilon^2 \Delta\hat{\mathcal{K}}^{(k-1)}_{mi,2} + \cdots,\\  [2mm]
\hat{z}_{m}^{(k)} = \hat{z}_{m,0}^{(k)} + \varepsilon \hat{z}_{m,1}^{(k)} + \varepsilon^2 \hat{z}_{m,2}^{(k)} + \cdots,\\  [2mm]
\hat{Z}_{mi}^{(k)} = \hat{Z}_{mi,0}^{(k)} + \varepsilon \hat{Z}_{mi,1}^{(k)} + \varepsilon^2 \hat{Z}_{mi,2}^{(k)} + \cdots,\\  [2mm]
\Delta\hat{\mathcal{L}}^{(k-1)}_{mi}= \varepsilon^{-1}\Delta\hat{\mathcal{L}}^{(k-1)}_{mi,-1} + \Delta\hat{\mathcal{L}}^{(k-1)}_{mi,0} + \varepsilon \Delta\hat{\mathcal{L}}^{(k-1)}_{mi,1} + \varepsilon^2 \Delta\hat{\mathcal{L}}^{(k-1)}_{mi,2} + \cdots.
\end{array}
\eeq
Inserting \eqref{ExpNum} into \eqref{EvF} we obtain
\beqa 
\label{Index1FG}
\eps^0: \quad
\Delta \hat{\mathcal{K}}^{({k-1})}_{mi,0}
& =& 
f(\hat{Y}^{(k)}_{mi,0} , \hat{Z}^{(k)}_{mi,0})-f(  P^{c_{mi}} (\bar{\hat{y}}^{(k-1)}_0),  P^{c_{mi}} (\bar{\hat{z}}^{(k-1)}_0)) 
+ \mathcal{O}(h^{M}), \\ [2mm]
 \eps^1: \quad 
\Delta \hat{\mathcal{K}}^{(k-1)}_{mi,1}
&=&
\left(f_y (\hat{Y}^{(k)}_{mi,0}, \hat{Z}^{(k)}_{mi,0})
\hat{Y}^{(k)}_{mi,1}  
+ f_z (\hat{Y}^{(k)}_{mi,0}, \hat{Z}^{(k)}_{mi,0})
\hat{Z}^{(k)}_{mi,1} \right) \notag\\ [2mm]
&&-\left(f_y (P^{c_{mi}}(\bar{\hat{y}}^{(k-1)}_0), P^{c_{mi}}(\bar{\hat{z}}^{(k-1)}_0))
P^{c_{mi}}(\bar{\hat{y}}^{(k-1)}_1) \right. \notag \\ [2mm]
\label{Keps1NU}
&& \left.+ f_z (P^{c_{mi}}(\bar{\hat{y}}^{(k-1)}_0), P^{c_{mi}}(\bar{\hat{z}}^{(k-1)}_0))
P^{c_{mi}}(\bar{\hat{z}}^{(k-1)}_1)   \right) +\mathcal{O}(h^{M}). \\ [2mm]
%\QQ{\mbox{suggest to delete the equation below on} \quad \eps^{\nu}} 
&& \cdots \nonumber %\\[2mm] 
%\eps^{\nu}: \quad
%\Delta \hat{\mathcal{K}}^{(k-1)}_{mi,\nu}
%&=&
%\left(f_y ( \hat{Y}^{(k)}_{mi,0}, \hat{Z}^{(k)}_{mi,0})
%\hat{Y}^{(k)}_{mi,\nu}  
%+ f_z ( \hat{Y}^{(k)}_{mi,0}, \hat{Z}^{(k)}_{mi,0})
%\hat{Z}^{(k)}_{mi,\nu} \right) \notag\\ [2mm]
%&&-\left(f_y ( P^{c_{mi}}(\bar{\hat{y}}^{(k-1)}_0), P^{c_{mi}}(\bar{\hat{z}}^{(k-1)}_0))
%P^{c_{mi}}(\bar{\hat{y}}^{(k-1)}_\nu) \right. \notag \\ [2mm]
%&& \left.+ f_z ( P^{c_{mi}}(\bar{\hat{y}}^{(k-1)}_0), P^{c_{mi}}(\bar{\hat{z}}^{(k-1)}_0))
%P^{c_{mi}}(\bar{\hat{z}}^{(k-1)}_\nu)   \right)\notag  \\ [2mm]
%&&+ \psi_{\nu}( \hat{Y}^{(k)}_{mi,0}, \hat{Z}^{(k)}_{mi,0},...,  \hat{Y}^{(k)}_{mi,\nu-1}, \hat{Z}^{(k)}_{mi,\nu-1}) \notag \\ [2mm]
%&&+ \psi_{\nu}( P^{c_{mi}}(\bar{\hat{y}}^{(k-1)}_0), P^{c_{mi}}(\bar{\hat{z}}^{(k-1)}_0),..., P^{c_{mi}}(\bar{\hat{y}}^{(k-1)}_{\nu-1}), P^{c_{mi}}(\bar{\hat{z}}^{(k-1)}_{\nu-1})) \notag\\ [2mm]
%&& + \mathcal{O}(h^{M}).
\label{eq: K_appNU2}
\eeqa
ans so on. 
Similarly, we have 
\beqa 
\label{Index1FG_L}
 \Delta \hat{\mathcal{L}}^{(k-1)}_{mi,-1}
& =& 
g(  \hat{Y}^{(k)}_{mi,0}, \hat{Z}^{(k)}_{mi,0})-g( P^{c_{mi}} (\bar{\hat{y}}^{(k-1)}_0), P^{c_{mi}}(\bar{\hat{z}}^{(k-1)}_0)) 
+ \mathcal{O}(h^{M}), \\ [2mm]
\label{Leps12}
\Delta \hat{\mathcal{L}}^{(k-1)}_{mi,0}
&=&
\left(g_y (\hat{Y}^{(k)}_{mi,0}, \hat{Z}^{(k)}_{mi,0})
\hat{Y}^{(k)}_{mi,1}  
+ g_z (\hat{Y}^{(k)}_{mi,0}, \hat{Z}^{(k)}_{mi,0})
\hat{Z}^{(k)}_{mi,1} \right) \notag\\ [2mm]
&&-\left(g_y (P^{c_{mi}}(\bar{\hat{y}}^{(k-1)}_0), P^{c_{mi}}(\bar{\hat{z}}^{(k-1)}_0))
P^{c_{mi}}(\bar{y}^{(k-1)}_1) \right. \notag \\ [2mm]
&& \left.+ g_z (P^{c_{mi}}(\bar{\hat{y}}^{(k-1)}_0), P^{c_{mi}}(\bar{\hat{z}}^{(k-1)}_0))
P^{c_{mi}}(\bar{z}^{(k-1)}_1)   \right) \notag + \mathcal{O}(h^{M}).\\ [2mm]
&& \cdots  
\eeqa
ans so on.

Because of the linearity of relations (\ref{newK}) and (\ref{newapproach}), we have to order $\varepsilon^{\nu}$ with 
$\nu = -1$ in vectorial form
\beq
\label{NewCond1}
h A \Delta\hat{\mathcal{L}}_{\bar{m},-1}^{(k-1)}  + h S^{\overrightarrow{c}}(\bar{\hat{g}}) = 0, \quad
h b^T \Delta\mathcal{L}_{\bar{m},-1}^{(k-1)}  + h S^m(\bar{\hat{g}}) = 0,
\eeq
and for $\nu \ge 0$,
\beq\label{newK2}
\left(\begin{array}{c}
\hat{y}_{m+1,\nu}^{(k)} -  h S^{m,(k-1)}_{\bar{\hat{\mathbb{F}}}_\nu}\\[2mm]
\hat{z}_{m+1,\nu}^{(k)} - h S^{m,(k-1)}_{\bar{\hat{\mathbb{G}}}_{\nu{+1}}}
\end{array}\right)
 = \left(\begin{array}{c}
\hat{y}_{m,\nu}^{(k)} \\[2mm]
\hat{z}_{m,\nu}^{(k)} 
\end{array}\right) 
 + h \sum_{i=1}^{s} b_i \left(\begin{array}{c}
                          \Delta \hat{\mathcal{K}}^{(k-1)}_{mi,\nu}\\[2mm]
                           \Delta\hat{\mathcal{L}}^{(k-1)}_{mi,\nu}
                          \end{array}\right),
\eeq 
\beq
\label{newK3}
\left(\begin{array}{c}
\hat{Y}_{mi,\nu}^{(k)} - h S^{c_{mi},(k-1)}_{\bar{\hat{\mathbb{F}}}_\nu}\\[2mm]
\hat{Z}_{mi,\nu}^{(k)}  - h S^{c_{mi},(k-1)}_{\bar{\hat{\mathbb{G}}}_{\nu{+1}}} 
\end{array}\right)
 = \left(\begin{array}{c}
\hat{y}_{m,\nu}^{(k)} \\
\hat{z}_{m,\nu}^{(k)}
\end{array}\right)
 + h  \sum_{j=1}^{i} a_{ij}  \left(\begin{array}{c}
                         \Delta \hat{\mathcal{K}}_{mj,\nu}^{(k-1)}\\[2mm]
                         \Delta \hat{\mathcal{L}}_{mj,\nu}^{(k-1)}
                          \end{array}\right),
\eeq
where 
\beq\label{newSnu}
S^{{m},(k-1)}_{\bar{\hat{\mathbb{F}}}_\nu}=S^{m}(\bar{\hat{\mathbb{F}}}^{(k-1)}_\nu), \quad 
S^{{m},(k-1)}_{\bar{\hat{\mathbb{G}}}_{\nu{+1}}}=
S^{m}(\bar{\hat{\mathbb{G}}}^{(k-1)}_{\nu{+1}}).
 \eeq 
Similarly for  $S^{c_{mi},(k-1)}_{\bar{\hat{\mathbb{F}}}_\nu}$ and  $S^{c_{mi},(k-1)}_{\bar{\hat{\mathbb{G}}}_{\nu{+1}}}$.
%\QQ{I think this is OK: see eq. (2.21) (2.22)}

\subsection{Proof of Theorem~\ref{thm: IDC_RK}.}
\label{mainT}

Before proving Theorem~\ref{thm: IDC_RK}, we first give some preliminary results as propositions.
We remark that the crucial assumption in Theorem~\ref{thm: IDC_RK} is that the IRK method is \emph{stiffly accurate}. In the case that this property is not satisfied, the method becomes unstable and the numerical solutions diverge, (see Figure~\ref{fig2.2}). 

In order to justify this, from the invertibility of matrix $A$ and by the first formula in (\ref{NewCond1}) we get
\beq
\label{NewCond3}
\Delta\hat{\mathcal{L}}_{m,-1}^{(k)}  = - A^{-1}S^{\overrightarrow{c}}(\bar{\hat{g}}^{(k)}), 
\eeq
substituting now into the second formula in (\ref{NewCond1}) yields
\beq
\label{eq: sa_condition}
-b^T A^{-1} S^{\bar{c}}(\bar{\hat{g}}^{(k-1)}) + S^m(\bar{\hat{g}}^{(k-1)}) = 0.
\eeq
Then we have the following result as an immediate consequence of the fact that the IRK method is stiffly accurate:
\begin{prop}
\label{prop: SA}
Equation~\eqref{eq: sa_condition} is automatically satisfied, if the IRK methods in the prediction and correction steps of the InDC method are stiffly accurate. 
\end{prop}
\begin{proof}
An IRK method is stiffly accurate if 
\beq
\label{SAp}
b^T A^{-1} = e_s^T,
\eeq
with $e_s = (0, \cdots, 0, 1)^T$.
From \eqref{eq: sa_condition}  we get 
\beq
\label{saproof}
-e_s^TS^{\bar{c}}(\bar{\hat{g}}^{(k-1)}) + S^m(\bar{\hat{g}}^{(k-1)}) = 0.
\eeq
Since the last row of  the spectral integration matrix is $s^{m,k} = \int_{\tau_m}^{\tau_m + c_s h} \alpha_{k}(\tau) d\tau$ by (\ref{SAp}) we get $c_s = 1$ and then $\int_{\tau_m}^{\tau_m + c_s h} \alpha_{k}(\tau) d\tau=\int_{\tau_m}^{t_{m + 1}} \alpha_{k}(\tau) d\tau$. This yields that $e_s^TS^{\bar{c}}(\bar{\hat{g}}^{(k-1)}) = S^m(\bar{\hat{g}}^{(k-1)}) $, and   the equation (\ref{saproof}) is satisfied. 
%$\Box$
\end{proof}

{Furthermore, similar to the Proposition~\ref{prop: idc_be_sa}, we have the following result for InDC-IRK methods. This Proposition follows from Remark \ref{InDC_SA}. In fact, similar reformulation have been performed for the InDC method constructed with explicit RK methods in the prediction and correction steps \cite{christlieb2009comments}. 
\begin{prop}
\label{prop: IDC_RK_B}
The InDC method constructed with stiffly accurate IRK methods can be considered again as a stiffly accurate IRK method with a corresponding Butcher Tableau as in\eqref{eq: B_table} with the matrix $A$ invertible.
\end{prop}
{Now we are in the position to prove Theorem~\ref{thm: IDC_RK} via the following two lemmas.} 
%\begin{lem}
%\label{lem1_for_main}
%($\eps^0$ error term)
%Consider the same assumptions as in Theorem~\ref{thm: IDC_RK} and the limiting case, $\eps = 0$. 
%The numerical solutions {of the InDC method} after $k$ correction loops have the following local error estimates at the interior nodes $\tau_m$,
%$m=0, \cdots M$,
%\begin{eqnarray}
%\label{eq: idc_rk_index1}
%\begin{array}{l}
%e^{(k)}_{m,0} = \mathcal{O}(h^{min(s_k+1, M+1)}), \quad 
%d^{(k)}_{m,0} = \mathcal{O}(h^{min(s_k+1, {M+1})}).
%\end{array}
%\end{eqnarray} 
%\end{lem}
%%
%\SB{ Change with:
\begin{lem}
\label{lem1_for_main}
($\eps^0$ error term)
Consider the reduced system (\ref{reduced}) satisfying (\ref{redg}) with consistent initial values. The numerical solutions {of the InDC method} after $k$ correction loops have the following local error estimates at the interior nodes $\tau_m$,
$m=0, \cdots M$,
\begin{eqnarray}
\label{eq: idc_rk_index1}
\begin{array}{l}
e^{(k)}_{m,0} = \mathcal{O}(h^{min(s_k+1, M+1)}), \quad 
d^{(k)}_{m,0} = \mathcal{O}(h^{min(s_k+1, {M+1})}).
\end{array}
\end{eqnarray} 
\end{lem}
\noindent
{\bf Proof.}
Since the IRK method in the prediction step is stiffly accurate, by definition \eqref{DefSA}, we have $b^TA^{-1} = e^T_s$. 
This implies that the numerical solution is equal to the last stage of the method, i.e.  $\hat{z}_{m+1,0}^{(0)} = \hat{Z}^{(0)}_{ms,0}$ and $\hat{y}_{m+1,0}^{(0)} = \hat{Y}^{(0)}_{ms,0}$. By $g(\hat{Y}^{(0)}_{mi,0},\hat{Z}^{(0)}_{mi,0}) = 0$, we get $\hat{Z}^{(0)}_{mi,0} = \mathcal{G}(\hat{Y}^{(0)}_{mi,0})$ for all $mi$ and, in particular, $\hat{Z}^{(0)}_{ms,0} = \mathcal{G}(\hat{Y}^{(0)}_{ms,0})$.  Then this gives $\hat{z}_{m+1,0}^{(0)} = \mathcal{G}(\hat{y}_{m+1,0}^{(0)})$.

{Now, {by the fact that the IRK method is stiffly accurate and that $\bar{\hat{g}}^{(0)}_0 = (g(\hat{y}^{(0)}_{1, 0}, \hat{z}^{(0)}_{1, 0}), \cdots g(\hat{y}^{(0)}_{M, 0}, \hat{z}^{(0)}_{M, 0})) = \vec{0}$ in the prediction step},  for the {first} correction step, {i.e. $k = 1$},  it follows {from \eqref{newK2}} with $\nu = 0$}  
\beq
\label{yzSol}
%\left\{
\begin{array}{l}
\hat{y}_{m+1,0}^{(1)} = 
\hat{y}_{m,0}^{(1)} +  h S^m(\bar{\hat{f}}^{(0)}_{0})
 + h \sum_{i=1}^{s} b_i \Delta \hat{\mathcal{{K}}}^{(0)}_{mi,0},\\
 g(\hat{y}^{(1)}_{m+1, 0}, \hat{z}^{(1)}_{m+1, 0}) =0,
\end{array}
%\right. 
\eeq
{where, from \eqref{newK3}, we have for} the internal stages
\beq
\label{YZsol}
%\left\{
\begin{array}{l}
\hat{Y}_{mi,0}^{(1)} = 
\hat{y}_{m,0}^{(1)} + hS^{c_{mi}}(\bar{\hat{f}}^{(0)}_0)
 + h  \sum_{j=1}^{i} a_{ij} 
                         \Delta  \hat{\mathcal{{K}}}_{mj,0}^{(0)},\\
g( \hat{Y}^{(1)}_{mi,0}, \hat{Z}^{(1)}_{mi,0})=0.
\end{array}
%\right.
\eeq
Now, from the invertibility of function $g_z$, {by (\ref{yzSol}) and  (\ref{YZsol})} we get $\hat{Z}^{(1)}_{mi,0} = \mathcal{G}(\hat{Y}^{(1)}_{mi,0})$ and
$\hat{z}_{m+1,0}^{(1)} = \mathcal{G}( \hat{y}_{m+1,0}^{(1)})$. Thus the IRK method reads 
\beq
\label{RKeps0}
\begin{array}{c}
\hat{Y}^{(1)}_{mi,0} = \hat{y}^{(1)}_{m,0} + h \sum_{j=1}^s a_{ij} \Delta \mathcal{\hat{K}}^{(0)}_{mj,0}+h S^{{c}_{mi}}(\bar{\hat{f}}_0^{(0)}),\\   
\hat{y}^{(1)}_{m+1,0} = \hat{y}^{(1)}_{m,0} + h \sum_{i=1}^s b_{i} \Delta\mathcal{\hat{K}}^{(0)}_{mi,0}+h S^m(\bar{\hat{f}}_0^{(0)}),
\end{array}
\eeq
where
$
\bar{\hat{f}}^{(0)}_0 = (f(\hat{y}^{(0)}_{0, 0}, \mathcal{G}(\hat{y}^{(0)}_{0, 0})), \cdots f(\hat{y}^{(0)}_{M, 0}, \mathcal{G}(\hat{y}^{(0)}_{M, 0})).
$
The scheme (\ref{RKeps0}) of updating $\hat{y}^{(1)}_{m+1,0}$ can be interpreted as the applying a correction step {of the InDC method} to the ordinary differential equation (\ref{eqy}). Therefore applying similar local truncation error estimates as in \cite{christlieb2009comments, christlieb2009integral} for InDC frameworks using RK methods when applied to a classical ordinary differential equation, we obtain the local error estimate 
\begin{eqnarray}
\label{e-est}
e^{(1)}_{m,0} = \mathcal{O}(h^{min(s_2+1, {M+1})}), 
\end{eqnarray}
for $m=0, \cdots M$, with $s_2 = p^{(0)} + p^{(1)}$.
By $\hat{z}^{(1)}_{m,0} = \mathcal{G}(\hat{y}^{(1)}_{m,0})$, using
the Lipschitz condition of $\mathcal{G}$, we get
\beq
\label{e-est-d}
d^{(1)}_{m,0} = z_{m,0} - \hat{z}^{(1)}_{m,0} = \mathcal{O}(h^{min(s_2+1, {M+1})}).
\eeq
Similarly, at internal stages of the IRK method, by $\hat{Z}^{(1)}_{mi,0} = \mathcal{G}(\hat{Y}^{(1)}_{mi,0})$, we have the following local error estimates,
\beq
\label{e-est2}
E^{(1)}_{mi,0}  = y_0(\tau_m + c_i h) - \hat{Y}^{(1)}_{mi,0} =  \mathcal{O}(h^{min(s_1 + q^{(1)} + 1, M+1)}),
\eeq
and
\beq
\label{e-est2-d}
D^{(1)}_{mi,0} =z_0(\tau_m + c_i h) - \hat{Z}^{(1)}_{mi,0} 
= \mathcal{O}(h^{min(s_1 + q^{(1)} + 1, M+1)}),
\eeq
where $q^{(1)}$ is the stage order for the IRK method applied to the first correction loop.
We note that the proof of the general $k$ is similar. 
$\Box$

\begin{rem}
The local truncation error estimate \eqref{e-est} from \cite{christlieb2009integral} is quite technically involved; it is related to 
estimating the smoothness of rescaled error functions. The estimate \eqref{e-est2} follows a similar fashion. We refer readers to the original paper \cite{christlieb2009integral} for details.
\end{rem}

\begin{rem}
With the estimates in the above Lemma, i.e. equations \eqref{e-est}-\eqref{e-est2-d}, it follows from equation \eqref{Keps1NU} 
\beqa
%\label{eq: K_app}
\Delta \hat{\mathcal{K}}^{(k-1)}_{mi,1}
&=& f_y({y}_{mi,0} , {z}_{mi, 0})\hat{E}^{(k-1)}_{mi,1} 
+ f_z({y}_{mi,0} , {z}_{mi, 0})\hat{D}^{(k-1)}_{mi,1} + \mathcal{O}(h^{s_{k-1}+1}).\notag\\   
\label{eq: K_appNU}
&\doteq& \Delta \mathcal{K}^{(k-1)}_{mi,1}  + \mathcal{O}(h^{s_{k-1}+1})
\eeqa
where $\hat{E}^{(k-1)}_{mi,1}$ and $\hat{D}^{(k-1)}_{mi,1}$ are defined by the corresponding $\eps$-expansion of equation \eqref{Newk}, $s_k = \sum_{r =0}^{k} p^{(r)}$, and
$\Delta \mathcal{K}^{(k-1)}_{mi,1}\doteq f_y({y}_{mi,0} , {z}_{mi, 0})\hat{E}^{(k-1)}_{mi,1} + f_z({y}_{mi,0} , {z}_{mi, 0})\hat{D}^{(k-1)}_{mi,1}$. Here we have used the abbreviations $y_{mi,0}$ and $z_{mi,0}$, i.e. the exact solution $y(t)$ and $z(t)$ at the position $t = \tau_m + c_i h$ respectively.
We note that, {from \eqref{Keps1NU}}, we replaced $\hat{Y}^{(k)}_{mi,0}$ and $P^{c_{mi}}(\bar{\hat{y}}^{(k-1)}_0)$ by adding and subtracting $y_{mi,0}$ with an error of $\mathcal{O}(h^{s_{k-1}+q^{(k)}+1})$ and $\mathcal{O}(h^{s_{k-1}+1})$, the same for $\hat{Z}^{(k)}_{mi,0}$ and $P^{c_{mi}}(\bar{\hat{z}}^{(k-1)}_0)$.\\
Similarly, we have from \eqref{Leps12}
\beqa\label{Leps1}
\Delta \hat{\mathcal{L}}^{(k-1)}_{mi,0}
&=& g_y({y}_{mi,0} , {z}_{mi, 0})\hat{E}^{(k-1)}_{mi,1} +
g_z({y}_{mi,0} , {z}_{mi, 0})\hat{D}^{(k-1)}_{mi,1} + \mathcal{O}(h^{s_{k-1}+1})\notag\\   
\label{eq: L_app}
&\doteq& \Delta \mathcal{L}^{(k-1)}_{mi,0}  + \mathcal{O}(h^{s_{k-1}+1}),
\eeqa
where $\Delta \mathcal{L}^{(k-1)}_{mi,0}\doteq  g_y({y}_{mi,0} , {z}_{mi, 0})\hat{E}^{(k-1)}_{mi,1} +
g_z({y}_{mi,0} , {z}_{mi, 0})\hat{D}^{(k-1)}_{mi,1}$.\\
\end{rem}

\begin{lem}
\label{lem2_for_main}
($\eps^\nu$ error term)
Consider the same assumptions as in Theorem~\ref{thm: IDC_RK} with $0<\eps<<1$. Then the numerical solutions of the InDC method after $k$ correction loops have the following local error estimates at the interior nodes $\tau_m$
with $m=0, \cdots M$
\beq
\label{error_nu}
e^{(k)}_{m,\nu} = y_{m,\nu} - \hat{y}^{(k)}_{m,\nu} = \mathcal{O}(h^{{q}^{(0)}+2-\nu}), \quad
d^{(k)}_{m,\nu} = z_{m,\nu} - \hat{z}^{(k)}_{m,\nu} = \mathcal{O}(h^{{q}^{(0)}+1-\nu}),
\eeq
with $1 \leq \nu \leq q^{(0)} + 1$.
\end{lem}
\begin{proof} We first prove \eqref{error_nu} in the case $\nu = 1$.
In the prediction step $(k = 0)$, under the assumption of stiffly accurate IRK method, by the Corollary 3.10 in \cite{hairer1993solving2}, we get that the error estimates for $\eps^1$ in (\ref{errorExp}) 
at the interior nodes of the InDC {method} with $m=0, \cdots M$ satisfy
\begin{eqnarray}\label{edd}
e^{(0)}_{m,1}= y_{m,1}-\hat{y}^{(0)}_{m, 1} = \mathcal{O}(h^{q^{(0)} + 1}), \quad d^{(0)}_{m,1}= z_{m,1}-\hat{z}^{(0)}_{m,1} = \mathcal{O}(h^{q^{(0)}}).
\end{eqnarray}

We consider $\varepsilon$-expansions of $\hat{y}^{(1)}_{m}$, $\hat{z}^{(1)}_{m}$ and $ \hat{E}^{(1)}_{mi}$ and $\hat{D}^{(1)}_{mi}$ as in \eqref{ExpNum}. Inserting them onto equations \eqref{eq: K_appNU}, \eqref{eq: L_app}, from \eqref{newK2} and \eqref{newK3} for the power $\varepsilon^1$ with {$k = 1$ and $\nu = 1$}, we have
\beq 
\label{RKeps1}
\left(\begin{array}{c}
\hat{y}^{(1)}_{m+1,1} - h S^{{m},(0)}_{\bar{\hat{\mathbb{F}}}_1}\\
\hat{z}^{(1)}_{m+1,0} - h S^{{m},(0)}_{\bar{\hat{\mathbb{G}}}_1} \end{array}\right) 
=  \left(\begin{array}{c}
\hat{y}^{(1)}_{m,1}\\
\hat{z}^{(1)}_{m,0} 
\end{array}\right)
+ h \sum_{i=1}^s b_{i} 
 \left(\begin{array}{c}
\Delta {\mathcal{K}}^{(0)}_{mi,1},\\ 
\Delta {\mathcal{L}}^{(0)}_{mi,0}
\end{array}\right)
+ \mathcal{O}(h^{p^{(0)}+2}),
\eeq
and
\beq
\label{RKeps1bis}
\left(\begin{array}{c}
\hat{Y}^{(1)}_{mi,1} - h S^{c_{m,i},(0)}_{\bar{\hat{\mathbb{F}}}_1}\\
\hat{Z}^{(1)}_{mi,0} -hS^{c_{m,i},(0)}_{\bar{\hat{\mathbb{G}}}_1}
\end{array}\right)
= \left(\begin{array}{c}
\hat{y}^{(1)}_{m,1} \\
\hat{z}^{(1)}_{m,0}
\end{array}\right)
+ h \sum_{j=1}^s a_{ij} 
 \left(\begin{array}{c}
\Delta {\mathcal{K}}^{(0)}_{mj,1}\\
\Delta {\mathcal{L}}^{(0)}_{mj,0}
\end{array}\right)
+ \mathcal{O}(h^{p^{(0)}+2}).
\eeq

Now from (\ref{eq:exact-1})  we have for $\varepsilon^1$, 
%\SB{removing this and replace "we consider the $\varepsilon$-expansion of the exact solution"} 
\beq
\label{exact_eps1}
\begin{array}{c}
{y}_{m+1,1} = {y}_{m,1} +\int_{\tau_m}^{\tau_{m+1}}{\mathbb{F}}_1(t) dt,\quad
{z}_{m+1,0} = {z}_{m,0} +\int_{\tau_m}^{\tau_{m+1}}{\mathbb{G}}_1(t) dt.
\end{array}
\eeq
We subtract \eqref{RKeps1} from \eqref{exact_eps1} and so obtain 
\beq
\label{RKeps2}
\left(\begin{array}{c}
e^{(1)}_{m+1,1} +h S^{{m},(0)}_{\bar{\hat{\mathbb{F}}}_1}- \int_{\tau_m}^{\tau_{m+1}}{\mathbb{F}_1}(t) dt\\
d^{(1)}_{m+1,0} +h S^{{m},(0)}_{\bar{\hat{\mathbb{G}}}_1}- \int_{\tau_m}^{\tau_{m+1}}{\mathbb{G}_1}(t) dt
\end{array}\right)
 = \left(\begin{array}{c}
 e^{(1)}_{m,1} \\
  d^{(1)}_{m,0} 
  \end{array}\right)
- h \sum_{i=1}^s b_{i}
\left(\begin{array}{c}
 \Delta{\mathcal{K}}^{(0)}_{mi,1}  \\
\Delta {\mathcal{L}}^{(0)}_{mi,0}  \end{array}\right)+ \mathcal{O}(h^{p^{(0)}+2}).
\eeq
From  the Corollay 3.10 in \cite{hairer1993solving2} and \eqref{edd}, we have the following estimates for the local errors
\beq
\label{eq: Est_index2}
\begin{array}{l}
e^{(0)}_{m,0} = y_{m,0} -  \hat{y}^{(0)}_{m,0} = \mathcal{O}(h^{p^{(0)} + 1}), \quad
d^{(0)}_{m,0} = z_{m,0} -  \hat{z}^{(0)}_{m,0} = \mathcal{O}(h^{p^{(0)}+1}),\\
e^{(0)}_{m,1} = y_{m,1} -  \hat{y}^{(0)}_{m,1} = \mathcal{O}(h^{q^{(0)} + 1}), \quad
d^{(0)}_{m,1} = z_{m,1} -  \hat{z}^{(0)}_{m,1} = \mathcal{O}(h^{q^{(0)}}).\\
\end{array}
\eeq
Similarly as done in the proof of Lemma \ref{lemma3}, on the right hand-side of (\ref{RKeps2}) 
we add and subtract the quantities $S^m(\bar{\mathbb{F}}_1)$ and $S^m(\bar{\mathbb{G}}_1)$, {these are the integrals of $(M-1)^{th}$ degree interpolating polynomials on $(\tau_m, \mathbb{F}_1(\tau_m))^M_{m=1}$ and
$(\tau_m,\mathbb{G}_1(\tau_m))^M_{m =1}$ over the subinterval $[\tau_m,\tau_{m+1}]$.
Hence, $\int_{\tau_m}^{\tau_{m+1}} \mathbb{F}_1(\tau)d\tau - h S^m(\bar{\mathbb{F}}_1) = \mathcal{O}(h^{M+1})$ and by \eqref{eq: Est_index2}, we have 
$S^m(\bar{\mathbb{F}}_1) - S^{{m},(0)}_{\bar{\hat{\mathbb{F}}}_1} =  \mathcal{O}(h^{q^{(0)}})$ and $ S^m(\bar{\mathbb{G}}_1) - S^{{m},(0)}_{\bar{\hat{\mathbb{G}}}_1} =  \mathcal{O}(h^{q^{(0)}})$.}
Then we {have from (\ref{RKeps2})} 
\beqa
\label{RKeps_22}
\begin{array}{lll}
e^{(1)}_{m+1,1} 
&=&
e^{(1)}_{m,1} - h \sum_{i=1}^s b_{i} \Delta \mathcal{K}^{(0)}_{mi,1} + \mathcal{O}(h^{q^{(0)} + 1}),\\
d^{(1)}_{m+1,0} 
&=&d^{(1)}_{m,0} - h \sum_{i=1}^s b_{i} \Delta \mathcal{L}^{(0)}_{mi,0} + \mathcal{O}(h^{q^{(0)} + 1}). 
\end{array}
\eeqa
Now we consider the $\eps$-expansion of the error at internal stages $\tau_m + c_i h$, and as in equation \eqref{eq: eee} we get
\beq
\label{eq: Err_Errhat}
\begin{array}{l}
E^{(1)}_{mi, 1} = P^{c_{mi}}(\bar{{e}}^{(0)}_{1}) - \hat{E}^{(0)}_{mi, 1}, \quad
D^{(1)}_{mi, 1} = P^{c_{mi}}(\bar{{d}}^{(0)}_{1}) - \hat{D}^{(0)}_{mi, 1},\\
\end{array}
\quad
\forall k\ge0, m.
\eeq
where $\bar{e}^{(0)}_{1}  = (e^{(0)}_{m1,1}, \cdots, e^{(0)}_{ms,1})$, $\bar{d}^{(0)}_{1}  = (d^{(0)}_{m1,0}, \cdots, d^{(0)}_{ms,0})$, $s$ is the number of internal stages in an IRK method.
Especially, by (\ref{eq: Est_index2}), it follows from (\ref{eq: Err_Errhat}), 
\beq
\label{eq: Est_Corr_index2}
\begin{array}{l}
\hat{E}^{(0)}_{mi,1} = - E^{(1)}_{mi, 1} + \mathcal{O}(h^{q^{(0)} + 1}),\quad\\
\hat{D}^{(0)}_{mi,1} = - D^{(1)}_{mi, 1} +\mathcal{O}(h^{q^{(0)}}).
\end{array}
\eeq
Similarly as equations \eqref{RKeps_22}, from the definition of stage order for the prediction step, we have for the internal stages in vectorial form
\beq
\label{RKeps2-internal}
\begin{array}{ccc}
\bar{E}^{(1)}_{1} &=& e^{(1)}_{m,1} \mathbf{1} - h A \Delta \bar{\mathcal{K}}^{(0)}_{1} + \mathcal{O}(h^{q^{(0)} + 1}),\\
\bar{D}^{(1)}_{0} &=& d^{(1)}_{m,0} \mathbf{1}- h A \Delta \bar{\mathcal{L}}^{(0)}_{0} + \mathcal{O}(h^{q^{(0)} + 1}),
\end{array}
\eeq
where $\bar{E}^{(1)}_{1}  = (E^{(1)}_{m1,1}, \cdots, E^{(1)}_{ms,1})$, $\bar{D}^{(1)}_{0}  = (D^{(1)}_{m1,0}, \cdots, D^{(1)}_{ms,0})$ and 
$\mathbf{1} = (1, 1, \cdots, 1)^T$ is a vector of size $s$.
Now from the second equation in \eqref{RKeps2-internal} and using \eqref{eq: idc_rk_index1} and \eqref{eq: Est_Corr_index2}, we get 
\beq
\label{eps1tris}
\begin{array}{c}
A(g_y(y_{mi, 0}, z_{mi, 0}) \hat{E}^{(0)}_{mi,1} + 
g_z(y_{mi, 0}, z_{mi, 0})\hat{E}^{(0)}_{mi,1}) = \mathcal{O}(h^{q^{(0)}}),
\end{array}
\eeq
%where we replace $\hat{E}^{(0)}_{mi,1}$ by $E^{(1)}_{mi,1}$ with $\mathcal{O}(h^{q^{(0)}+2})$ error,
%and replace $\hat{D}^{(0)}_{mi,1}$ by $D^{(1)}_{mi,1}$ with $\mathcal{O}(h^{q^{(0)}})$ error
%due to \eqref{eq: Est_Corr_index2}. 
Thus, from the invertibility of matrix $A$ we have
\beq
\label{eq: D_E}
\hat{D}^{(0)}_{mi,1} = -(g_z^{-1} g_y)(y_{mi, 0}, z_{mi, 0}) \hat{E}^{(0)}_{mi,1} + \mathcal{O}(h^{q^{(0)}}), 
\eeq
for $mi = m1, \cdots, ms$. Plug the above equation \eqref{eq: D_E} into equation \eqref{eq: K_appNU} and replace $\hat{E}^{(0)}_{mi,1}$ by $E^{(1)}_{mi,1}$ with $\mathcal{O}(h^{q^{(0)}+1})$ error and $\hat{D}^{(0)}_{mi,1}$ by $D^{(1)}_{mi,1}$ with $\mathcal{O}(h^{q^{(0)}})$ error,
by \eqref{eq: Est_Corr_index2} we obtain
%first equation of \eqref{RKeps2-internal} gives
\beq
\Delta \mathcal{K}^{(0)}_{mi,1} 
%&&=(f_y -f_z g_z^{-1}g_y)(y_{mi, 0}, z_{mi, 0}) \hat{E}^{(0)}_{mi,1} + \mathcal{O}(h^{q^{(0)}}) \notag \\
= (f_y -f_z g_z^{-1}g_y)(y_{mi, 0}, z_{mi, 0}) E^{(1)}_{mi,1} + \mathcal{O}(h^{q^{(0)}}).
\label{relKL} 
\eeq
Our next aim now is to prove the local error $e^{(1)}_{m,1} =  \mathcal{O}(h^{q^{(0)}+1})$ by mathematical induction w.r.t. $m$. Especially, we would like to show that $e^{(1)}_{m+1,1}= \mathcal{O}(h^{q^{(0)}+1})$, if we assume the local error $e^{(1)}_{l,1} = \mathcal{O}(h^{q^{(0)}+1})$, $\forall l\le m$. To show this, we plug equation \eqref{relKL} into the first equation \eqref{RKeps2-internal} and obtain
$
E^{(1)}_{mi,1} =  \mathcal{O}(h^{q^{(0)}+1}),
$
for $mi = m1, \cdots, ms$. From \eqref{relKL}, $\Delta \mathcal{K}^{(0)}_{mi,1} =  \mathcal{O}(h^{q^{(0)}})$ and plug this estimate into the first equation of \eqref{RKeps_22}, we obtain the desired estimate of
\beq\label{locallocal}
e^{(1)}_{m+1,1}= \mathcal{O}(h^{q^{(0)}+1}).
\eeq
Thus, from \eqref{eq: D_E} and \eqref{eq: Est_Corr_index2},  it follows
\beq
\label{eq: D_error}
D^{(1)}_{mi,1} =  \mathcal{O}(h^{q^{(0)}}).
\eeq 
Now in order to prove the estimate $d^{(1)}_{m,1} =  \mathcal{O}(h^{q^{(0)}})$, we start to considering equation \eqref{eq: z_linear}. Since the IRK method is stiffly accurate, from Remark~\ref{rem: sa}, we have
$
\hat{z}^{(1)}_{m+1, 1} =  \hat{Z}^{(1)}_{ms, 1}.
$
Hence  from \eqref{eq: D_error},
\beq
\label{locallocald}
z_1(\tau_{m+1}) - \hat{z}^{(1)}_{m+1, 1} = d^{(1)}_{m+1} = D^{(1)}_{ms,1}  \stackrel{\eqref{eq: D_error}}{=}
\mathcal{O}(h^{q^{(0)}}), \quad m=0, \cdots M-1.
\eeq

The above proof can be generalized for the InDC method with different IRK methods applied to $k$ correction steps. 
The local error estimates at the interior nodes of the InDC method $\tau_m$
with $m=0, \cdots M$ are
\[
e^{(k)}_{m, 1} =  \mathcal{O}(h^{q^{(0)}+1}), \quad
d^{(k)}_{m, 1} =  \mathcal{O}(h^{q^{(0)}}).
\]
We have thus proved equation \eqref{error_nu} with $\nu = 1$. The general estimates for $\nu > 1$ in equation \eqref{error_nu} can be obtained in a similar fashion to the case of $\nu=1$, as in the Theorem 3.4 in Chap.VI of \cite{hairer1993solving2}. %\QQ{suggest to delete ", then we have  for the local errors
%$
%e^{(k)}_{m, \nu} =  \mathcal{O}(h^{q^{(0)}+2-\nu}), \quad d^{(k)}_{m, \nu} =  \mathcal{O}(h^{q^{(0)}+1-\nu})."
%$ 
%}
\end{proof}

\noindent
{\em Proof of Theorem~\ref{thm: IDC_RK}.} The proof is similar to that for Theorem~\ref{thm: IDC_BE}. 
In fact, we obtain estimates (\ref{final_estimate})  by using the results of Lemmas \ref{lem1_for_main} and \ref{lem2_for_main}. From Lemma \ref{lem1_for_main} we have the local error estimate, for one time step $t_0$ to $t_1$,
\[
e^{(K)}_{M,0} = \mathcal{O}(H^{min(s_K + 1, M+1)}).
\]
From local to global error, we obtain $e^{(K)}_{n,0} = \mathcal{O}(H^{min(s_K, M)})$. From eq.~\eqref{eq: lemma2_g1} and the Lipshitsz condition of $\mathcal{G}$, we get 
\[
d^{(K)}_{n,0} = \mathcal{O}(H^{min(s_K, M)}).
\]
Now in order to complete the proof of the Theorem, we consider the estimates (\ref{error_nu}) in Lemma \ref{lem2_for_main} with $\nu = 1$. 
Then we have for the local error estimates after one step (from $t_0$ to $t_1$)
%\QQ{shall we have the estimate for $e$ only, but not for $d$ below?}

%\SB{The estimate of  (\ref{error_nu}) in Lemma \ref{lem2_for_main} are given for $e$ and $d$, why is there a algebraic relaxation between them??? after one step the estimates for $e$ and $d$ are given in Lemma!! We have an algebraic relation only in the index one case! }
\[
e^{(K)}_{M,1} = \mathcal{O}(H^{q^{(0)}+1}), \quad d^{(K)}_{M,1} = \mathcal{O}(H^{q^{(0)}}).
\]  
Finally, the global estimate (\ref{final_estimate}) from the local estimate above is a consequence of the Theorem 4.5 and 4.6 in Chap. VII of \cite{hairer1993solving2}.

%Thus, we omit it for brevity.

\begin{rem}\label{notY}
We remark that we can not improve the estimate of the global error for the $y$-component as done in Theorem 3.4 in  \cite{hairer1993solving2} for high-indices.
% where in order to do this the authors introduced a new variable. 
Indeed  the reason for such loss of accuracy is related to the evaluation of the integrals in equation (\ref{newSnu}). These integrals are obtained from the prediction step, and {the algebraic variable $z$ obtained in the prediction step is the cause }that reduces the order of the differential variable $y$ {in the correction steps}. %\QQ{suggest to delete "A definition of a new variable as done in Theorem 3.4 in \cite{hairer1993solving2} can not produce any benefit to the order of accuracy for the $y$-component."} 
This can be seen in the evaluation from equation (\ref{RKeps2}) to (\ref{RKeps_22}) due to (\ref{eq: Est_index2}). We note that a similar conclusion for the remainder can be drawn. 
%\QQ{suggest to delete "The proof of this fact is rather technical and therefore for brevity omitted."}
\end{rem}

%\textcolor{red}{This remark suggests us the following proposition:
%\SB{I prefer to remove this paragraph....}
%(\QQ{Proposition~\ref{prop: IDC_RK_B}} could \QQ{suggest to delete "have the advantage to"} suggest us that the classical results for IRK methods in \cite{hairer1993solving2} can be directly applied to InDC-IRK  method. Then Theorem  \ref{thm: IDC_RK} is an extension of the classical one, i.e. Theorem \ref{Thm:PrincTh} for stiffly accurate IRK method. \SB{Similarly a completely equal proof of the estimates of the theorem can be given but going through such proof considering the remark \ref{notY}, we can conclude no improving for the $y$-component.} \QQ{this is not clear to me.}
%.......)\\

%\begin{cor}
%\label{cor:IDC(SA)}
%Under the assumptions of theorem (\ref{Thm:PrincTh}) the global error of a InDC(SA-IRK)-method satisfies
%beq\label{eq:yz2}
%	 y_n - y(t_n)= \mathcal{O}(h^{p}) + \mathcal{O}(\varepsilon h^{q}),\ \ \ \ \ \	  z_n - z(t_n) = \mathcal{O}(h^{p}) + \mathcal{O}(\varepsilon h^{q}). 
%\eeq 
%\end{cor}
%\begin{rem}
%From this corollary we show that a completely equal proof of the estimates of the theorem (\ref{thm: IDC_RK}), and similarly for those of theorem (\ref{cor:IDC(SA)}), %(\ref{eq:yz2}), 
%for the InDC(SA-IRK)-method can be given following the original proof of Theorem in \cite{hairer1993solving2} under the remark \ref{notY}. 
%\end{rem}

%%%%%%
\subsection{Estimation of the Remainder} \label{SectRem}
{
Finally, in order to estimate the remainder  for the global error functions $e^{(K)}_n$ and $d^{(K)}_n$, 
we have the following result. %gives a rigorous estimate of the remainder.
\begin{thm}
\label{prop: remainder}
Under the same hypothesis as in Theorem \ref{thm: IDC_RK} for any fixed constant $C>0$ and $\nu \le q^{(0)}+1$, the global error satisfies for $\eps \le CH$ 
%and $\nu \le q^{(0)} + 1$ %q_0
\beq
\label{remainder}
%\begin{array}{l}
e^{(K)}_n = %\sum_{\nu = 0}^{\nu} e^{(K)}_{n, \nu} 
e^{(K)}_{n,0} + \eps e^{(K)}_{n,1} + \cdots + \eps^{\nu} e^{(K)}_{n,\nu} 
+\mathcal{O}(\eps^{\nu+1}/H), \quad
d^{(K)}_n = % \sum_{\nu = 0}^{\nu} d^{(K)}_{n,\nu} 
d^{(K)}_{0,n} + \eps d^{(K)}_{n,1} + \cdots + \eps^{\nu} d^{(K)}_{n,\nu} 
+\mathcal{O}(\eps^{\nu+1}/H).
%\end{array}
\eeq
These estimates %\eqref{remainder} 
hold uniformly for $H\leq H_0$ and $nH \le Const$.   
\end{thm}
\noindent
{\bf Proof.}
By the estimates (\ref{error_nu}) it is sufficient to prove the result for $\nu = q^{(0)} + 1$. 

Through the  $\eps$-asymptotic expansion  (\ref{eq: eps_error})   for the global error functions $e^{(k)}_n$ and $d^{(k)}_n$, by considering estimates (\ref{error_nu}) globally, i.e. 

\beq
\label{error_nu_glob}
e^{(k)}_{n,\nu} = y_{n,\nu} - \hat{y}^{(k)}_{n,\nu} = \mathcal{O}(H^{{q}^{(0)}+1-\nu}), \quad
d^{(k)}_{n,\nu} = z_{n,\nu} - \hat{z}^{(k)}_{n,\nu} = \mathcal{O}(H^{{q}^{(0)}+1-\nu}),
\eeq

 and $\nu = q^{(0)} + 1$, we get: 
\beq\label{ValR}
{e}^{(K)}_n = 
e^{(K)}_{n,0} + \eps e^{(K)}_{n,1} + \cdots + \eps^{\nu} e^{(K)}_{n,\nu} 
+\mathcal{O}(\eps^{\nu+1}/H), 
\quad
{d}^{(K)}_n =  
d^{(K)}_{0,n} + \eps d^{(K)}_{n,1} + \cdots + \eps^{\nu} d^{(K)}_{n,\nu} 
+\mathcal{O}(\eps^{\nu+1}/H).
\eeq
with $e^{(K)}_{n,0} = y_{n,0}^{(K)} - \hat{y}_{n,0}^{(K)} $ and $d^{(K)}_{n,0}=z_{n,0}^{(K)} - \hat{z}_{n,0}^{(K)} $, $\cdots$ (see formula (\ref{eq: eps_error})).

In order to estimate the remainder, we consider the truncated series of the quantities in (\ref{ExpNum}):
\begin{eqnarray}\label{yYZ}
\hat{\underline{{y}}}_n^{(K)} = \hat{{y}}_{n,0}^{(K)} + \varepsilon \hat{{y}}_{n,1}^{(K)} + \cdots + \varepsilon^{\nu}\hat{y}_{n,\nu}^{(K)}, \quad
% \begin{eqnarray}\label{zZ}
 \hat{\underline{{z}}}_n^{(K)} = \hat{{z}}_{n,0}^{(K)} + \varepsilon \hat{{z}}_{n,1}^{(K)} + \cdots + \varepsilon^{\nu}\hat{{z}}_{n,\nu}^{(K)},
 \end{eqnarray}  
\begin{eqnarray}\label{Yy}
\hat{\underline{{Y}}}_{ni}^{(K)} = \hat{{Y}}_{ni,0}^{(K)} + \varepsilon \hat{{Y}}_{ni,1}^{(k)} + \cdots + \varepsilon^{\nu}\hat{{Y}}_{ni,\nu}^{(k)}, \quad
%\end{eqnarray}
%
% \begin{eqnarray}\label{zZ}
 \hat{\underline{{Z}}}_{ni}^{(K)} = \hat{{Z}}_{{ni},0}^{(K)} + \varepsilon \hat{{Z}}_{{ni},1}^{(k)} + \cdots + \varepsilon^{\nu}\hat{{Z}}_{{ni},\nu}^{(K)} ,
  \end{eqnarray} 
and
\begin{eqnarray*}
%\underline
\Delta \hat{{Y}}^{(K)}_{ni} \doteq  \hat{Y}_{ni}^{(K)}-\hat{\underline{Y}}^{(k)}_{ni}, \quad  
%\underline
\Delta\hat{{Z}}^{(K)}_{ni} \doteq  \hat{Z}_{ni}^{(K)}-\hat{\underline{Z}}^{(K)}_{ni}.
\end{eqnarray*}
\begin{eqnarray}\label{Kk}
\begin{array}{c}
\displaystyle \Delta \hat{\underline{\mathcal{K}}}^{(K-1)}_{ni} = \Delta \hat{\mathcal{K}}^{(K-1)}_{ni,0}+ \varepsilon \Delta \hat{\mathcal{K}}^{(K-1)}_{ni,1} + \cdots + \varepsilon^{\nu}\Delta \hat{\mathcal{K}}^{(K-1)}_{ni,\nu},\\
\displaystyle \Delta \hat{\underline{\mathcal{L}}}^{(K-1)}_{ni} = \Delta \hat{\mathcal{L}}^{(K-1)}_{ni,0}+ \varepsilon \Delta \hat{\mathcal{L}}^{(K-1)}_{ni,1} + \cdots + \varepsilon^{\nu}\Delta \hat{\mathcal{L}}^{(K-1)}_{ni,\nu},
\end{array}
\end{eqnarray}
and we use the notation for the remainder
 \begin{eqnarray}\label{Sol_e}
%\underline
\Delta \hat{{y}}^{(K)}_{n} \doteq  \hat{y}_{n}^{(K)}-\hat{\underline{y}}^{(K)}_{n}, \quad  
%\underline
\Delta \hat{{z}}^{(K)}_{n} \doteq \hat{z}_{n}^{(K)}-\hat{\underline{z}}^{(K)}_{n}.
%\underline
%\Delta \hat{{Y}}^{(K)}_{ni} \doteq  \hat{Y}_{ni}^{(K)}-\hat{\underline{Y}}^{(k)}_{ni}, \quad  
%\underline
%\Delta\hat{{Z}}^{(K)}_{ni} \doteq  \hat{Z}_{ni}^{(K)}-\hat{\underline{Z}}^{(K)}_{ni}.
\end{eqnarray}

Then using (\ref{eq: eps_error}) and (\ref{error_nu_glob}), (\ref{ValR}) is then equivalent to 
\begin{eqnarray}\label{EstD}
\Delta \hat{{y}}^{(K)}_{n} = \mathcal{O}(\eps^{\nu+1}/H), \quad \Delta \hat{{z}}^{(K)}_{n} = \mathcal{O}(\eps^{\nu+1}/H).
\end{eqnarray}
 The proof of (\ref{EstD}) is similar to the proof of Theorem 3.8, Chap.~VI in \cite{hairer1993solving2}.  For the sake of brevity, here we point out the main differences and we give some partial results.

We consider the InDC method (\ref{newK})-(\ref{newapproach})-(\ref{newS}).
Inserting (\ref{yYZ})  into (\ref{eq: spect_inte_accu}), it gives
\begin{eqnarray}\label{Sc}
S^{c_{mi},(k-1)}_{\bar{\hat{f}}} - S^{c_{mi},(k-1)}_{\bar{\underline{f}}} = \mathcal{O}(\varepsilon^{\nu+1}),\quad
S^{c_{mi},(k-1)}_{\hat{\hat{g}}} - S^{c_{mi},(k-1)}_{\underline{\hat{g}}} = \mathcal{O}(\varepsilon^{\nu+1}).
\end{eqnarray}
Similarly inserting the quantities (\ref{yYZ}) (\ref{Yy}) and (\ref{Kk}) with $m$ (index of quadrature nodes) instead of $n$ into (\ref{EvF2}) and by (\ref{Index1FG}), %(\ref{eq: K_appNU2}), (\ref{Leps12})
 (\ref{Index1FG_L}) and (\ref{Sc}), we obtain by Lemma (\ref{lem2_for_main}) and  $\nu \le q^{(0)} + 1$,
\beq\label{estimate!1}
\left(\begin{array}{c}
\underline{\hat{Y}}_{mi}^{(k)}-\underline{\hat{y}}_{m}^{(k)}\\[2mm]
\eps (\underline{\hat{Z}}_{mi}^{(k)}- \underline{\hat{z}}_{m}^{(k)} ) 
\end{array}\right)
 = 
h  \sum_{j=1}^{s} a_{ij}  \left(\begin{array}{c}
\Delta \underline{\mathcal{\hat{K}}}^{(k-1)}_{mi}\\ 
\Delta \underline{\mathcal{\hat{L}}}^{(k-1)}_{mi} 
                          \end{array}\right)
                          + \left(\begin{array}{c}
 \mathcal{O}(h\varepsilon^{\nu+1})\\[2mm]
  \mathcal{O}(\varepsilon^{\nu + 1})
   \end{array}\right).
\eeq
This represents the defect when  (\ref{yYZ}) (\ref{Yy}) and (\ref{Kk}) are inserted into the InDC R-K method (\ref{newK})-(\ref{newapproach}).

From now on the the proof is similar to Theorem 3.8, in \cite{hairer1993solving2}. 
In fact, applying Theorem 3.6 in  \cite{hairer1993solving2} to (\ref{estimate!1}), 
 it yields
 \beq
 \begin{array}{c}\label{EiDi}
 \displaystyle||\Delta{\hat{Y}}^{(k)}_{mi}||\le C(|| \Delta {\hat{y}}^{(k)}_{m} || + \eps ||\Delta{\hat{z}}^{(k)}_{m} ||) + \mathcal{O}(\eps^{\nu+1}),\\[2mm]
\displaystyle ||\Delta{\hat{Z}}^{(k)}_{mi}||\le C(|| \Delta {\hat{y}}^{(k)}_{m} || + \eps/h ||\Delta{\hat{z}}^{(k)}_{m} ||) + \mathcal{O}(\eps^{\nu+1}/h).
  \end{array}
 \eeq
 where here the quantities $\delta_i$ and $\theta_i$ in Theorem 3.6 in  \cite{hairer1993solving2} are given by: $\delta_i = \mathcal{O}(\eps^{\nu+1})$, $\theta_i = \mathcal{O}(\eps^{\nu+1}/h)$.
 
 In a similar fashion as the point b) of the proof in Theorem 3.8 in \cite{hairer1993solving2}, we obtain for the quantities $\Delta \hat{y}^{(k)}_{m}$ and $\Delta \hat{z}^{(k)}_{m}$ in (\ref{Sol_e})  the recursion 
\beq\label{estimate!2}
\left(\begin{array}{c}
|| \Delta \hat{y}_{m+1}^{(k)}|| \\[2mm]
|| \Delta \hat{z}_{m+1}^{(k)}|| 
\end{array}\right)
 = 
   \left(\begin{array}{cc}
	1 + \mathcal{O}(h) & \mathcal{O}(\eps)\\
        \mathcal{O}(1) & \alpha + \mathcal{O}(\eps) 
                          \end{array}\right)
                          \left(\begin{array}{c}
                          || \Delta \hat{y}_{m}^{(k)}|| \\[2mm]
|| \Delta \hat{z}_{m}^{(k)}|| 
\end{array}\right)
 + \left(\begin{array}{c}
 \mathcal{O}(\varepsilon^{\nu+1})\\[2mm]
  %\varepsilon^{\nu+1} \Delta \hat{\mathcal{L}}^{(k-1)}_{mi,\nu} 
  \mathcal{O}(\varepsilon^{\nu + 1}/h)
   \end{array}\right),
\eeq
with $\alpha < 1$. The value of $\alpha$ is specified in the Theorem 3.8 in \cite{hairer1993solving2}. 
 
Finally applying Lemma 3.9  in \cite{hairer1993solving2} to the difference inequalities in (\ref{estimate!2}) gives 
%$$
%\Delta e_{n+1}^{(k)} = \mathcal{O}(\eps^{\nu+1}/h), \quad \Delta d_{n+1}^{(k)} = \mathcal{O}(\eps^{\nu+1}/h)
%$$ 
% for the the points b) and c) of , where we get
 \begin{eqnarray}\label{induct}
\Delta \hat{y}^{(k)}_m =  \mathcal{O}(\eps^{\nu+1}/h), \quad  \Delta \hat{z}^{(k)}_m  =  \mathcal{O}(\eps^{\nu+1}/h).
\end{eqnarray}
for $mh \le Const$. % i.e. (\ref{EstD}). 
%Then we get (\ref{induct}), i.e. 
% \begin{eqnarray}\label{ee} 
%e^{(k)}_m = 
%e^{(k)}_{m,0} + \eps e^{(k)}_{m,1} + \cdots + \eps^{\nu} e^{(k)}_{m,\nu} 
%+\mathcal{O}(\eps^{\nu+1}/h), \quad
%d^{(k)}_m = 
%d^{(k)}_{0,m} + \eps d^{(k)}_{m,1} + \cdots + \eps^{\nu} d^{(k)}_{m,\nu} 
%+\mathcal{O}(\eps^{\nu+1}/h).
%\end{eqnarray}
%for $mh \le Const$. 
By $H = Mh$, then we get (\ref{EstD}), i.e., the statement of the Theorem.  
}

{
\begin{rem}
We note that from  (\ref{remainder}) for Theorem \ref{thm: IDC_BE} and $\nu \le 2$ we get:
\beq
\begin{array}{lll}
e^{(K)}_n  &=& \mathcal{O}(H^{\min\{K+1, M\}}) + \mathcal{O}(\eps H)+ \mathcal{O}(\eps^2) + \mathcal{O}(\eps^3/H),\\
d^{(K)}_n &=& \mathcal{O}(H^{\min\{K+1, M\}}) + \mathcal{O}(\eps H) + \mathcal{O}(\eps^2) + \mathcal{O}(\eps^3/H),
\end{array}
\eeq
% $\mathcal{O}(\eps^2)$ in (\ref{eq: final_estimate1}), and $\mathcal{O}(\eps^{q^{(0)}+1})$ in (\ref{final_estimate}) where $v = 0,1$
%\noindent
%{\em Proof.}
%\begin{rem}
 and for Theorem \ref{thm: IDC_RK} with $\nu = q^{(0)}+1$: 
\beq
\begin{array}{lll}
e^{(K)}_n  &=& \mathcal{O}(H^{\min\{s_{K}, M\}}) +\mathcal{O}(\eps H^{q^{(0)}})+\cdots +\mathcal{O}(\eps^{q^{(0)}+1})  + \mathcal{O}(\eps^{q^{(0)}+2}/H), \\
d^{(K)}_n  &=& \mathcal{O}(H^{\min\{s_{K}, M\}}) +\mathcal{O}(\eps H^{q^{(0)}})+\cdots+\mathcal{O}(\eps^{q^{(0)}+1})  + \mathcal{O}(\eps^{q^{(0)}+2}/H). 
\end{array}
\eeq
\end{rem}
}

\bibliography{refer}

\end{document}